\documentclass[hyperref]{article}
\usepackage[toc]{appendix}
\usepackage{float}
\usepackage{amssymb}
\usepackage{amsmath}
\usepackage{mathtools}
\usepackage{latexsym}
\usepackage{mathrsfs}
\usepackage[thmmarks,amsmath]{ntheorem}
\usepackage{extarrows}
\usepackage[hang,flushmargin]{footmisc}
\usepackage[noend]{algpseudocode}
\usepackage{fancyhdr}
\usepackage{pdfpages}
\usepackage{mathtools}
\usepackage{tikz}
\usepackage{tikz-cd}
\usepackage[all]{xy}
\usepackage{dsfont}
\usepackage{tipa}
\usepackage{changepage}
\usepackage{pifont}
\usepackage{imakeidx}
\usepackage{multicol}
\usepackage{hyperref}
\hypersetup{linktocpage}
\usepackage{titletoc}
\usepackage{titlesec}
\usepackage{docmute}
\usepackage[font=footnotesize,skip=0pt,textfont=rm,labelfont=rm]{caption,subcaption}
\usepackage{booktabs}
\usepackage{tocloft}

{
	\theoremstyle{nonumberplain}
	\theoremheaderfont{\bfseries}
	\theorembodyfont{\normalfont}
	\theoremsymbol{\mbox{$\Box$}}
	\newtheorem{proof}{Proof}
}

\usepackage{theorem}
\newtheorem{thm}{Theorem}
\newtheorem{cor}[thm]{Corollary}
\newtheorem{prop}[thm]{Proposition}

\newtheorem{lem}[thm]{Lemma}
{
	\theoremheaderfont{\bfseries}
	\theorembodyfont{\normalfont}
	\newtheorem{defn}[thm]{Definition}

	\newtheorem{remark}[thm]{Remark}
}
\numberwithin{thm}{section}
\numberwithin{equation}{section}

\usepackage{enumitem}

\graphicspath{{figure/}}                                 
\usepackage[a4paper]{geometry}                           
\usepackage{lmodern, mathtools}
\usepackage{old-arrows} 

\makeatletter
\newcommand{\xdbheadrightarrow}[2][]{%
	\ext@arrow 0099\xdbheadfill@{#1}{#2}}%
\newcommand{\xdbheadfill@}{%
	\arrowfill@\relbar\relbar{\mathrel{\vphantom{\rightarrow}\smash{\twoheadrightarrow}}}}
\makeatother

\DeclareRobustCommand\longtwoheadrightarrow
{\relbar\joinrel\twoheadrightarrow}

\newcommand{\id}{\mathrm{Id}}

\newcommand{\qp}{\mathbb{Q}_p}
\newcommand{\gal}{\mathrm{Gal}}

\newcommand{\modet}{\mathrm{Mod}^{\mathrm{\acute{e}t,proj}}}
\newcommand{\modetfp}{\mathrm{Mod}^{\mathrm{\acute{e}t,fp}}}
\newcommand{\repfr}{\mathrm{Rep}^{\mathrm{free}}}
\newcommand{\rep}{\mathrm{Rep}}
\newcommand{\zokgen}{Z_{\mathcal{O}_K}^{\mathrm{gen}}}

\usepackage[backend=bibtex, style=alphabetic, maxalphanames=5, maxnames=99, useprefix=true]{biblatex}
\title{Multivariable $(\varphi_q,\mathcal{O}_K^{\times})$-modules associated to $p$-adic representations of $\gal(\overline{K}/K)$}
\author{Changjiang DU}

\begin{document}	
	\maketitle
	
	\begin{abstract}
		Let $K$ be an unramified extension of $\qp$, and $E$ a finite extension of $K$ with ring of integers $\mathcal{O}_E$.
		We associate to every finite type continuous $\mathcal{O}_E$-representation $\rho$ of $\gal(\overline{K}/K)$ an étale $(\varphi_q,\mathcal{O}_K^{\times})$-module $D_{A_{\mathrm{mv},E}}^{(0)}(\rho)$ over $A_{\mathrm{mv},E}$, where $A_{\mathrm{mv},E}$ is the $p$-adic completion of a completed localization of the Iwasawa algebra $\mathcal{O}_E[\negthinspace[\mathcal{O}_K]\negthinspace]$. Furthermore, we prove that the functor $D_{A_{\mathrm{mv},E}}^{(0)}$ is fully faithful and exact. This functor is a $p$-adic analogue of $D_A^{(0)}$ in the recent work of Breuil, Herzig, Hu, Morra and Schraen.
	\end{abstract}
	
	\tableofcontents
	
	\section{Introduction}
	
	Let $p$ be a prime number and let $K$ be a finite unramified extension of $\qp$. Fontaine constructed the category of cyclotomic $(\varphi,\Gamma)$-modules (\cite{Fontaine1990}) and established an equivalence of categories between this category and the category of $p$-adic representations of $\gal(\overline{K}/K)$. The category of cyclotomic $(\varphi,\Gamma)$-modules has many remarkable properties and a broad range of applications. For instance, it plays an important role in Colmez's construction of the local Langlands correspondence for $\mathrm{GL}_2(\qp)$ (\cite{Colmez2010}). 
	
	Let $\mathbb{F}_q$ be the residue field of $K$ and let $\mathbb{F}$ be a finite extension of $\mathbb{F}_q$, where $q=p^f$. In a recent paper \cite{breuil2023multivariable},  Breuil-Herzig-Hu-Morra-Schraen constructed a fully faithful exact functor $D_A^{(i)}$ (for $0\leq i\leq f-1$) from the category of finite dimensional continuous $\mathbb{F}$-representations of $\gal(\overline{K}/K)$ to the category of multivariable étale $(\varphi_q,\mathcal{O}_K^{\times})$-modules over $A$, where $\mathcal{O}_K$ is the ring of integers of $K$ and $A$ is a completed localization of the Iwasawa algebra $\mathbb{F}[\negthinspace[\mathcal{O}_K]\negthinspace]$. Furthermore, the functor $D_A^{\otimes}\coloneq \otimes_{i=0}^{f-1}D_A^{(i)}$ is related to $\mathbb{F}$-representations of $\mathrm{GL}_2(K)$ coming from the cohomology of Shimura curves (see also \cite{wang2024lubintatemultivariablevarphimathcaloktimesmodulesdimension} for further developments).
	
	Let us briefly recall the construction of $D_A^{(i)}$ in \cite{breuil2023multivariable}. First we define the ring $A$. Let $E$ be a finite extension of $K$ with ring of integers $\mathcal{O}_E$, uniformizer $\varpi$ and residue field $\mathbb{F}$. Let $\sigma_{0},\dots,\sigma_{f-1}$ be all the embeddings $K\hookrightarrow E$, we can write the Iwasawa algebra $\mathbb{F}[\negthinspace[\mathcal{O}_K]\negthinspace]$ as $\mathbb{F}[\negthinspace[T_{\sigma_{0}},\dots,T_{\sigma_{f-1}}]\negthinspace]$ for $T_{\sigma_{i}}\coloneq \sum_{\lambda\in\mathbb{F}_q^{\times}}\sigma_i(\lambda)^{-1}[\lambda]$, where $[\lambda]\in\mathcal{O}_K^{\times}$ is the Teichmüller lift of $\lambda$. The ring $A$ is the $(T_{\sigma_{0}},\dots,T_{\sigma_{f-1}})$-adic completion of $\mathbb{F}[\negthinspace[\mathcal{O}_K]\negthinspace]\left[T_{\sigma_{0}}^{-1},\dots,T_{\sigma_{f-1}}^{-1}\right]$. 
	The multiplication by $p$ and $\mathcal{O}_K^{\times}$ on $\mathcal{O}_K$ induces an endomorphism $\varphi$ and an $\mathcal{O}_K^{\times}$-action on $A$, and we let $\varphi_q\coloneq \varphi^f$. This gives a continuous $(\varphi_q,\mathcal{O}_K^{\times})$-action on $A$ (for the definition of a continuous $(\varphi,\Gamma)$-action, see Definition \ref{continuous (phi,Gamma)-action}).

	We fix an embedding $\sigma_{0}:\mathbb{F}_q\hookrightarrow \mathbb{F}$, and we let $\mathbb{F}(\negthinspace(T_{\mathrm{LT}})\negthinspace)$ be the fraction field of $\mathbb{F}[\negthinspace[T_{\mathrm{LT}}]\negthinspace]$. The Lubin-Tate actions of $p$ and $\mathcal{O}_K^{\times}$ give an $\mathbb{F}$-linear $(\varphi_q,\mathcal{O}_K^{\times})$-action on $\mathbb{F}(\negthinspace(T_{\mathrm{LT}})\negthinspace)$, see (\ref{phi_q,O_k^* action ass. to Lubin-Tate formal groups}). Let $\overline{\rho}$ be a continuous finite dimensional $\mathbb{F}$-representation of $\gal(\overline{K}/K)$, by the theory of Lubin-Tate $(\varphi,\Gamma)$-modules, we can associate to $\overline{\rho}$ an étale $(\varphi_q,\mathcal{O}_K^{\times})$-module $D_{\mathrm{LT},\sigma_{0}}(\overline{\rho})$ over $\mathbb{F}(\negthinspace(T_{\mathrm{LT}})\negthinspace)$ (see for example \cite{kisin2009galois} and \cite{schneider2017galois}).
	
	In order to define $D_A^{(i)}$, we need to go from the Lubin-Tate variable $T_{\mathrm{LT}}$ to the variables $T_{\sigma_{i}}\in A$. For this, we need the perfectoid $\mathbb{F}$-algebra $A_{\infty}'\coloneq\mathbb{F}\left(\negthinspace\left(T_{\mathrm{LT},0}^{1/p^{\infty}}\right)\negthinspace\right)\left\langle\left(\frac{T_{\mathrm{LT},i}}{T_{\mathrm{LT},0}^{p^{i}}}\right)^{\pm 1/p^{\infty}}: 1\leq i\leq f-1\right\rangle$ as an intermediate ring. The ring $A_{\infty}'$ is endowed with an $\mathbb{F}$-linear continuous $(\varphi_q,(\mathcal{O}_K^{\times})^f)$-action (\ref{phi_q,O_K*-action on A_{infty}'}), along with a map $\mathrm{pr}_i: \mathbb{F}(\negthinspace(T_{\mathrm{LT}})\negthinspace)\to A_{\infty}', T_{\mathrm{LT}}\to T_{\mathrm{LT},i}$ which commutes with the $(\varphi_q,(\mathcal{O}_K^{\times})^f)$-actions. Here $(\mathcal{O}_K^{\times})^f$ acts on $\mathbb{F}(\negthinspace(T_{\mathrm{LT}})\negthinspace)$ via the $i$-th projection $(\mathcal{O}_K^{\times})^f\to\mathcal{O}_K^{\times}$. 
	Thus $A_{\infty}'\otimes_{\mathrm{pr}_i,\mathbb{F}(\negthinspace(T_{\mathrm{LT}})\negthinspace)}D_{\mathrm{LT},\sigma_{0}}(\overline{\rho})$ is an étale $(\varphi_q,(\mathcal{O}_K^{\times})^{f})$-module over $A_{\infty}'$.
	
	Now, let $A_{\infty}$ be the completed perfection of $A$, which is a perfectoid $\mathbb{F}$-algebra. There is a ring map $m: A_{\infty}\to A_{\infty}'$ coming from the theory of \cite{fargues2020simple}, see \cite[$\S$2.4]{breuil2023multivariable}, and the map $m$ commutes with the continuous $(\varphi_q,(\mathcal{O}_K^{\times})^{f})$-actions, where $(\mathcal{O}_K^{\times})^{f}$ acts on $A_{\infty}$ via the product map $\Pi: (\mathcal{O}_K^{\times})^{f}\to \mathcal{O}_K^{\times},(a_0,\dots,a_{f-1})\mapsto \prod_{i=0}^{f-1}a_i$. In fact, the corresponding morphism $m:\mathrm{Spa}(A_{\infty}',(A_{\infty}')^{\circ})\to \mathrm{Spa}(A_{\infty},A_{\infty}^{\circ})$ of affinoid perfectoid spaces is a pro-étale $\Delta_1$-torsor, where $\Delta_1\coloneq\ker \Pi$, $(A_{\infty}')^{\circ}$ and $A_{\infty}^{\circ}$ are subrings of $A_{\infty}'$ and $A_{\infty}$ respectively, consisting of power-bounded elements, see \cite[Proposition 2.4.4]{breuil2023multivariable}. In particular, $\left(A_{\infty}'\right)^{\Delta_1}=A_{\infty}$. As a consequence,  $\left(A_{\infty}'\otimes_{\mathrm{pr}_i,\mathbb{F}(\negthinspace(T_{\mathrm{LT}})\negthinspace)}D_{\mathrm{LT},\sigma_{0}}(\overline{\rho})\right)^{\Delta_1}$ is an étale $(\varphi_q,\mathcal{O}_K^{\times})$-module over $A_{\infty}$ (\cite[Theorem 2.5.1]{breuil2023multivariable}). 
	
	Finally, by \cite[Theorem 2.6.4]{breuil2023multivariable}, $\left(A_{\infty}'\otimes_{\mathrm{pr}_i,\mathbb{F}(\negthinspace(T_{\mathrm{LT}})\negthinspace)}D_{\mathrm{LT},\sigma_{0}}(\overline{\rho})\right)^{\Delta_1}$ canonically descends to an étale $(\varphi_q,\mathcal{O}_K^{\times})$-module $D_A^{(i)}(\overline{\rho})$ over $A$. Besides, by investigating the relation between $D_A^{(i)}$ and Fontaine's cyclotomic $(\varphi,\Gamma)$-modules, the fully faithfulness of $D_A^{(i)}$ follows.

	In this paper, for $0\leq i\leq f-1$, we construct a fully faithful exact functor $D_{A_{\mathrm{mv},E}}^{(i)}$ which is the $p$-adic analogue of $D_A^{(i)}$. We write $\mathcal{O}_E[\negthinspace[\mathcal{O}_K]\negthinspace]$ as $\mathcal{O}_E[\negthinspace[T_{\sigma_{0}},\dots,T_{\sigma_{f-1}}]\negthinspace]$ for $T_{\sigma_{i}}\coloneq \sum_{\lambda\in\mathbb{F}_q^{\times}}\sigma_i([\lambda])^{-1}[\lambda]$ where $[\lambda]\in\mathcal{O}_K^{\times}$ is the Teichmüller lift of $\lambda$, then we define $A_{\mathrm{mv},E}$ as the $p$-adic completion of the $(T_{\sigma_{0}},\dots,T_{\sigma_{f-1}})$-adic completion of $\mathcal{O}_E[\negthinspace[\mathcal{O}_K]\negthinspace]\left[\frac{1}{T_{\sigma_{i}}}:0\leq i\leq f-1\right]$ (``mv''  stands for ``multivariable''), see (\ref{defn of A_{mv,K}}) for the precise definition. The multiplication by $p$ and $\mathcal{O}_K^{\times}$ on $\mathcal{O}_K$ induces an endomorphism $\varphi$ and an $\mathcal{O}_K^{\times}$-action on $A_{\mathrm{mv},E}$. Let $\varphi_q\coloneq \varphi^f$. 
	The functor $D_{A_{\mathrm{mv},E}}^{(i)}$ sends finite type continuous $\mathcal{O}_E$-representations of $\gal(\overline{K}/K)$ to finitely presented multivariable étale $(\varphi_q,\mathcal{O}_K^{\times})$-modules over $A_{\mathrm{mv},E}$.
	
	Let $A_{\mathrm{LT},E,\sigma_{0}}$ denote the $p$-adic completion of $\mathcal{O}_E\otimes_{\sigma_{0},\mathcal{O}_K}\mathcal{O}_K[\negthinspace[T_{\mathrm{LT}}]\negthinspace][\frac{1}{T_{\mathrm{LT}}}]$. This is a discrete valuation ring with a continuous $(\varphi_q,\mathcal{O}_K^{\times})$-action given by the Lubin-Tate actions of $p$ and $\mathcal{O}_K^{\times}$, see (\ref{phi_q,O_k^* action ass. to Lubin-Tate formal groups}). The theory of Lubin-Tate $(\varphi,\Gamma)$-modules associates to every finite type continuous $\mathcal{O}_E$-representation $\rho$ of $\gal(\overline{K}/K)$ a finite type étale $(\varphi_q,\mathcal{O}_K^{\times})$-module $D_{\mathrm{LT},\sigma_{0}}(\rho)$ over $A_{\mathrm{LT},E,\sigma_{0}}$, and the functor $D_{\mathrm{LT},\sigma_{0}}$ is an equivalence of categories. Starting with $A_{\mathrm{LT},E,\sigma_{0}}$, we now explain how to construct the functor $D_{A_{\mathrm{mv},E}}^{(i)}$. Motivated by the construction of $D_A^{(i)}$, we consider the following diagram of $\mathcal{O}_E$-algebras:
	\begin{equation}\label{diagram of topological rings}
		\begin{tikzcd}
			A_{\mathrm{mv},E}\arrow[r,hook]& W_E(A_{\infty})\arrow[r,hook,"m"]&W_E(A_{\infty}')\\
			& & A_{\mathrm{LT},E,\sigma_{0}}\arrow[u,hook,"\mathrm{pr}_i"]
		\end{tikzcd}
	\end{equation}
	where $W_E(A_{\infty})\coloneq \mathcal{O}_E\otimes_{W(\mathbb{F})}W(A_{\infty})$, $W_E(A_{\infty}')\coloneq \mathcal{O}_E\otimes_{W(\mathbb{F})}W(A_{\infty}')$, and $m: W_E(A_{\infty})\to W_E(A_{\infty}')$ is the lift of $m: A_{\infty}\to A_{\infty}'$. By the construction of Witt vectors, the $\mathbb{F}$-linear $(\varphi_q,\mathcal{O}_K^{\times})$-action on $A_{\infty}$ (resp. the $\mathbb{F}$-linear $(\varphi_q,(\mathcal{O}_K^{\times})^f)$-action on $A_{\infty}'$) uniquely lifts to an $\mathcal{O}_E$-linear $(\varphi_q,\mathcal{O}_K^{\times})$-action on $W_E(A_{\infty})$ (resp. an $\mathcal{O}_E$-linear $(\varphi_q,(\mathcal{O}_K^{\times})^f)$-action on $W_E(A_{\infty}')$).
	
	The first challenge is to define the $(\varphi_q,(\mathcal{O}_K^{\times})^{f})$-equivariant inclusion $\mathrm{pr}_i:A_{\mathrm{LT},E,\sigma_{0}}\hookrightarrow W_E(A_{\infty}')$, where $(\mathcal{O}_K^{\times})^{f}$ acts on $A_{\mathrm{LT},E,\sigma_{0}}$ via the $i$-th projection $(\mathcal{O}_K^{\times})^{f}\to \mathcal{O}_K^{\times}$. A natural candidate is the map $h: A_{\mathrm{LT},E,\sigma_{0}}\hookrightarrow W_E(A_{\infty}')$ sending $T_{\mathrm{LT}}$ to the Teichmüller lift $[T_{\mathrm{LT},i}]$. However, $h$ does not commute with $\varphi_q$: $\varphi_q(h(T_{\mathrm{LT}}))=\varphi_q([T_{\mathrm{LT},i}])=[T_{\mathrm{LT},i}^q]$, while $h(\varphi_q(T_{\mathrm{LT}}))=h(p_{\mathrm{LT}}(T_{\mathrm{LT}}))=p_{\mathrm{LT}}([T_{\mathrm{LT},i}])$, where $p_{\mathrm{LT}}(T)\in \mathcal{O}_K[\negthinspace[T]\negthinspace]$ is a fixed Frobenius power series associated to the uniformizer $p\in\mathcal{O}_K$. Hence $h$ is not the desired map. The correct way to define $\mathrm{pr}_i:A_{\mathrm{LT},E,\sigma_{0}}\hookrightarrow W_E(A_{\infty}')$ is to use the ring $\varinjlim_{\varphi_q}A_{\mathrm{LT},E,\sigma_{0}}$, the colimit of $A_{\mathrm{LT},E,\sigma_{0}}$ indexed by $\mathbb{Z}_{\geq 0}$. We endow $\varinjlim_{\varphi_q}A_{\mathrm{LT},E,\sigma_{0}}$ with a metric topology such that the topological completion of $\varinjlim_{\varphi_q}A_{\mathrm{LT},E,\sigma_{0}}$ is $W_E\left(\mathbb{F}(\negthinspace(T_{\mathrm{LT}}^{1/p^{\infty}})\negthinspace)\right)$. Then $\mathrm{pr}_i$ is defined as the composition:
	\[A_{\mathrm{LT},E,\sigma_{0}}\hookrightarrow \varinjlim_{\varphi_q}A_{\mathrm{LT},E,\sigma_{0}} \hookrightarrow W_E\left(\mathbb{F}(\negthinspace(T_{\mathrm{LT}}^{1/p^{\infty}})\negthinspace)\right) \xrightarrow{[T_{\mathrm{LT}}]\mapsto [T_{\mathrm{LT},i}]}W_E(A_{\infty}').\]
	
	The same approach applies to define the $(\varphi_q,\mathcal{O}_K^{\times})$-equivariant injection $A_{\mathrm{mv},E}\hookrightarrow W_E(A_{\infty})$ as the composition $A_{\mathrm{mv},E}\hookrightarrow\varinjlim_{\varphi_q}A_{\mathrm{mv},E}\hookrightarrow W_E(A_{\infty})$, where $\varinjlim_{\varphi_q}A_{\mathrm{mv},E}$ is the colimit of $A_{\mathrm{mv},E}$ indexed by $\mathbb{Z}_{\geq 0}$. With an appropriate metric topology, the topological completion of $\varinjlim_{\varphi_q}A_{\mathrm{mv},E}$ is isomorphic to $W_E(A_{\infty})$. For details, see Section \ref{section 2}.

	There are two key results related to (\ref{diagram of topological rings}), both of which play an essential role in the construction of the functor $D_{A_{\mathrm{mv},E}}^{(i)}$:
	
	\begin{thm}\label{main_thm}
		\begin{enumerate}[label=(\alph*), ref=\ref{main_thm}.(\alph*)]
			\item 
			The functor $W_E(A_{\infty})\otimes_{A_{\mathrm{mv},E}}-$ sending finitely presented étale $(\varphi_q,\mathcal{O}_K^{\times})$-modules over $A_{\mathrm{mv},E}$ to finitely presented étale $(\varphi_q,\mathcal{O}_K^{\times})$-modules over $W_E(A_{\infty})$
			is an equivalence of abelian categories. See Corollary \ref{descent of f.p. phi_q,O_K*-modules}.\label{thm 0a}
			\item 
			The functor $W_E(A_{\infty}')\otimes_{m,W_E(A_{\infty})}-$ sending finitely presented étale $(\varphi_q,\mathcal{O}_K^{\times})$-modules over $W_E(A_{\infty})$ to finitely presented étale $(\varphi_q,(\mathcal{O}_K^{\times})^f)$-modules over $W_E(A_{\infty}')$
			is an equivalence of abelian categories, and a quasi-inverse is given by $(-)^{\Delta_1}$. 
			See Theorem \ref{descent of f.p. from W_E(A{infty}') to W_E(A{infty})}.\label{thm 0b}
		\end{enumerate}
	\end{thm}
	
	Here we briefly explain the proof of Theorem \ref{main_thm}. 
	
	For Theorem \ref{thm 0a}, we first apply variants of a result of Katz (Proposition \ref{Katz's method: descend using phi_q}, Proposition \ref{perfect case of Katz's result}), along with a limit argument, to show that every finite projective étale $\varphi_q$-module over $W_E(A_{\infty})$ descends to a finite projective étale $\varphi_q$-module over $A_{\mathrm{mv},E}$. For a general finitely presented étale $\varphi_q$-module $M_{\infty}$ over $W_E(A_{\infty})$, a standard lemma (Lemma \ref{existence of equivariant surjection of phi_q-modules}) ensures the existence of a $\varphi_q$-equivariant surjection $f: F_{\infty}\twoheadrightarrow M_{\infty}$, where $F_{\infty}$ is a finite projective étale $\varphi_q$-module over $W_E(A_{\infty})$ that descends to a finite projective étale $\varphi_q$-module $F$ over $A_{\mathrm{mv},E}$. Since the map $A_{\mathrm{mv},E}\hookrightarrow W_E(A_{\infty})$ is faithfully flat (Proposition \ref{A_{mv,E} to W_E(A_{infty}) is faithfully flat}), we may view $F$ as an $A_{\mathrm{mv},E}$-submodule of $F_{\infty}$, then we can check that $f(F)\subset M_{\infty}$ is a well-defined finitely presented étale $\varphi_q$-module over $A_{\mathrm{mv},E}$. Finally, incorporating the $\mathcal{O}_K^{\times}$-action completes the proof of Theorem \ref{thm 0a}.
	
	For Theorem \ref{thm 0b}, we first show that the continuous cohomology group $H^{1}_{\mathrm{cont}}(\Delta_1,A_{\infty}')$ vanishes (Theorem \ref{vanishing of continuous cohomology of A_{infty}'}). This follows from the fact that the morphism of affinoid perfectoid spaces induced by $A_{\infty}\hookrightarrow A_{\infty}'$ is a pro-étale $\Delta_1$-torsor (\cite[Proposition 2.4.4]{breuil2023multivariable}). Using this vanishing result, we prove that Theorem \ref{thm 0b} is true for finite projective modules. Finally, since the structure of finitely presented étale $\varphi_q$-modules over $W_E(A_{\infty}')$ is clear (Proposition \ref{description of finitely presented etale phi_q-module over W_E(A_{infty}')}), we apply a dévissage argument to complete the proof of Theorem \ref{thm 0b}.
	
	With (\ref{diagram of topological rings}) at hand, the construction of the functor $D_{A_{\mathrm{mv},E}}^{(i)}$ is easy to describe. For a finite type $\mathcal{O}_E$-representation $\rho$ of $\gal(\overline{K}/K)$, there is a finite type étale $(\varphi_q,\mathcal{O}_K^{\times})$-module $D_{\mathrm{LT},\sigma_{0}}(\rho)$ over $A_{\mathrm{LT},E,\sigma_{0}}$ associated to $\rho$. Then $W_E(A_{\infty}')\otimes_{\mathrm{pr}_i,A_{\mathrm{LT},E,\sigma_{0}}}D_{\mathrm{LT},\sigma_{0}}(\rho)$ is a finitely presented étale $(\varphi_q,(\mathcal{O}_K^{\times})^f)$-module over $W_E(A_{\infty}')$. By Theorem \ref{thm 0b}, $\left(W_E(A_{\infty}')\otimes_{\mathrm{pr}_i,A_{\mathrm{LT},E,\sigma_{0}}}D_{\mathrm{LT},\sigma_{0}}(\rho)\right)^{\Delta_1}$ is a finitely presented étale $(\varphi_q,\mathcal{O}_K^{\times})$-module over $W_E(A_{\infty})$, which canonically descends to a finitely presented étale $(\varphi_q,\mathcal{O}_K^{\times})$-module $D_{A_{\mathrm{mv},E}}^{(i)}(\rho)$ over $A_{\mathrm{mv},E}$ by Theorem \ref{thm 0a}. 
	
	Here are the main properties of the functor $D_{A_{\mathrm{mv},E}}^{(i)}$:
	\begin{thm}\label{thm in intro}
		Let $\rho$ be a finite type $\mathcal{O}_E$-representation of $\gal(\overline{K}/K)$, $0\leq i\leq f-1$.
		\begin{enumerate}
			\item [(a)]
			There is a functorial isomorphism of finitely presented étale $(\varphi_q,\mathcal{O}_K^{\times})$-modules over $A_{\mathrm{mv},E}$:
			\begin{align}
				\phi_i: A_{\mathrm{mv},E}\otimes_{\varphi,A_{\mathrm{mv},E}}D_{A_{\mathrm{mv},E}}^{(i)}(\rho)\xrightarrow{\sim} D_{A_{\mathrm{mv},E}}^{(i+1)}(\rho).
			\end{align}
			Moreover, the composition $\phi_{f-1}\circ \phi_{f-2}\circ\cdots\circ \phi_0: A_{\mathrm{mv},E}\otimes_{\varphi_q,A_{\mathrm{mv},E}}D_{A_{\mathrm{mv},E}}^{(0)}(\rho)\to D_{A_{\mathrm{mv},E}}^{(0)}(\rho)$ is the linearization map $\mathrm{id}\otimes{\varphi_q}$, see Lemma \ref{ralation of different f.p. D_{A_E}^i}.
			\item [(b)]
			There exists an embedding $\sigma_d\in\{\sigma_{0},\dots,\sigma_{f-1}\}$ such that there is a canonical isomorphism of finitely presented étale $(\varphi_q,\mathbb{Z}_p^{\times})$-modules over $A_{\mathrm{cycl},E}$:
			\begin{align}
				A_{\mathrm{cycl},E}\otimes_{\mathrm{tr},A_{\mathrm{mv},E}}D_{A_{\mathrm{mv},E}}^{(0)}(\rho)\cong D_{\mathrm{cycl},\sigma_d}(\rho)
			\end{align}
			where $A_{\mathrm{cycl},E}$ is the $p$-adic completion of $\mathcal{O}_E[\negthinspace[\mathbb{Z}_p]\negthinspace]\left[\frac{1}{T_{\mathrm{cycl}}}\right]$ endowed with a continuous $(\varphi_q,\mathbb{Z}_p^{\times})$-action, $T_{\mathrm{cycl}}\coloneq [1]-1\in \mathcal{O}_E[\negthinspace[\mathbb{Z}_p]\negthinspace]$, $\mathrm{tr}:A_{\mathrm{mv},E}\to A_{\mathrm{cycl},E}$ is induced by $\mathrm{tr}: \mathcal{O}_K\to\mathbb{Z}_p$, and $D_{\mathrm{cycl},\sigma_d}(\rho)$ is the cyclotomic $(\varphi_q,\mathbb{Z}_p^{\times})$-module over $A_{\mathrm{cycl},E}$ using the embedding $\sigma_{d}: \mathcal{O}_K\hookrightarrow \mathcal{O}_E$, see Theorem \ref{relation to classical cycltomic (phi,Gamma)-modules}.
			\item [(c)]
			The functor $D_{A_{\mathrm{mv},E}}^{(i)}$ from the category of finite type continuous $\mathcal{O}_E$-representations of $\gal(\overline{K}/K)$ to the category of finitely presented étale $(\varphi_q,\mathcal{O}_K^{\times})$-module over $A_{\mathrm{mv},E}$ is fully faithful and exact. Moreover, if $\rho$ is a finite free $\mathcal{O}_E$-representations of $\gal(\overline{K}/K)$ of rank $r$, then $D_{A_{\mathrm{mv},E}}^{(i)}(\rho)$ is a finite free étale $(\varphi_q,\mathcal{O}_K^{\times})$-module over $A_{\mathrm{mv},E}$ of rank $r$, see Theorem \ref{f.p. D_{A_{mv,E}}^{(i)} is fully faithful}.
		\end{enumerate}
	\end{thm}
	
	Inspired by \cite{breuil2023multivariable}, for a finite type $\mathcal{O}_E$-representation $\rho$ of $\gal(\overline{K}/K)$, we define $D_{A_{\mathrm{mv},E}}^{\otimes}(\rho)\coloneq \bigotimes_{i=0}^{f-1}D_{A_{\mathrm{mv},E}}^{(i)}(\rho)$ which is a finitely presented étale $(\varphi,\mathcal{O}_K^{\times})$-module over $A_{\mathrm{mv},E}$. We expect the functor $D_{A_{\mathrm{mv},E}}^{\otimes}$ to be related to $p$-adic representations of $\mathrm{GL}_2(K)$.
	
	Let $B_{\mathrm{mv},E}\coloneq A_{\mathrm{mv},E}[1/p]$, we can also associate to every finite dimensional $E$-representation $\rho$ of $\gal(\overline{K}/K)$ a finite free étale $(\varphi_q,\mathcal{O}_K^{\times})$-module over $B_{\mathrm{mv},E}$: we choose an $\mathcal{O}_E$-lattice $\rho_0$ of $\rho$ and define $D_{B_{\mathrm{mv},E}}^{(i)}(\rho)\coloneq D_{A_{\mathrm{mv},E}}^{(i)}(\rho_0)[1/p]$. It is straightforward to verify that $D_{B_{\mathrm{mv},E}}^{(i)}(\rho)$ is well-defined. Moreover, it follows from Theorem \ref{thm in intro} that $D_{B_{\mathrm{mv},E}}^{(i)}$ is also exact, fully faithful and preserves the rank of representations. Finally, we define $D_{B_{\mathrm{mv},E}}^{\otimes}(\rho)\coloneq \bigotimes_{i=0}^{f-1}D_{B_{\mathrm{mv},E}}^{(i)}(\rho)$ which is a finite type étale $(\varphi,\mathcal{O}_K^{\times})$-module over $B_{\mathrm{mv},E}$. \\
	
	We now describe the content of each section. 
	
	In Section \ref{section 1}, we introduce the multivariable topological ring $A_{\mathrm{mv},E}$. This ring is a $p$-adically complete Noetherian domain, equipped with a continuous faithfully flat finite type endomorphism $\varphi_q$ and a continuous $\mathcal{O}_K^{\times}$-action. 
	
	In Section \ref{section 2}, we define the $(\varphi_q,\mathcal{O}_K^{\times})$-equivariant injection $A_{\mathrm{mv},E}\hookrightarrow W_E(A_{\infty})$. We first introduce the ring $\varinjlim_{\varphi_q}A_{\mathrm{mv},E}$, defined as the colimit of $A_{\mathrm{mv},E}$ indexed by $\mathbb{Z}_{\ge 0}$ with transition maps given by $\varphi_q$. This ring is endowed with a metric topology, and the continuous $(\varphi_q,\mathcal{O}_K^{\times})$-action on $A_{\mathrm{mv},E}$ uniquely extends to $\varinjlim_{\varphi_q}A_{\mathrm{mv},E}$. We use the theory of strict $p$-rings to prove that the topological completion of $\varinjlim_{\varphi_q}A_{\mathrm{mv},E}$ is $(\varphi_q,\mathcal{O}_K^{\times})$-equivariantly isomorphic to $W_E(A_{\infty})$, then we construct the desired injection as the composition
	\[A_{\mathrm{mv},E}\hookrightarrow\varinjlim_{\varphi_q}A_{\mathrm{mv},E}\hookrightarrow W_E(A_{\infty}).\]
	A similar construction applies to $A_{\mathrm{LT},E,\sigma_{0}}$, which yields a $(\varphi_q,\mathcal{O}_K^{\times})$-equivariant injection $A_{\mathrm{LT},E,\sigma_{0}}\hookrightarrow W_E\left(\mathbb{F}(\negthinspace(T_{\mathrm{LT}}^{1/p^{\infty}})\negthinspace)\right)$. Then we define the map $\mathrm{pr}_i$ as the composition
	\[\mathrm{pr}_i:A_{\mathrm{LT},E,\sigma_{0}}\hookrightarrow \varinjlim_{\varphi_q}A_{\mathrm{LT},E,\sigma_{0}}\hookrightarrow W_E\left(\mathbb{F}(\negthinspace(T_{\mathrm{LT}}^{1/p^{\infty}})\negthinspace)\right)\xrightarrow{[T_{\mathrm{LT}}]\mapsto [T_{\mathrm{LT},i}]}W_E(A_{\infty}').\]
	
	In Section \ref{section 3}, we prove Theorem \ref{thm 0b}. First, we show that every finite projective étale $\varphi_q$-module over $W_E(A_{\infty})/\varpi^n$ ($n\geq 1$) descends to a finite projective étale $\varphi_q$-module over $A_{\mathrm{mv},E}/\varpi^n$. Passing to the projective limit, we deduce that every finite free étale $\varphi_q$-module over $W_E(A_{\infty})$ descends to a finite projective étale $\varphi_q$-module over $A_{\mathrm{mv},E}$.
	Using a standard lemma in algebraic $K$-theory (Lemma \ref{existence of equivariant surjection of phi_q-modules}), we extend this result from finite free modules to finitely presented étale $\varphi_q$-modules. Finally, we take the $\mathcal{O}_K^{\times}$-action into consideration.
	
	In Section \ref{section 4}, we prove Theorem \ref{thm 0a}. We begin with a vanishing result for the continuous cohomology group $H^1_{\mathrm{cont}}(\Delta_1,A_{\infty}')$, which implies that the functor $(-)^{\Delta_1}$ sends finite free étale $(\varphi_q,(\mathcal{O}_K^{\times})^f)$-modules over $W_E(A_{\infty}')/\varpi^n$ to finite free étale $(\varphi_q,\mathcal{O}_K^{\times})$-modules over $W_E(A_{\infty})/\varpi^n$ for any $n\geq 1$. Passing to the projective limit, we prove that $(-)^{\Delta_1}$ is a functor sending finite free étale $(\varphi_q,(\mathcal{O}_K^{\times})^f)$-modules over $W_E(A_{\infty}')$ to finite free étale $(\varphi_q,(\mathcal{O}_K^{\times})^f)$-modules over $W_E(A_{\infty})$, with a quasi-inverse given by $W_E(A_{\infty}')\otimes_{m,W_E(A_{\infty})}-$. Using a structure theorem for finitely presented étale $\varphi_q$-modules over $W_E(A_{\infty}')$ (Proposition \ref{description of finitely presented etale phi_q-module over W_E(A_{infty}')}) and a dévissage argument, we conclude the proof.
	
	In Section \ref{section 5}, we recall some results on cyclotomic $(\varphi_q,\mathbb{Z}_p^{\times})$-modules and Lubin-Tate $(\varphi_q,\mathcal{O}_K^{\times})$-modules associated to finite type $\mathcal{O}_E$-representations of $\gal(\overline{K}/K)$. 
	
	In Section \ref{section 6}, we define the functor $D_{A_{\mathrm{mv},E}}^{(i)}$ for $0\leq i\leq f-1$ and show its connection to the cyclotomic $(\varphi_q,\mathbb{Z}_p^{\times})$-modules. As a consequence, we deduce the fully faithfulness of $D_{A_{\mathrm{mv},E}}^{(i)}$. If $K=\qp$, the functor $D_{A_{\mathrm{mv},E}}^{(i)}$ coincides with the classical cyclotomic $(\varphi_q,\mathbb{Z}_p^{\times})$-modules. However, for $K\neq \qp$, $D_{A_{\mathrm{mv},E}}^{(i)}$ is not essentially surjective. Additionally, we define the functors  $D_{A_{\mathrm{mv},E}}^{\otimes}$, $D_{B_{\mathrm{mv},E}}^{(i)}$ and $D_{B_{\mathrm{mv},E}}^{\otimes}$ and investigate some basic properties.
	
	We now give the main general notation (many have already been introduced, more specific notation will be introduced throughout).
	Let $p$ be a prime number. Let $K$ be a finite unramified extension of $\qp$ with ring of integers $\mathcal{O}_K$, uniformizer $p$ and residue field $\mathbb{F}_q$ where $q=p^f$, and let $E$ be a finite extension of $K$ with ring of integers $\mathcal{O}_E$, uniformizer $\varpi$ and residue field $\mathbb{F}$. Let $q'=q^{f'}\coloneq\#\mathbb{F}$. We fix an algebraic closure $\overline{K}$ of $K$, and let $\mathbb{C}_p$ be the $p$-adic completion of $\overline{K}$. We define $E_0\coloneq W(\mathbb{F})[1/p]$, which is the maximal unramified subextension of $E/K$, and we fix an embedding $E_0\hookrightarrow E$. We fix an embedding $\sigma_{0}: \mathcal{O}_K\hookrightarrow W(\mathbb{F})$, and for $1\leq i\leq f-1$, let $\sigma_i:\mathcal{O}_K\hookrightarrow W(\mathbb{F})$ be the unique embedding such that $\sigma_{i}(x)-\sigma_0(x)^{p^{i}}\in p W(\mathbb{F})$ for every $x\in\mathcal{O}_K$. We also denote by $\sigma_i$ the composition $K\xhookrightarrow{\sigma_i}E_0\hookrightarrow E$ for $0\leq i\leq f-1$. For a perfect $\mathbb{F}_p$-algebra $R$, we denote by $W(R)$ the ring of Witt vectors. For $r\in R$, we denote by $[r]\in W(R)$ the Teichmüller lift of $r$. If in addition the maximal algebraic extension of $\mathbb{F}_p$ contained in $R$ is $\mathbb{F}_q$ or $\mathbb{F}$, we define the ring of ramified Witt vectors $W_E(R)\coloneq \mathcal{O}_E\otimes_{\sigma_{0},\mathcal{O}_K}W(R)$ or $W_E(R)\coloneq \mathcal{O}_E\otimes_{W(\mathbb{F})}W(R)$. For a topological ring $R$, we denote by $R^{\circ}$ the set of power-bounded elements of $R$. 
	If $R$ is perfectoid, we denote its tilt by $R^{\flat}$ (\cite[Proposition 5.17]{scholze2012perfectoid}).
	
	\textbf{Acknowledgements}:
	We thank Christophe Breuil for suggesting this problem, for his helpful discussions, and for his careful reading of earlier drafts of this paper. His guidance has been invaluable throughout the course of this work. We also express our gratitude to Laurent Berger and Stefano Morra for their insightful comments and to Zhicheng Lyu for identifying typos.
	
	This work was supported by the Ecole Doctorale de Mathématiques Hadamard.
	
	\section{The multivariable topological ring $A_{\mathrm{mv},E}$}\label{section 1}
	
	In this section, we introduce the multivariable topological ring $A_{\mathrm{mv},E}$ endowed with a continuous endomorphism $\varphi_q$ and a continuous action by $\mathcal{O}_K^{\times}$ commuting with $\varphi_q$. The ring $A_{\mathrm{mv},E}$ is the $p$-adic analogue of the ring $A$ (\cite[$\S$2.2]{breuil2023multivariable}).
	
	Let $K$ be a finite unramified extension of $\qp$ with ring of integers $\mathcal{O}_K$, uniformizer $p$ and residue field $\mathbb{F}_q$ where $q=p^f$, and let $E$ be a finite extension of $K$ with ring of integers $\mathcal{O}_E$, uniformizer $\varpi$ and residue field $\mathbb{F}$. We define $E_0\coloneq W(\mathbb{F})[1/p]$, which is the maximal unramified subextension of $E/K$, and we fix an embedding $E_0\hookrightarrow E$. We fix an embedding $\sigma_{0}: K\hookrightarrow E_0$ once and for all, and for $1\leq i\leq f-1$, let $\sigma_i:K\hookrightarrow E_0$ be the unique embedding such that $\sigma_{i}(x)-\sigma_0(x)^{p^{i}}\in p W(\mathbb{F})$ for every $x\in\mathcal{O}_K$. We also denote by $\sigma_i$ the composition $K\xhookrightarrow{\sigma_i}E_0\hookrightarrow E$ for $0\leq i\leq f-1$. Let $N_0$ be the group of unipotent matrices of $\mathrm{GL}_2(\mathcal{O}_K)$, i.e. $N_0=\begin{pmatrix}
		1 & \mathcal{O}_K\\
		0& 1
	\end{pmatrix}$, and we consider the Iwasawa algebra $\mathcal{O}_K[\negthinspace[N_0]\negthinspace]$. If $K\neq \mathbb{Q}_2$, for $0\leq i\leq f-1$, we put
	\begin{align}\label{choice of variables}
		T_i\coloneq \sum_{\lambda\in\mathbb{F}_q^{\times}}[\lambda]^{-p^i}\begin{pmatrix}
			1 & [\lambda]\\
			0& 1
		\end{pmatrix}\in\mathcal{O}_K[\negthinspace[N_0]\negthinspace]
	\end{align}
	where $[\lambda]\in\mathcal{O}_K^{\times}$ is the Teichmüller lift of $\lambda\in \mathbb{F}_q^{\times}$, and for $K=\mathbb{Q}_2$ we put $T_0\coloneq\begin{pmatrix}
		1 & 1\\
		0& 1
	\end{pmatrix}-1$,
	then
	\[\mathcal{O}_K[\negthinspace[N_0]\negthinspace]=\mathcal{O}_K[\negthinspace[T_0,\dots,T_{f-1}]\negthinspace].\]
	The maximal ideal of $\mathcal{O}_K[\negthinspace[N_0]\negthinspace]$ is generated by $p,T_0,\dots,T_{f-1}$. We denote by $\mathfrak{p}_{N_0}$ the prime ideal of $\mathcal{O}_K[\negthinspace[N_0]\negthinspace]$ generated by $T_0,\dots,T_{f-1}$. Let $S$ be the multiplicative system generated by $\prod_{i=0}^{f-1}T_i$, on the ring $S^{-1}\mathcal{O}_K[\negthinspace[N_0]\negthinspace]$, we can define a valuation
	\[v_{N_0}\left(\frac{P}{(\prod_{i=0}^{f-1}T_{i})^k}\right)\coloneq  \mathrm{ord}_{\mathfrak{p}_{N_0}}(P)-kf\in\mathbb{Z}\cup\{+\infty\},\]
	where $P\in \mathcal{O}_K[\negthinspace[N_0]\negthinspace]=\mathcal{O}_K[\negthinspace[T_0,\dots,T_{f-1}]\negthinspace]$ and $\mathrm{ord}_{\mathfrak{p}_{N_0}}(P)=\sup\left\{m\in\mathbb{Z}_{\geq 0}: P\in \mathfrak{p}_{N_0}^m\right\}\in \mathbb{Z}_{\geq 0}\cup\{+\infty\}$.
	Let $\left(S^{-1}\mathcal{O}_K[\negthinspace[N_0]\negthinspace]\right)^{\wedge}$ be the completion of $S^{-1}\mathcal{O}_K[\negthinspace[N_0]\negthinspace]$ with respect to the valuation $v_{N_0}$.  
	We denote by $A_{\mathrm{mv},K}$ the $p$-adic completion of $\left(S^{-1}\mathcal{O}_K[\negthinspace[N_0]\negthinspace]\right)^{\wedge}$, i.e.
	\begin{align}\label{defn of A_{mv,K}}
		A_{\mathrm{mv},K}\coloneq \varprojlim_n \left(S^{-1}\mathcal{O}_K[\negthinspace[N_0]\negthinspace]\right)^{\wedge}/p^n.
	\end{align}
	Here, $\mathrm{mv}$ is the abbreviation for \textit{multivariable}. 
	Let $\left(\mathcal{O}_K[\negthinspace[T_0]\negthinspace]\left[\frac{1}{T_0}\right]\right)^{\wedge}$ denote the $p$-adic completion of $\mathcal{O}_K[\negthinspace[T_0]\negthinspace]\left[\frac{1}{T_0}\right]$, then any element $P$ of $A_{\mathrm{mv},K}$ can be uniquely written in the following form
	\begin{align}\label{expression of element of A_{mathrm{m.v.},E_0}}
		P=\sum_{\underline{n}\in\mathbb{Z}^{f-1}}f_{\underline{n}}(T_0)\prod_{i=1}^{f-1} \left(\frac{T_i}{T_0}\right)^{n_i}
	\end{align}
	with $f_{\underline{n}}(T_0)\in \left(\mathcal{O}_K[\negthinspace[T_0]\negthinspace]\left[\frac{1}{T_0}\right]\right)^{\wedge}$ such that for any $m\geq 1$, $f_{\underline{n}}(T_0)\in T_0^m\mathcal{O}_K[\negthinspace[T_0]\negthinspace]+p^m\left(\mathcal{O}_K[\negthinspace[T_0]\negthinspace]\left[\frac{1}{T_0}\right]\right)^{\wedge}$ for almost all $\underline{n}\in\mathbb{Z}^{f-1}$. 
	\begin{lem}\label{A_{mathrm{m.v.},K} is noetherian}
		The ring $A_{\mathrm{mv},K}$ is a Noetherian domain.
	\end{lem}
	
	\begin{proof}
		Since $\left(S^{-1}\mathcal{O}_K[\negthinspace[N_0]\negthinspace]\right)^{\wedge}/p\cong A_q$ is a Noetherian domain (\cite[Corollary 3.1.1.2]{breuil2023conjectures}), where $A_q$ is defined in \cite[$\S$2.2]{breuil2023multivariable}, by \cite[\href{https://stacks.math.columbia.edu/tag/05GH}{Tag 05GH}]{stacks-project}, $A_{\mathrm{mv},K}$ is a Noetherian ring. Using (\ref{expression of element of A_{mathrm{m.v.},E_0}}), we deduce that $p$ is not a zero-divisor of $A_{\mathrm{mv},K}$. Then, by reducing modulo $p$, one can easily deduce that $A_{\mathrm{mv},K}$ is a domain.
	\end{proof}
	
	\begin{defn}\label{continuous (phi,Gamma)-action}
		Let $\Gamma$ be a topological group, and let $R$ be a topological ring. A \textit{continuous} $(\varphi,\Gamma)$-\textit{action} on $R$ is a continuous endomorphism $\varphi$ of $R$ and a continuous action of $\Gamma$ on $R$ which commutes with $\varphi$. 
	\end{defn}
	
	We endow $\mathcal{O}_K[\negthinspace[N_0]\negthinspace]$ with the $(p,T_0,\dots,T_{f-1})$-adic topology, then $\mathcal{O}_K[\negthinspace[N_0]\negthinspace]$ is a complete Hausdorff adic ring. Since $N_0\cong\mathcal{O}_K$, the multiplication by $\mathcal{O}_K$ on itself induces a continuous $(\varphi,\mathcal{O}_K^{\times})$-action ($\varphi$ is given by the multiplication by $p$) which satisfies
	\begin{align}\label{phi,O_K^{*} action on A_{mv,E_0}}
		\begin{cases}
			\varphi(T_i)\equiv T_{i-1}^p \mod p, & \\
			a(T_i)\equiv [\overline{a}]^{p^i} T_i \mod p\mathcal{O}_K[\negthinspace[N_0]\negthinspace]+(T_0,\dots,T_{f-1})^p, & a\in\mathcal{O}_K^{\times}. 
		\end{cases}
	\end{align}
	Hence the continuous $(\varphi,\mathcal{O}_K^{\times})$-action on $\mathcal{O}_K[\negthinspace[N_0]\negthinspace]$ extends uniquely to a continuous $(\varphi,\mathcal{O}_K^{\times})$-action on $A_{\mathrm{mv},K}$. Define $A_{\mathrm{mv},E_0}\coloneq W(\mathbb{F})\otimes_{\sigma_0,\mathcal{O}_K}A_{\mathrm{mv},K}$, $T_{\sigma_i}\coloneq 1\otimes T_i$, the continuous $(\varphi,\mathcal{O}_K^{\times})$-action on $A_{\mathrm{mv},K}$ extends $W(\mathbb{F})$-linearly to a continuous $(\varphi,\mathcal{O}_K^{\times})$-action on $A_{\mathrm{mv},E_0}$. Let $\varphi_q\coloneq \varphi^f$, we have:
	\begin{align}\label{phi_q,O_K^{*} action on A_{mv,E_0}}
		\begin{cases}
			\varphi_q(T_{\sigma_{i}})\in T_{{\sigma_{i}}}^q + pW(\mathbb{F})[\negthinspace[N_0]\negthinspace], & \\
			a(T_{\sigma_{i}})\in  \sigma_{i}([\overline{a}])T_{\sigma_{i}} + pW(\mathbb{F})[\negthinspace[N_0]\negthinspace]+\mathfrak{p}_{N_0}^p, & a\in\mathcal{O}_K^{\times}. 
		\end{cases}
	\end{align}
	where $\mathfrak{p}_{N_0}=T_{\sigma_{0}}W(\mathbb{F})[\negthinspace[N_0]\negthinspace]+\cdots+T_{\sigma_{f-1}}W(\mathbb{F})[\negthinspace[N_0]\negthinspace]$ is a prime ideal of $W(\mathbb{F})[\negthinspace[N_0]\negthinspace]$.
	
	Recall that there is a Tate algebra $A\coloneq\mathbb{F}(\negthinspace(T_{\sigma_{0}})\negthinspace)\left\langle\left(\frac{T_{\sigma_i}}{T_{\sigma_{0}}}\right)^{\pm 1}: 1\leq i\leq f-1\right\rangle$, and there is an $\mathbb{F}$-linear continuous $(\varphi_q,\mathcal{O}_K^{\times})$-action on $A$ (see \cite[$\S$2.2 (18)]{breuil2023multivariable}, where the variable $Y_{\sigma_i}$ of \textit{loc. cit.} is denoted by $T_{\sigma_{i}}$ here) satisfying:
	\begin{align}\label{action of phi_q on A}
		\begin{cases}
			\varphi_q(T_{\sigma_{i}})= T_{\sigma_{i}}^q & \\
			a(T_{\sigma_{i}})\in  \sigma_{i}(\overline{a})T_{\sigma_{i}}+ \mathfrak{m}_{N_0}^p, & a\in\mathcal{O}_K^{\times}
		\end{cases}
	\end{align}
	where $\mathfrak{m}_{N_0}=T_{\sigma_{0}}\mathbb{F}[\negthinspace[N_0]\negthinspace]+\cdots+T_{\sigma_{f-1}}\mathbb{F}[\negthinspace[N_0]\negthinspace]$ is the maximal ideal of $\mathbb{F}[\negthinspace[N_0]\negthinspace]$. 
	The subset of $A$ consisting of power bounded elements is $A^{\circ}=\mathbb{F}[\negthinspace[T_{\sigma_{0}}]\negthinspace]\left\langle\left(\frac{T_{\sigma_i}}{T_{\sigma_{0}}}\right)^{\pm 1}: 1\leq i\leq f-1\right\rangle$, which is an open adic subring of $A$. Then $A_{\mathrm{mv},E_0}/p\to A, T_{\sigma_i} \mapsto T_{\sigma_i}$ is an isomorphism which commutes with $\varphi_q$ and the action of $\mathcal{O}_K^{\times}$.

	\begin{lem}\label{phi_q is flat}
		Let $A_{\mathrm{mv},E}\coloneq \mathcal{O}_E\otimes_{W(\mathbb{F})}A_{\mathrm{mv},E_0}$, and we still denote by $\varphi_q$ the endomorphism $\mathrm{id}\otimes\varphi_q$ of $A_{\mathrm{mv},E}$. Then $\varphi_q: A_{\mathrm{mv},E}\to A_{\mathrm{mv},E}$ is a finitely presented faithfully flat ring map.
	\end{lem}
	\begin{proof}
		Let $M= A_{\mathrm{mv},E}$ be an $ A_{\mathrm{mv},E}$-module, where the scalar multiplication is given by $(a,m)\mapsto \varphi_q(a)\cdot m$, $a\in A_{\mathrm{mv},E}$, $m\in M$. First we show that $M$ is a flat $A_{\mathrm{mv},E}$-module. The ring $A_{\mathrm{mv},E}$ is a $\varpi$-adically complete Noetherian domain by Lemma \ref{A_{mathrm{m.v.},K} is noetherian} where $\varpi$ is a uniformizer of $\mathcal{O}_E$. Since $ A_q$ is a regular domain (\cite[Corollary 3.1.1.2]{breuil2023conjectures}), by Kunz's theorem (\cite{kunz1969characterizations}), the endomorphism $\varphi_q: A_q\to A_q$ is flat. As $A_{\mathrm{mv},E}/\varpi\cong A=\mathbb{F}\otimes_{\sigma_{0},\mathbb{F}_q}A_q$ and $\varphi_q: A\to A$ is the base change of $\varphi_q: A_q\to A_q$, we see that $\varphi_q: A_{\mathrm{mv},E}/\varpi\to A_{\mathrm{mv},E}/\varpi$ is flat, i.e. $M/\varpi$ is a flat $A_{\mathrm{mv},E}/\varpi$-module. Moreover, applying the functor $M\otimes_{A_{\mathrm{mv},E}}-$ to the short exact sequence
		\[0\to A_{\mathrm{mv},E}\xrightarrow{\cdot\varpi} A_{\mathrm{mv},E}\to A_{\mathrm{mv},E}/\varpi\to 0,\]
		we get a long exact sequence
		\[\cdots\to\mathrm{Tor}_1^{A_{\mathrm{mv},E}}(M,A_{\mathrm{mv},E})\to \mathrm{Tor}_1^{A_{\mathrm{mv},E}}(M,A_{\mathrm{mv},E}/\varpi)\to M\xrightarrow{\cdot \varpi}M\to M/\varpi\to 0.\]
		Since $\mathrm{Tor}_1^{A_{\mathrm{mv},E}}(M,A_{\mathrm{mv},E})=0$ and $M\xrightarrow{\cdot \varpi}M$ is injective,  $\mathrm{Tor}_1^{A_{\mathrm{mv},E}}(M,A_{\mathrm{mv},E}/\varpi)=0$. Then applying \cite[\href{https://stacks.math.columbia.edu/tag/0AGW}{Tag 0AGW}]{stacks-project}, we deduce that $M$ is a flat $A_{\mathrm{mv},E}$-module, i.e. $\varphi_q: A_{\mathrm{mv},E}\to A_{\mathrm{mv},E}$ is flat. 
		
		To show that $\varphi_q: A_{\mathrm{mv},E}\to A_{\mathrm{mv},E}$ is faithfully flat, by \cite[\href{https://stacks.math.columbia.edu/tag/00HQ}{Tag 00HQ}]{stacks-project}, it suffices to show that every closed point of $\mathrm{Spec}A_{\mathrm{mv},E}$ lies in the image of the morphism
		$\varphi_q: \mathrm{Spec}A_{\mathrm{mv},E}\to \mathrm{Spec}A_{\mathrm{mv},E}$. Since $\varpi$ lies in the Jacobson radical of $A_{\mathrm{mv},E}$, and the Jacobson radical is the intersection of all maximal ideals, it suffices to show the surjectivity of the morphism
		\[\varphi_q: \mathrm{Spec}A_{\mathrm{mv},E}/\varpi\to \mathrm{Spec}A_{\mathrm{mv},E}/\varpi.\]
		As $A_{\mathrm{mv},E}/\varpi\cong A$, and $\varphi_q^{f'}: A\to A$ is the Frobenius map $x\mapsto x^{q'}$, where $\#\mathbb{F}=q'=q^{f'}$, we see that $\varphi_q^{f'}: \mathrm{Spec}A\to \mathrm{Spec}A$ is surjective. Therefore, $\varphi_q: A_{\mathrm{mv},E}\to A_{\mathrm{mv},E}$ is faithfully flat.
		
		It remains to show that $\varphi_q$ is finitely presented. 
		We put 
		\[S\coloneq \left\{\underline{n}=(n_0,\dots,n_{f-1})\in\mathbb{Z}^f: 0\leq n_0,\dots,n_{f-1}\leq q-1\right\},\]
		and $T_{\underline{n}}\coloneq T_{\sigma_{0}}^{n_0}\prod_{i=1}^{f-1}\left(\dfrac{T_{\sigma_{i}}}{T_{\sigma_{0}}}\right)^{n_i}$ for $\underline{n}= (n_0,\dots,n_{f-1})\in S$. By (\ref{action of phi_q on A}), we have
		\[A_{\mathrm{mv},E}=\sum_{\underline{n}\in S}\varphi_q(A_{\mathrm{mv},E})T_{\underline{n}}+\varpi A_{\mathrm{mv},E}.\]
		Since $\varpi$ is contained in the Jacobson radical of $A_{\mathrm{mv},E}$, by Nakayama's lemma, $A_{\mathrm{mv},E}=\sum\limits_{\underline{n}\in S}\varphi_q(A_{\mathrm{mv},E})T_{\underline{n}}$, hence $\varphi_q: A_{\mathrm{mv},E}\to A_{\mathrm{mv},E}$ is of finite type. Since $A_{\mathrm{mv},E}$ is a Noetherian domain by Lemma \ref{A_{mathrm{m.v.},K} is noetherian}, we deduce that $\varphi_q$ is finitely presented.
	\end{proof}

	For $n\geq 1$, $A_{\mathrm{mv},E_0}/p^n$ is the set of elements of the form
	\begin{align*}
		\sum_{\underline{n}\in\mathbb{Z}^{f-1}}f_{\underline{n}}(T_{\sigma_{0}})\prod_{i=1}^{f-1}\left(\frac{T_{\sigma_{i}}}{T_{\sigma_{0}}}\right)^{n_i}
	\end{align*}
	with $f_{\underline{n}}(T_{\sigma_{0}})\in (W(\mathbb{F})/p^n) [\negthinspace[T_{\sigma_{0}}]\negthinspace]\left[\frac{1}{T_{\sigma_{0}}}\right]$ such that for any $m\geq 0$, $f_{\underline{n}}(T_{\sigma_{0}})\in T_{\sigma_{0}}^m\cdot (W(\mathbb{F})/p^n)[\negthinspace[T_{\sigma_{0}}]\negthinspace]$ for almost all $\underline{n}$. Thus the map
	\begin{align*}
		\mathrm{ord}_{T_{\sigma_{0}}}: (W(\mathbb{F})/p^n) [\negthinspace[T_{\sigma_{0}}]\negthinspace]\left[\frac{1}{T_{\sigma_{0}}}\right]&\to \mathbb{Z}\cup \{+\infty\}\\
		f(T_{\sigma_{0}})&\mapsto \sup\left\{m\in\mathbb{Z}: f(T_{\sigma_{0}})\in T_{\sigma_{0}}^m (W(\mathbb{F})/p^n)[\negthinspace[T_{\sigma_{0}}]\negthinspace] \right\}
	\end{align*}
	induces a map
	\begin{align*}
		v_{n}:\quad A_{\mathrm{mv},E_0}\qquad \longtwoheadrightarrow\qquad A_{\mathrm{mv},E_0}/p^n \qquad & \longrightarrow \mathbb{Z}\cup \{+\infty\}\\
		\sum_{\underline{n}\in\mathbb{Z}^{f-1}}f_{\underline{n}}(T_{\sigma_{0}})\prod_{i=1}^{f-1}\left(\frac{T_{\sigma_{i}}}{T_{\sigma_{0}}}\right)^{n_i} & \longmapsto \inf_{\underline{n}\in\mathbb{Z}^{f-1}}\mathrm{ord}_{T_{\sigma_{0}}}\left(f_{\underline{n}}(T_{\sigma_{0}})\right).
	\end{align*}
	Let $P,Q$ be elements of $A_{\mathrm{mv},E_0}$, $m\in\mathbb{Z}$, $n\geq 1$, it is easy to check that $v_n$ satisfies the following properties:
	\begin{enumerate}
		\item[(i)] 
		$v_{n}(P)=+\infty$ if and only if $P\in p^nA_{\mathrm{mv},E_0}$,
		\item[(ii)] 
		$v_n(P+Q)\geq \min(v_n(P),v_n(Q))$,
		\item[(iii)] 
		$v_n(P\cdot Q)\geq v_n(P)+v_{n}(Q)$ (it is an inequality instead of an equality, because $W(\mathbb{F})/p^n$ is not an integral domain if $n\geq 2$),
		\item[(iv)] 
		$v_n(P)\geq v_{n+1}(P)$,
		\item[(v)] 
		$v_{n}(T_{\sigma_i}^mP)=m+v_n(P)$, $v_{n+m}(p^mP)=v_{n}(P)$.
	\end{enumerate}
	
	We endow the ring $A_{\mathrm{mv},E_0}$ with the initial topology with respect to $\{v_n: n\geq 1\}$, i.e. the coarsest topology on $A_{\mathrm{mv},E_0}$ such that the map $A_{\mathrm{mv},E_0}\to \mathbb{R}_{\geq 0}, P\mapsto p^{-v_n(P)}$ is continuous for every $n\geq 1$. In other words, let
	\[U_{m,n}\coloneq \{P\in A_{\mathrm{mv},E_0}: v_m(P)\geq n\},\quad m\geq 1, n\in\mathbb{Z},\]
	then $\{U_{m,n}:m\geq 1, n\in\mathbb{Z}\}$ is an open neighbourhood basis of $0\in A_{\mathrm{mv},E_0}$. 
	
	Since $\{U_{m,n}:m\geq 1, n\in\mathbb{Z}\}$ are abelian subgroups of $A_{\mathrm{mv},E_0}$, we see that $A_{\mathrm{mv},E_0}$ is an additive topological group, i.e. the map $A_{\mathrm{mv},E_0}\times A_{\mathrm{mv},E_0}\to A_{\mathrm{mv},E_0}$, $(P,Q)\mapsto P-Q$ is continuous. 
	
	For $n\in\mathbb{Z}$, we put $F_n\coloneq\left\{P\in  A_{\mathrm{mv},E_0}: v_m(P)\geq n, \forall\ m\geq 1\right\}$, then $U_{m,n}=p^m A_{\mathrm{mv},E_0}+F_n.$
	For $0\leq i\leq f-1$, by (\ref{phi_q,O_K^{*} action on A_{mv,E_0}}), we can write
	\begin{align}\label{varphi_q(T_{sigma_{i}})}
		\varphi_q(T_{\sigma_{i}})=T_{\sigma_{i}}^q+pQ_{i}
	\end{align}
	for some $Q_i\in W(\mathbb{F})[\negthinspace[N_0]\negthinspace]$.
	In $A_{\mathrm{mv},E_0}$, we have
	\begin{equation}\label{computation 1}
		\begin{aligned}
			\varphi_q\left(\frac{T_{\sigma_{i}}}{T_{\sigma_{0}}}\right) &= \dfrac{T_{\sigma_{i}}^q+pQ_i}{T_{\sigma_{0}}^q+pQ_0} = \dfrac{T_{\sigma_{i}}^q+pQ_i}{T_{\sigma_{0}}^q} \cdot\dfrac{1}{1+pT_{\sigma_{0}}^{-q}Q_0}\\
			&=\left(\left(\frac{T_{\sigma_{i}}}{T_{\sigma_{0}}}\right)^q+pT_{\sigma_{0}}^{-q}Q_i\right)\cdot\sum_{n=0}^{\infty}(-1)^np^nT_{\sigma_{0}}^{-qn}Q_0^n\\
			&=\left(\frac{T_{\sigma_{i}}}{T_{\sigma_{0}}}\right)^q+\sum_{n=1}^{\infty}(-1)^n\left(\frac{p}{T_{\sigma_{0}}^q}\right)^nQ_0^{n-1}\left(\left(\frac{T_{\sigma_{i}}}{T_{\sigma_{0}}}\right)^qQ_0-Q_i\right).
		\end{aligned}
	\end{equation}
	Let $S_{0,i,n}\coloneq (-1)^nQ_0^{n-1}\left(\left(\frac{T_{\sigma_{i}}}{T_{\sigma_{0}}}\right)^qQ_0-Q_i\right)$ for $n\geq 1$, then $S_{0,i,n}\in F_0$, and
	\begin{align}
		\varphi_q\left(\frac{T_{\sigma_{i}}}{T_{\sigma_{0}}}\right) =\left(\frac{T_{\sigma_{i}}}{T_{\sigma_{0}}}\right)^q+\sum_{n=1}^{\infty} \left(\frac{p}{T_{\sigma_{0}}^q}\right)^nS_{0,i,n}.
	\end{align}
	Similarly, by an easy computation, we can show that for any $\underline{n}=(n_1,\dots,n_{f-1})\in\mathbb{Z}^{f-1}$, there exists $\{S_{\underline{n},m}\in F_0\}_{m\geq 1}$ such that
	\begin{align}\label{computation 2}
		\varphi_q\left(\prod_{i=1}^{f-1}\left(\frac{T_{\sigma_{i}}}{T_{\sigma_{0}}}\right)^{n_i}\right)=\prod_{i=1}^{f-1}\left(\frac{T_{\sigma_{i}}}{T_{\sigma_{0}}}\right)^{qn_i}+\sum_{m=1}^{\infty} \left(\frac{p}{T_{\sigma_{0}}^q}\right)^m S_{\underline{n},m}.
	\end{align}
	For any $P\in A_{\mathrm{mv},E_0}$ and $n\geq 1$, we write
	\[P\equiv \sum_{\underline{n}\in\mathbb{Z}^{f-1}}\left(\sum_{k\in\mathbb{Z}}a_{\underline{n},k} T_{\sigma_{0}}^k\right)\prod_{i=1}^{f-1}\left(\frac{T_{\sigma_{i}}}{T_{\sigma_{0}}}\right)^{n_i} \mod p^n,\ a_{\underline{n},k}\in W(\mathbb{F})/p^n,\]
	with $a_{\underline{n},k}=0$ for $k\ll 0$, then $v_n(P)=\inf\{k: a_{\underline{n},k}\neq 0\ \text{for some}\ \underline{n}\}$, and
	\begin{align*}
		\varphi_q(P)&\equiv \sum_{\underline{n}\in\mathbb{Z}^{f-1},k\in\mathbb{Z}}a_{\underline{n},k} \varphi_q(T_{\sigma_{0}})^k\varphi_q\left(\prod_{i=1}\left(\frac{T_{\sigma_{i}}}{T_{\sigma_{0}}}\right)^{n_i}\right) \mod p^n
	\end{align*}
	Using (\ref{varphi_q(T_{sigma_{i}})})  and (\ref{computation 2}), we get
	\begin{align*}
		\varphi_q(T_{\sigma_{0}})^k\varphi_q\left(\prod_{i=1}^{f-1}\left(\frac{T_{\sigma_{i}}}{T_{\sigma_{0}}}\right)^{n_i}\right)&=\left(T_{\sigma_{0}}^q+pQ_0\right)^k\cdot \left( \prod_{i=1}^{f-1}\left(\frac{T_{\sigma_{i}}}{T_{\sigma_{0}}}\right)^{qn_i}+\sum_{m=1}^{\infty} \left(\frac{p}{T_{\sigma_{0}}^q}\right)^m\cdot S_{\underline{n},m}\right)\\
		&=T_{\sigma_{0}}^{qk}\left(1+\dfrac{p}{T_{\sigma_{0}}^q} Q_0\right)^k \cdot \left( \prod_{i=1}^{f-1}\left(\frac{T_{\sigma_{i}}}{T_{\sigma_{0}}}\right)^{qn_i}+\sum_{m=1}^{\infty} \left(\frac{p}{T_{\sigma_{0}}^q}\right)^m\cdot S_{\underline{n},m}\right)\\
		&=T_{\sigma_{0}}^{qk} \left(\sum_{m=0}^{k}\binom{k}{m}\left(\dfrac{p}{T_{\sigma_{0}}^q}\right)^mQ_0^m\right)\cdot \left( \prod_{i=1}^{f-1}\left(\frac{T_{\sigma_{i}}}{T_{\sigma_{0}}}\right)^{qn_i}+\sum_{m=1}^{\infty} \left(\frac{p}{T_{\sigma_{0}}^q}\right)^m\cdot S_{\underline{n},m}\right)\\
		&=T_{\sigma_{0}}^{qk}\cdot\sum_{m=0}^{\infty}\left(\frac{p}{T_{\sigma_{0}}^q}\right)^m\cdot S_{\underline{n},m}'
	\end{align*}
	for some $S_{\underline{n},m}'\in F_0$, $m\geq 0$, which implies that for $n\geq 1$,
	\[v_n\left(\varphi_q(T_{\sigma_{0}})^k\varphi_q\left(\prod_{i=0}^{f-1}\left(\frac{T_{\sigma_{i}}}{T_{\sigma_{0}}}\right)^{n_i}\right)\right)\geq qk-q(n-1),\]
	hence
	\[v_{n}(\varphi_q(P))\geq q v_{n}(P)-q(n-1),\quad P\in A_{\mathrm{mv},E_0}.\]
	By induction on $m\geq 1$, we have
	\begin{align}\label{phi_q and the semi-valuation}
		v_n(\varphi_q^m(P))\geq q^m v_{n}(P)-(n-1)\cdot\dfrac{q^{m+1}-q}{q-1}, \quad P\in A_{\mathrm{mv},E_0}, n\geq 1.
	\end{align}
	Moreover, since $v_1$ is induced by the Gauss norm on the Tate algebra $A_{\mathrm{mv},E_0}/p\cong A$ which is multiplicative (\cite[$\S$2.2 p.13]{bosch2014lectures}) and $\varphi_q^{f'}$ induces the $q'$-th power map on $A_{\mathrm{mv},E_0}/p\cong A$, we get
	\begin{align}\label{phi_q and the semi-valuation v_1}
		v_1(\varphi_q^m(P))=q^mv_1(P),\quad m\geq 1, P\in A_{\mathrm{mv},E_0}.
	\end{align}
	
	\begin{cor}\label{inverse image of a multiple of p under phi_q is a multiple of p}
		Let $P$ be an element of $A_{\mathrm{mv},E_0}$ such that $\varphi_q(P)\in p^nA_{\mathrm{mv},E_0}$, $n\geq 1$. Then $P\in p^nA_{\mathrm{mv},E_0}$.
	\end{cor}
	\begin{proof}
		It is a consequence of (\ref{phi_q and the semi-valuation v_1}) and the property (i), (v). Assume that $P=p^{n'}P'$ for some $P'\notin pA_{\mathrm{mv},E_0}$ and $n'+1\leq n$, then $v_1(P')$ is finite, hence
		\[v_{n'+1}(\varphi_q(P))=v_{n'+1}(p^{n'}\varphi_q(P'))=v_{1}(\varphi_q(P'))=q\cdot v_{1}(P')\]
		is also finite, while $n'+1\leq n$ and $\varphi_q(P)\in p^nA_{\mathrm{mv},E_0}$ implies that $v_{n'+1}(\varphi_q(P))\geq v_{n}(\varphi_q(P))=+\infty$ which leads to a contradiction. Thus $P\in p^nA_{\mathrm{mv},E_0}$.
	\end{proof}

	Using (\ref{phi_q,O_K^{*} action on A_{mv,E_0}}), for any $a\in\mathcal{O}_K^{\times}$, by a computation similar to (\ref{computation 1}), we deduce that
	\begin{align}\label{v_n and the action of a}
		v_n(a(P))\geq v_n(P)-(n-1),\quad  P\in A_{\mathrm{mv},E_0}, n\geq 1.
	\end{align}
	
	\begin{lem}
		The ring $A_{\mathrm{mv},E_0}$ is a complete Hausdorff topological ring, and the endomorphism $\varphi_q$ and the action of $\mathcal{O}_K^{\times}$ are continuous.
	\end{lem}
	
	\begin{proof}
		The claim that $A_{\mathrm{mv},E_0}$ is a complete Hausdorff topological ring is easy to check, and the continuity of $\varphi_q$ and the $\mathcal{O}_K^{\times}$-action are direct consequences of (\ref{phi_q and the semi-valuation}) and (\ref{v_n and the action of a}), 
		hence we omit the details. 
	\end{proof}
	
	\section{Two $(\varphi_q,\mathcal{O}_K^{\times})$-equivariant injections from non-perfect coefficient rings to the ring of Witt vectors}\label{section 2}
	
	The goal of this section is to construct a $(\varphi_q,\mathcal{O}_K^{\times})$-equivariant injection $A_{\mathrm{mv},E}\hookrightarrow W_E(A_{\infty})$ (\ref{(varphi_q,Z_p^*)-equivariant injection from A_{mv,E} to W_E(A_{infty})}). 
	Similar construction also applies to $A_{\mathrm{LT},E,\sigma_{0}}$, which gives a $(\varphi_q,\mathcal{O}_K^{\times})$-equivariant injection $A_{\mathrm{LT},E,\sigma_{0}}\hookrightarrow W_E\left(\mathbb{F}(\negthinspace(T_{\mathrm{LT}}^{1/p^{\infty}})\negthinspace)\right)$ (\ref{equivairant injection Lubin-Tate case}).
	
	\subsection{A $(\varphi_q,\mathcal{O}_K^{\times})$-equivariant injection from $A_{\mathrm{mv},E_0}$ to $W_E(A_{\infty})$}\label{How to define a phi_q,O_K^{*} equivariant inclusion}
	
	
	Consider the inductive system $(A_{\mathrm{mv},E_0}^{(n)})_{n\geq 0}$, where $A_{\mathrm{mv},E_0}\to A_{\mathrm{mv},E_0}^{(n)}, T_{\sigma_i}\mapsto T_{\sigma_i}^{(n)}$ is an isomorphism for every $n\geq 0$, and the transition maps are given by $A_{\mathrm{mv},E_0}^{(n)}\hookrightarrow A_{\mathrm{mv},E_0}^{(n+1)}, T_{\sigma_i}^{(n)}\mapsto \varphi_q(T_{\sigma_i}^{(n+1)})$. We denote the colimit of the inductive system $(A_{\mathrm{mv},E_0}^{(n)})_{n\geq 0}$ by $\varinjlim_{\varphi_q} A_{\mathrm{mv},E_0}$. In this section, we use the theory of strict $p$-rings to show that with an appropriate metric topology, the completion of $\varinjlim_{\varphi_q} A_{\mathrm{mv},E_0}$ is $(\varphi_q,\mathcal{O}_K^{\times})$-equivariantly isomorphic to $W(A_{\infty})$, then there is a $(\varphi_q,\mathcal{O}_K^{\times})$-equivariant injection $A_{\mathrm{mv},E_0}\hookrightarrow \varinjlim_{\varphi_q} A_{\mathrm{mv},E_0}\hookrightarrow W(A_{\infty})$. By tensoring with $\mathcal{O}_E$, we obtain the required $(\varphi_q,\mathcal{O}_K^{\times})$-equivariant injection $A_{\mathrm{mv},E}\hookrightarrow  W_E(A_{\infty})$.
	
	For $x\in \varinjlim_{\varphi_q}A_{\mathrm{mv},E_0}$, there exists $N\geq 0$ such that $P\in A_{\mathrm{mv},E_0}^{(N)}$ represents $x$. We put 
	\begin{align}\label{definition of w_n}
		w_n(x)\coloneq \liminf_{m\to +\infty}\dfrac{v_n(\varphi_q^m(P))}{q^{m+N}}\in\mathbb{R}\cup\{\pm\infty\}.
	\end{align}
	Then (\ref{phi_q and the semi-valuation}) implies that $w_n(x)\geq q^{-N}\left(v_n(P)-\frac{q(n-1)}{q-1}\right)$, hence is either finite or $+\infty$, and we can directly check that $w_n(x)$ does not depend on the choice of $N$ and $P$.
	\begin{lem}\label{semi-valuation on the colimit}
		The map $w_n: \varinjlim_{\varphi_q}A_{\mathrm{mv},E_0}\to\mathbb{R}\cup\{+\infty\}$ satisfies
		\begin{enumerate}
			\item [(i)]
			$w_n(x)=+\infty$ if and only if $x\in p^n \varinjlim_{\varphi_q}A_{\mathrm{mv},E_0}$,
			\item [(ii)]
			$w_n(x+x')\geq \min(w_n(x),w_n(x'))$,
			\item [(iii)]
			$w_{n}(x\cdot x')\geq w_n(x)+w_n(x')$,
			\item [(iv)]
			$w_n(x)\geq w_{n+1}(x)$,
			\item [(v)] 
			 $w_{n+m}(p^mx)=w_n(x)$,
			\item [(vi)]
			If $w_n(x)=+\infty$ for every $n\geq 1$, then $x=0$.
		\end{enumerate}
	\end{lem}
	\begin{proof}
		Since (ii), (iii), (iv) and (v) are direct consequences of the corresponding properties for $v_n$, we only need to check (i) and (vi). 
		
		If $x\in p^n \varinjlim_{\varphi_q}A_{\mathrm{mv},E_0}$, write $x=p^n P$ for some $P\in A_{\mathrm{mv},E_0}^{(N)}$, then $v_n(\varphi_q^m(p^nP))=+\infty$, for any $m\geq 1$, hence $w_n(x)=+\infty$. We prove the converse by induction on $n$. For $n=1$, by (\ref{phi_q and the semi-valuation v_1}), if $x$ is represented by $P\in A_{\mathrm{mv},E_0}^{(N)}$, then $w_1(x)=q^{-N}\cdot v_1(P)=+\infty$, hence $P\in pA_{\mathrm{mv},E_0}^{(N)}$, and $x\in p\varinjlim_{\varphi_q}A_{\mathrm{mv},E_0}$. For $n\geq 2$, using (iv) and the induction hypothesis, $x\in p^{n-1}\varinjlim_{\varphi_q}A_{\mathrm{mv},E_0}$. Write $x=p^{n-1}x'$ for some $x\in A_{\mathrm{mv},E_0}$, then (v) implies that $w_1(x')=w_n(p^{n-1}x')=+\infty$, thus $x\in p^n\varinjlim_{\varphi_q}A_{\mathrm{mv},E_0}$.
		
		If $w_n(x)=+\infty$ for every $n$, let $P\in A_{\mathrm{mv},E_0}^{(N)}$ be a representative of $x$, by $(i)$, $x\in p^n\varinjlim_{\varphi_q}A_{\mathrm{mv},E_0}$, hence for every $n\geq 1$, there exists $m\geq 1$ such that $\varphi_q^m(P) \in p^nA_{\mathrm{mv},E_0}$, hence $P\in p^nA_{\mathrm{mv},E_0}$ for every $n$ by Corollary \ref{inverse image of a multiple of p under phi_q is a multiple of p}, thus $P=0$.
	\end{proof}
	
	We endow the ring $\varinjlim_{\varphi_q}A_{\mathrm{mv},E_0}$ with the initial topology with respect to $\{w_n: n\geq 1\}$, i.e. the coarsest topology such that the map $\varinjlim_{\varphi_q}A_{\mathrm{mv},E_0}\to\mathbb{R}_{\geq 0},\ x\mapsto p^{-w_n(x)}$ is continuous for every $n\geq 1$. Then the ring $\varinjlim_{\varphi_q}A_{\mathrm{mv},E_0}$ is a Hausdorff topological ring by Lemma \ref{semi-valuation on the colimit}. 
	We define an invariant metric on $\varinjlim_{\varphi_q}A_{\mathrm{mv},E_0}$ by
	\[d(x,y)\coloneq \sum_{n=1}^{+\infty}2^{-n}\frac{p^{-w_n(x-y)}}{1+p^{-w_n(x-y)}},\quad x,y\in \varinjlim_{\varphi_q}A_{\mathrm{mv},E_0},\]
	then the topology on $\varinjlim_{\varphi_q}A_{\mathrm{mv},E_0}$ coincides with the metric topology. Let $\left(\varinjlim_{\varphi_q}A_{\mathrm{mv},E_0}\right)^{\wedge}$ be the topological completion of $\varinjlim_{\varphi_q} A_{\mathrm{mv},E_0}$ with respect to the metric $d$, which is a complete metric space.

	\begin{lem}\label{limit of valuation on Cauchy sequences}
		For $x\in \left(\varinjlim_{\varphi_q}A_{\mathrm{mv},E_0}\right)^{\wedge}$, take a Cauchy sequence $(x_l)_{l\geq 1}$ of $\varinjlim_{\varphi_q}A_{\mathrm{mv},E_0}$ converging to $x$. Then for every $n\geq 1$, exactly one of the following holds
		\begin{enumerate}
			\item[(1)] 
			$\lim\limits_{l\to+\infty}w_n(x_l)=+\infty$,
			\item[(2)] 
			there exists $L\geq 1$ such that $w_n(x_{l})=w_{n}(x_{L})$ is finite for any $l\geq L$.
		\end{enumerate}
	\end{lem}
	\begin{proof}
		It is a direct consequence of the strong triangle inequality (Lemma \ref{semi-valuation on the colimit} (ii)), and we leave the details as an easy exercise to the reader.
	\end{proof}
	
	As a consequence, for $n\geq 1$, $w_n(x)\coloneq \lim\limits_{l\to+\infty}w_n(x_l)\in \mathbb{R}\cup\{+\infty\}$ does not depend on the choice of the Cauchy sequence $(x_l)_l$, hence is well-defined and extends the map $w_n$ on $\varinjlim_{\varphi_q}A_{\mathrm{mv},E_0}$. It is easy to check that $w_n$ also satisfies the properties (ii) to (v) in Lemma \ref{semi-valuation on the colimit}. 
	
	For $x,y\in \left(\varinjlim_{\varphi_q}A_{\mathrm{mv},E_0}\right)^{\wedge}$, let $(x_l)_l$, $(y_l)_l$ be Cauchy sequences in $\varinjlim_{\varphi_q}A_{\mathrm{mv},E_0}$ which converge to $x$ and $y$ respectively, then by Lemma \ref{limit of valuation on Cauchy sequences}, we have
	\[d(x,y)=\lim_{l\to+\infty}d(x_l,y_l)=\sum_{n=1}^{+\infty}2^{-n}\frac{\lim\limits_{l\to+\infty}p^{-w_n(x_l-y_l)}}{1+\lim\limits_{l\to+\infty}p^{-w_n(x_l-y_l)}}2^{-n}= \sum_{n=1}^{+\infty}2^{-n}\frac{p^{-w_n(x-y)}}{1+p^{-w_n(x-y)}},\]
	where the second equality holds since $\left(\sum_{n=1}^k2^{-n}\frac{p^{-w_n(x_l-y_l)}}{1+p^{-w_n(x_l-y_l)}}\right)_{k\geq 1}$ converges to $d(x_l,y_l)$ uniformly.
	Thus the metric topology on $\left(\varinjlim_{\varphi_q}A_{\mathrm{mv},E_0}\right)^{\wedge}$ coincides with the initial topology with respect to $\{w_n: n\geq 1\}$. Moreover, let $(x_l)_l$, $(y_l)_l$ be Cauchy sequences in $\varinjlim_{\varphi_q}A_{\mathrm{mv},E_0}$ converging to $x$ and $y$ respectively, then $(x_l+y_l)$ and $(x_l\cdot y_l)$ are also Cauchy sequences, thus the addition and the multiplication on $\varinjlim_{\varphi_q}A_{\mathrm{mv},E_0}$ extends uniquely to $\left(\varinjlim_{\varphi_q}A_{\mathrm{mv},E_0}\right)^{\wedge}$, which makes $\left(\varinjlim_{\varphi_q}A_{\mathrm{mv},E_0}\right)^{\wedge}$ into a complete Hausdorff metrizable topological ring. 
	
	For $x\in \varinjlim_{\varphi_q}A_{\mathrm{mv},E_0}$ and $n\geq 1$, by (\ref{definition of w_n}) and (\ref{v_n and the action of a}), we have
	\begin{align*}
		\begin{cases}
			w_n(\varphi_q(x))=qw_n(x),&\\
			w_n(a(x))=w_n(x), & a\in\mathcal{O}_K^{\times}.
		\end{cases}
	\end{align*}
	This implies that $\varphi_q$ extends to a continuous endomorphism of $\left(\varinjlim_{\varphi_q}A_{\mathrm{mv},E_0}\right)^{\wedge}$ and the action of $\mathcal{O}_K^{\times}$ on $\varinjlim_{\varphi_q}A_{\mathrm{mv},E_0}$ extends to a continuous $\mathcal{O}_K^{\times}$-action on $\left(\varinjlim_{\varphi_q}A_{\mathrm{mv},E_0}\right)^{\wedge}$.
	
	
	
	Since colimits commute with quotients, there is a natural continuous map
	\[\varinjlim_{\varphi_q}A_{\mathrm{mv},E_0}\twoheadrightarrow\left.\left(\varinjlim_{\varphi_q}A_{\mathrm{mv},E_0}\right)\middle/ p\right. \xrightarrow{\sim}\varinjlim_{\varphi_q}\left(A_{\mathrm{mv},E_0}/p\right)\xrightarrow[\sim]{T_{\sigma_{i}}^{(0)}\mapsto T_{\sigma_{i}}}\varinjlim_{\varphi_q}A,\]
	where the topology on $\left.\left(\varinjlim_{\varphi_q}A_{\mathrm{mv},E_0}\right)\middle/ p\right.$ is the quotient topology, the topology on $\varinjlim_{\varphi_q}A$ is the $T_{\sigma_{0}}$-adic topology, the topology on $\varinjlim_{\varphi_q}\left(A_{\mathrm{mv},E_0}/p\right)$ is induced by the last isomorphism, and the second map is an isomorphism of topological rings. Note that $w_1$ induces a valuation on $\left.\left(\varinjlim_{\varphi_q}A_{\mathrm{mv},E_0}\right)\middle/ p\right.$ which defines the quotient topology, thus the quotient $\left.\left(\varinjlim_{\varphi_q}A_{\mathrm{mv},E_0}\right)\middle/ p\right.$ is also metrizable.
	Taking completions, we get a continuous map
	\[\left(\varinjlim_{\varphi_q}A_{\mathrm{mv},E_0}\right)^{\wedge}\xrightarrow{g} \left(\left.\left(\varinjlim_{\varphi_q}A_{\mathrm{mv},E_0}\right)\middle/ p\right.\right)^{\wedge} \xrightarrow{\sim}A_{\infty},\]
	where $A_{\infty}\coloneq\mathbb{F}(\negthinspace(T_{\sigma_{0}}^{p^{-\infty}})\negthinspace)\left\langle\left(\frac{T_{\sigma_i}}{T_{\sigma_{0}}}\right)^{\pm p^{-\infty}}: 1\leq i\leq f-1\right\rangle\cong\left(\varinjlim_{\varphi_q}A\right)^{\wedge}$ is a perfectoid $\mathbb{F}$-algebra, the second map is an isomorphism of topological rings. Since $g$ sends $p$ to $0$, we get a continuous map
	\begin{align}\label{residue rings}
		\left.\left(\varinjlim_{\varphi_q}A_{\mathrm{mv},E_0}\right)^{\wedge}\middle/p\right.\xrightarrow{\overline{g}} \left(\left.\left(\varinjlim_{\varphi_q}A_{\mathrm{mv},E_0}\right)\middle/ p\right.\right)^{\wedge} \xrightarrow{\sim} A_{\infty},
	\end{align}
	where the topology on $\left.\left(\varinjlim_{\varphi_q}A_{\mathrm{mv},E_0}\right)^{\wedge}\middle/p\right.$ is the quotient topology. 
	
	\begin{lem}\label{Description of the colimit modulo p}
		The map $\overline{g}$ is an isomorphism of topological rings. 
	\end{lem}
	
	\begin{proof}
		This is a consequence of \cite[Chapitre IX $\S$3.1 p.26 Corollaire]{bourbaki2007topologie2}. 
	\end{proof}
	
	\begin{lem}\label{p-adically complete}
		For $x\in \left(\varinjlim_{\varphi_q}A_{\mathrm{mv},E_0}\right)^{\wedge}$, $n\geq 1$, $w_n(x)=+\infty$ if and only if $x\in p^n\left(\varinjlim_{\varphi_q}A_{\mathrm{mv},E_0}\right)^{\wedge}$. As a consequence, $\left(\varinjlim_{\varphi_q}A_{\mathrm{mv},E_0}\right)^{\wedge}$ is Hausdorff and complete for the $p$-adic topology.
	\end{lem}
	\begin{proof}
		If $x\in p^n\left(\varinjlim_{\varphi_q}A_{\mathrm{mv},E_0}\right)^{\wedge}$, say $x=p^ny$, $y\in \left(\varinjlim_{\varphi_q}A_{\mathrm{mv},E_0}\right)^{\wedge}$, then we can take a Cauchy sequence $(y_l)_l$ in $\varinjlim_{\varphi_q} A_{\mathrm{mv},E_0}$ which converges to $y$, then $(p^ny_l)_l$ is a Cauchy sequence in $\varinjlim_{\varphi_q} A_{\mathrm{mv},E_0}$ which converges to $x$. Hence $w_n(x)=\lim\limits_{l\to+\infty}w_n(p^ny_l)=+\infty$.
		
		We prove the converse by induction on $n$. Suppose that $w_1(x)=+\infty$, where $x\in \left(\varinjlim_{\varphi_q} A_{\mathrm{mv},E_0}\right)^{\wedge}$. Pick a Cauchy sequence $(x_l)_{l}$ of $\varinjlim_{\varphi_q}A_{\mathrm{mv},E_0}$ converging to $x$, thus $\lim\limits_{l\to+\infty}w_1(x_l)=+\infty$. The proof of Lemma \ref{Description of the colimit modulo p} shows that $(x_l\mod p)_{l}$ is also a Cauchy sequence in $\left.\left(\varinjlim_{\varphi_q}A_{\mathrm{mv},E_0}\right)\middle/ p\right. $ with respect to the quotient topology, which converges to $0$. Thus $g(x)=0$, and by Lemma \ref{Description of the colimit modulo p}, $x\in p\left(\varinjlim_{\varphi_q} A_{\mathrm{mv},E_0}\right)^{\wedge}$. For $n\geq 2$, we have $w_{n-1}(x)\geq w_{n}(x)=+\infty$, thus by the induction hypothesis, $x\in p^{n-1}\left(\varinjlim_{\varphi_q} A_{\mathrm{mv},E_0}\right)^{\wedge}$, say $x=p^{n-1}x'$, then $w_1(x')=w_{n}(x)=+\infty$, thus $x'\in p\left(\varinjlim_{\varphi_q} A_{\mathrm{mv},E_0}\right)^{\wedge}$, hence $x\in p^n\left(\varinjlim_{\varphi_q} A_{\mathrm{mv},E_0}\right)^{\wedge}$.
		
		As a consequence, the metric topology on $\left(\varinjlim_{\varphi_q}A_{\mathrm{mv},E_0}\right)^{\wedge}$ is coarser than the $p$-adic topology. Then the Hausdorffness for the $p$-adic topology is automatic. The completeness for the $p$-adic topology is easy to check, and we omit the details.
	\end{proof}

	Recall that there is a continuous $\mathbb{F}$-linear $(\varphi_q,\mathcal{O}_K^{\times})$-action on $A_{\infty}$ (\cite[$\S$2.2 (17), Lemma 2.4.2, $\S$2.6 (50)]{breuil2023multivariable}) satisfying (\ref{action of phi_q on A}). 
	The properties of the ring of Witt vectors imply that we can uniquely lift the $(\varphi_q,\mathcal{O}_K^{\times})$-action of $A_{\infty}$ to a $(\varphi_q,\mathcal{O}_K^{\times})$-action of $W(A_{\infty})$ (\cite[Chapter II $\S$5 Proposition 10]{serre1979local}):
	\begin{align}
		\begin{cases}
			\varphi_q\left(\sum_{i=0}^{\infty}[x_i]p^i\right)=\sum_{i=0}^{\infty}[\varphi_q(x_i)]p^i, & x_i\in A_{\infty},\\
			a\left(\sum_{i=0}^{\infty}[x_i]p^i\right)=\sum_{i=0}^{\infty}[a(x_i)]p^i, & x_i\in A_{\infty}, a\in\mathcal{O}_K^{\times}
		\end{cases}
	\end{align}
	where $[x_i]\in W(A_{\infty})$ is the Teichmüller lift of $x_i\in A_{\infty}$.
	We equip $W(A_{\infty})$ with the \textit{weak topology} (\cite[$\S$1.4.3]{fargues2018courbes}) as follows: recall that there is a bijection of sets
	\begin{align*}
		W(A_{\infty})\xrightarrow{\sim} A_{\infty}^{\mathbb{N}},\
		\sum_{i=0}^{\infty}[x_i]p^n  \mapsto (x_i)_{i\geq 0},
	\end{align*}
	then the weak topology is induced by the product topology on $A_{\infty}^{\mathbb{N}}$. For $0\leq i\leq f-1$, let $[T_{\sigma_i}]\in W(A_{\infty})$ be the Teichmüller lift of $T_{\sigma_i}$. The following lemma is a standard result in the literature.
	
	\begin{lem}\label{equivalent description of weak topology}
		The following topologies on $W(A_{\infty})$ coincide with each other:
		\begin{enumerate}
			\item [(i)] the weak topology.
			\item [(ii)] the unique linear topology such that $\{V_{m,n}\coloneq  p^mW(A_{\infty})+[T_{\sigma_{0}}]^n W(A_{\infty}^{\circ}): m,n\geq 0\}$ forms a fundamental system of neighbourhoods of $0$.
			\item [(iii)] the coarsest topology such that $N_n:W(A_{\infty})\to \mathbb{R}_{\geq 0}$ is continuous for every $n\geq 0$ , where $N_n(\sum_{i=0}[x_i]p^i)\coloneq \sup_{0\leq i\leq n} |x_i|$ and $|\cdot|$ is a fixed norm of $A_{\infty}$ which defines the topology on $A_{\infty}$ (\cite[Lemma 2.4.2. (iii)]{breuil2023multivariable}).
		\end{enumerate}
		Moreover, $W(A_{\infty})$ is a complete Hausdorff topological ring with respect to the weak topology, and the $(\varphi_q,\mathcal{O}_K^{\times})$-action on $W(A_{\infty})$ is continuous.
	\end{lem}
	\begin{proof}
		See for example \cite[Proposition 1.4.11.]{fargues2018courbes} or \cite[$\S$16]{berger2010galois}
	\end{proof}
	
	\begin{prop}\label{B_E version: completion of colimit}
		There is an isomorphism of $W(\mathbb{F})$-algebras
		\[h:\left(\varinjlim_{\varphi_q}A_{\mathrm{mv},E_0}\right)^{\wedge}\xrightarrow{\sim}W(A_{\infty})\]
		which commutes with the continuous $(\varphi_q,\mathcal{O}_K^{\times})$-actions.
	\end{prop}
	\begin{proof}
		Recall that a \textit{strict} $p$-\textit{ring} $R$ is a $p$-torsion-free ring such that $R\cong \varprojlim_n R/p^n$ and $R/p$ is a perfect $\mathbb{F}_p$-algebra (\cite[Chapter II $\S$5]{serre1979local}). For instance, $W(A_{\infty})$ is a strict $p$-ring by definition. Moreover, $\left(\varinjlim_{\varphi_q}A_{\mathrm{mv},E_0}\right)^{\wedge}$ is also a strict $p$-ring: it is $p$-adically complete by Lemma \ref{p-adically complete}, and by Lemma \ref{Description of the colimit modulo p}, the residue ring is isomorphic to $A_{\infty}$ which is perfect. By \cite[Chapter II $\S$5 Proposition 10]{serre1979local}, there is a unique homomorphism of strict $p$-rings
		\[h:\left(\varinjlim_{\varphi_q}A_{\mathrm{mv},E_0}\right)^{\wedge}\to W(A_{\infty})\]
		such that $\overline{h}: \left(\varinjlim_{\varphi_q}A_{\mathrm{mv},E_0}\right)^{\wedge}\Big/p\to W(A_{\infty})/p\cong A_{\infty}$ is (\ref{residue rings}). As (\ref{residue rings}) is an isomorphism, by \cite[Chapter II $\S$5 Proposition 10]{serre1979local}, $h$ is also an isomorphism. Moreover, the map (\ref{residue rings}) commutes with the continuous $(\varphi_q,\mathcal{O}_K^{\times})$-actions, hence so does $h$.
	\end{proof}
	
	\begin{remark}
		We can directly construct a homomorphism $h^{-1}:W_E(A_{\infty})\to \left(\varinjlim_{\varphi_q} A_{\mathrm{mv},E_0}\right)^{\wedge}$ sending $[T_{\sigma_{i}}]\in W_E(A_{\infty})$ to the limit of the sequence $((T_{\sigma_{i}}^{(m)})^{q^m})_{m\geq 0}$ of $\left(\varinjlim_{\varphi_q} A_{\mathrm{mv},E_0}\right)^{\wedge}$. This construction gives a continuous $(\varphi_q,\mathcal{O}_K^{\times})$-equivariant isomorphism of rings, then we can define $h$ as the inverse of $h^{-1}$. In fact, $h^{-1}$ is an isomorphism of topological rings: it extends to a continuous bijective homomorphism of $E_0$-algebras $h^{-1}: W(A_{\infty})\left[\frac{1}{p}\right]\to \left(\varinjlim_{\varphi_q}A_{\mathrm{mv},E_0}\right)^{\wedge}\left[\frac{1}{p}\right]$. Since both $W(A_{\infty})[\frac{1}{p}]$ and $\left(\varinjlim_{\varphi_q}A_{\mathrm{mv},E_0}\right)^{\wedge}[\frac{1}{p}]$ are \textit{Fréchet spaces} over $E_0$ (\cite[Chapter I $\S$8 Definition]{schneider2013nonarchimedean}), we can apply Open Mapping Theorem (\cite[Chapter I $\S$8 Proposition 8.6]{schneider2013nonarchimedean}) to deduce that $h:\left(\varinjlim_{\varphi_q}A_{\mathrm{mv},E_0}\right)^{\wedge}\left[\frac{1}{p}\right]\to W(A_{\infty})\left[\frac{1}{p}\right]$ is continuous, hence $h: \left(\varinjlim_{\varphi_q}A_{\mathrm{mv},E_0}\right)^{\wedge}\to W(A_{\infty})$ is a continuous inverse of $h^{-1}$.
	\end{remark}
	
	As a consequence of Proposition \ref{B_E version: completion of colimit}, there is a $(\varphi_q,\mathcal{O}_K^{\times})$-equivariant injection 
	\[A_{\mathrm{mv},E_0}\xhookrightarrow{T_{\sigma_i}\mapsto T_{\sigma_i}^{(0)}} \varinjlim_{\varphi_q}A_{\mathrm{mv},E_0}\hookrightarrow \left(\varinjlim_{\varphi_q}A_{\mathrm{mv},E_0}\right)^{\wedge}\cong W(A_{\infty}).\]
	We define $A_{\mathrm{mv},E}\coloneq\mathcal{O}_E\otimes_{W(\mathbb{F})}A_{\mathrm{mv},E_0}$, $W_E(A_{\infty})\coloneq\mathcal{O}_E\otimes_{W(\mathbb{F})}W(A_{\infty})$. Viewing $A_{\mathrm{mv},E}$ as a finite free $A_{\mathrm{mv},E_0}$-module, we equip $A_{\mathrm{mv},E}$ with the product topology, and we equip $W_E(A_{\infty})$ with the weak topology. Then the continuous $(\varphi_q,\mathcal{O}_K^{\times})$-actions on $A_{\mathrm{mv},E_0}$ and $W(A_{\infty})$ extend $\mathcal{O}_E$-linearly to $A_{\mathrm{mv},E}$ and $W_E(A_{\infty})$. By tensoring with $\mathcal{O}_E$, we get a $(\varphi_q,\mathcal{O}_K^{\times})$-equivariant injection
	\begin{align}\label{(varphi_q,mathcal{O}_K^{*})-equivariant injection from A_{mv,E} to W_E(A_{infty})}
		A_{\mathrm{mv},E}\hookrightarrow  W_E(A_{\infty}).
	\end{align}
	
	Recall that there is a discrete valuation ring $A_{\mathrm{cycl},K}\coloneq\left(\mathcal{O}_K[\negthinspace[T_{\mathrm{cycl}}]\negthinspace]\left[\frac{1}{T_{\mathrm{cycl}}}\right]\right)^{\wedge}$ with residue field $\mathbb{F}_q(\negthinspace(T_{\mathrm{cycl}})\negthinspace)$, where the completion is the $p$-adic completion. We endow $A_{\mathrm{cycl},K}$ with the $(p,T_{\mathrm{cycl}})$-adic topology, i.e. an open neighbourhood basis of $0$ is given by $\left\{p^nA_{\mathrm{cycl},K}+T_{\mathrm{cycl}}^n\mathcal{O}_K[\negthinspace[T_{\mathrm{cycl}}]\negthinspace]: n\geq 1\right\}$.
	The ring $A_{\mathrm{cycl},K}$ admits a continuous $\mathcal{O}_K$-linear $(\varphi,\mathbb{Z}_p^{\times})$-action
	\begin{align}\label{cycltomic phi,Gmma action}
		\begin{cases}
			\varphi(T_{\mathrm{cycl}})=(1+T_{\mathrm{cycl}})^p-1, & \\
			a(T_{\mathrm{cycl}})=(1+T_{\mathrm{cycl}})^a-1, & a\in\mathbb{Z}_p^{\times}.
		\end{cases}
	\end{align} 
	Let $\varphi_q\coloneq \varphi^f$, $A_{\mathrm{cycl},E}\coloneq \mathcal{O}_E\otimes_{\sigma_{0},\mathcal{O}_K}A_{\mathrm{cycl},K}$, $\mathbb{F}(\negthinspace(T_{\mathrm{cycl}}^{1/p^{\infty}})\negthinspace)\coloneq \mathbb{F}\otimes_{\sigma_{0},\mathbb{F}_q}\mathbb{F}_q(\negthinspace(T_{\mathrm{cycl}}^{1/p^{\infty}})\negthinspace)$ and $W_E\left(\mathbb{F}(\negthinspace(T_{\mathrm{cycl}}^{1/p^{\infty}})\negthinspace)\right)\coloneq \mathcal{O}_E\otimes_{\sigma_{0},\mathcal{O}_K}W\left(\mathbb{F}_q(\negthinspace(T_{\mathrm{cycl}}^{1/p^{\infty}})\negthinspace)\right)$, where $\mathbb{F}_q(\negthinspace(T_{\mathrm{cycl}}^{1/p^{\infty}})\negthinspace)$ is the completed perfection of $\mathbb{F}_q(\negthinspace(T_{\mathrm{cycl}})\negthinspace)$. 
	
	When $K=\qp$, there is an isomorphism $A_{\mathrm{mv},E}\cong  A_{\mathrm{cycl},E}$ commuting with the continuous $(\varphi,\mathbb{Z}_p^{\times})$-actions. In this case, the map (\ref{(varphi_q,mathcal{O}_K^{*})-equivariant injection from A_{mv,E} to W_E(A_{infty})}) is a $(\varphi,\mathbb{Z}_p^{\times})$-equivariant injection
	\begin{align}\label{(varphi_q,Z_p^*)-equivariant injection from A_{mv,E} to W_E(A_{infty})}
		A_{\mathrm{cycl},E}\hookrightarrow W_E\left(\mathbb{F}(\negthinspace(T_{\mathrm{cycl}}^{1/p^{\infty}})\negthinspace)\right).
	\end{align}
	
	In the remain of this section, we prove the faithfully flatness of (\ref{(varphi_q,mathcal{O}_K^{*})-equivariant injection from A_{mv,E} to W_E(A_{infty})}).
	
	\begin{lem}\label{A to A_{infty} is faithfully flat}
		The ring homomorphism $A\hookrightarrow A_{\infty}$ is faithfully flat.
	\end{lem}
	\begin{proof}
		Since $A=A^{\circ}\left[\dfrac{1}{T_{\sigma_{0}}}\right]$, $A_{\infty}=A^{\circ}_{\infty}\left[\dfrac{1}{T_{\sigma_{0}}}\right]$, it suffices to show that $A^{\circ}\hookrightarrow A_{\infty}^{\circ}$ is faithfully flat. Similar to the proof of Lemma \ref{phi_q is flat}, as $A_{\infty}^{\circ}$ is $T_{\sigma_{0}}$-trosion-free, we have $\mathrm{Tor}_{1}^{A^{\circ}}(A_{\infty}^{\circ},A^{\circ}/(T_{\sigma_{0}}))=0$. Note that $A^{\circ}/(T_{\sigma_{0}})\cong \mathbb{F}\left[\left(\dfrac{T_{\sigma_{i}}}{T_{\sigma_{0}}}\right)^{\pm 1}: 1\leq i\leq f-1\right]$ is a Noetherian domain, and $A_{\infty}^{\circ}/(T_{\sigma_{0}})$ is a free $A^{\circ}/(T_{\sigma_{0}})$-module, by \cite[\href{https://stacks.math.columbia.edu/tag/0AGW}{Tag 0AGW}]{stacks-project}, the ring map $A^{\circ}\hookrightarrow A_{\infty}^{\circ}$ is flat.

		By \cite[\href{https://stacks.math.columbia.edu/tag/00HQ}{Tag 00HQ}]{stacks-project}, it remains to show that every closed point of $\mathrm{Spec}(A^{\circ})$ lies in the image of $\mathrm{Spec}(A_{\infty}^{\circ})\to \mathrm{Spec}(A^{\circ})$. Since $1+T_{\sigma_{0}}x$ is invertible for any $x\in A^{\circ}$, we see that $T_{\sigma_{0}}$ is contained in the Jacobson radical of $A^{\circ}$. Similarly, $T_{\sigma_{0}}$ is also contained in the Jacobson radical of $A^{\circ}_{\infty}$, thus it suffices to show that 
		$\mathrm{Spec}(A_{\infty}^{\circ}/(T_{\sigma_{0}}))\to \mathrm{Spec}(A^{\circ}/(T_{\sigma_{0}}))$ is surjective. 
		Note that $\big(A_{\infty}^{\circ}/(T_{\sigma_{0}})\big)/\mathrm{Nil}\big(A_{\infty}^{\circ}/(T_{\sigma_{0}})\big)$ is the perfection of $A^{\circ}/(T_{\sigma_{0}})$ where $\mathrm{Nil}\big(A_{\infty}^{\circ}/(T_{\sigma_{0}})\big)$ is the nilradical of $A_{\infty}^{\circ}/(T_{\sigma_{0}})$, hence the morphism $\mathrm{Spec}\left(\big(A_{\infty}^{\circ}/(T_{\sigma_{0}})\big)/\mathrm{Nil}\big(A_{\infty}^{\circ}/(T_{\sigma_{0}})\big)\right)\to \mathrm{Spec}(A^{\circ}/(T_{\sigma_{0}}))$ is a universal homeomorphism,
		then the surjectivity of the morphism $\mathrm{Spec}(A_{\infty}^{\circ}/(T_{\sigma_{0}}))\to \mathrm{Spec}(A^{\circ}/(T_{\sigma_{0}}))$ follows. Hence the ring homomorphism $A^{\circ}\hookrightarrow A_{\infty}^{\circ}$ is faithfully flat.
	\end{proof}
	
	\begin{prop}\label{A_{mv,E} to W_E(A_{infty}) is faithfully flat}
		The ring homomorphism $A_{\mathrm{mv},E}\hookrightarrow W_E(A_{\infty})$ is faithfully flat.
	\end{prop}
	\begin{proof}
		Note that $A_{\mathrm{mv},E}/\varpi\cong A$ is a Noetherian domain, and $W_E(A_{\infty})/\varpi\cong A_{\infty}$ is faithfully flat over $A$ by Lemma \ref{A to A_{infty} is faithfully flat}. Similar to the proof of Lemma \ref{phi_q is flat}, as $W_E(A_{\infty})$ is $\varpi$-trosion-free, we have $\mathrm{Tor}^{A_{\mathrm{mv},E}}_1(W_E(A_{\infty}),A)=0$. Hence by \cite[\href{https://stacks.math.columbia.edu/tag/0AGW}{Tag 0AGW}]{stacks-project}, the ring map $A_{\mathrm{mv},E}\hookrightarrow W_E(A_{\infty})$ is flat.
		
		By \cite[\href{https://stacks.math.columbia.edu/tag/00HQ}{Tag 00HQ}]{stacks-project}, it remains to show that every closed point of $\mathrm{Spec}A_{\mathrm{mv},E}$ lies in the image of $\mathrm{Spec}W_E(A_{\infty})\to \mathrm{Spec}A_{\mathrm{mv},E}$. As $\varpi$ is contained in the Jacobson radical of $A_{\mathrm{mv},E}$ and the Jacobson radical of $W_E(A_{\infty})$, it suffices to show the surjectivity of $\mathrm{Spec}(W_E(A_{\infty})/\varpi)\to \mathrm{Spec}(A_{\mathrm{mv},E}/\varpi)$, which is a direct consequence of Lemma \ref{A to A_{infty} is faithfully flat}.
	\end{proof}
	
	\subsection{A $(\varphi_q,\mathcal{O}_K^{\times})$ injection from $A_{\mathrm{LT},E,\sigma_{0}}$ to $W_E\left(\mathbb{F}(\negthinspace(T_{\mathrm{LT}}^{1/p^{\infty}})\negthinspace)\right)$}\label{injection in L-T case}
	
	We fix a Frobenius power series $p_{\mathrm{LT}}(T_{\mathrm{LT}})\in \mathcal{O}_K[\negthinspace[T_{\mathrm{LT}}]\negthinspace]$ associated to the uniformizer $p\in\mathcal{O}_K$, which gives a Lubin-Tate formal group $G_{\mathrm{LT}}=\mathrm{Spf}(\mathcal{O}_K[\negthinspace[T_{\mathrm{LT}}]\negthinspace])$ (\cite[Chapter 8, $\S$1]{lang2012cyclotomic}). For $a\in\mathcal{O}_K$, we denote the formal power series associated to $a$ by $a_{\mathrm{LT}}(T_{\mathrm{LT}})$. 
	
	Let $K_{\infty}$ be the Lubin-Tate extension of $K$. It is a totally ramified abelian extension of $K$ contained in $\overline{K}$, and $\gal(K_{\infty}/K)\cong\mathcal{O}_K^{\times}$.
	
	We consider the discrete valuation ring $A_{\mathrm{LT},K}\coloneq\left(\mathcal{O}_K[\negthinspace[T_{\mathrm{LT}}]\negthinspace][\frac{1}{T_{\mathrm{LT}}}]\right)^{\wedge}$ where the completion is the $p$-adic completion. The maximal ideal of $A_{\mathrm{LT},K}$ is generated by $p$, and the residue field is $\mathbb{F}_q(\negthinspace(T_{\mathrm{LT}})\negthinspace)$. We endow the ring $A_{\mathrm{LT},K}$ with the $(p,T_{\mathrm{LT}})$-adic topology, then $A_{\mathrm{LT},K}$ carries a continuous $\mathcal{O}_K$-linear $(\varphi_q,\mathcal{O}_K^{\times})$-action given by:
	\begin{align}\label{phi_q,O_k^* action ass. to Lubin-Tate formal groups}
		\begin{cases}
			\varphi_q(T_{\mathrm{LT}})= p_{\mathrm{LT}}(T_{\mathrm{LT}}), & \\
			a (T_{\mathrm{LT}})=a_{\mathrm{LT}}(T_{\mathrm{LT}}), & a\in \mathcal{O}_K^{\times}.
		\end{cases}
	\end{align}
	The $(\varphi_q,\mathcal{O}_K^{\times})$-action on $A_{\mathrm{LT},K}$ induces a continuous $(\varphi_q,\mathcal{O}_K^{\times})$-action on $\mathbb{F}_q(\negthinspace(T_{\mathrm{LT}})\negthinspace)$. We define $A_{\mathrm{LT},E_0,\sigma_{0}}\coloneq W(\mathbb{F})\otimes_{\sigma_0,\mathcal{O}_K}A_{\mathrm{LT},K}$, $A_{\mathrm{LT},E,\sigma_{0}}\coloneq \mathcal{O}_E\otimes_{\sigma_0,\mathcal{O}_K}A_{\mathrm{LT},K}$ and $\mathbb{F}(\negthinspace(T_{\mathrm{LT}})\negthinspace)\coloneq \mathbb{F}\otimes_{\sigma_{0},\mathbb{F}_q} \mathbb{F}_q(\negthinspace(T_{\mathrm{LT}})\negthinspace)$. The $(\varphi_q,\mathcal{O}_K^{\times})$-action on $A_{\mathrm{LT},K}$ extends $W(\mathbb{F})$-linearly (resp. $\mathcal{O}_E$-linearly) to $A_{\mathrm{LT},E_0,\sigma_{0}}$ (resp. $A_{\mathrm{LT},E,\sigma_{0}}$).
	
	Like for $A_{\mathrm{mv},E_0}$, we can define a metric on $\varinjlim_{\varphi_q}A_{\mathrm{LT},E_0,\sigma_{0}}$ so that its topological completion is $(\varphi_q,\mathcal{O}_K^{\times})$-equivariantly isomorphic to $W\left(\mathbb{F}(\negthinspace(T_{\mathrm{LT}}^{1/p^{\infty}})\negthinspace)\right)$, i.e. $\left(\varinjlim_{\varphi_q}A_{\mathrm{LT},E_0,\sigma_{0}}\right)^{\wedge}\cong W\left(\mathbb{F}(\negthinspace(T_{\mathrm{LT}}^{1/p^{\infty}})\negthinspace)\right)$. As a consequence, we obtain a $(\varphi_q,\mathcal{O}_K^{\times})$-equivariant injection
	\[A_{\mathrm{LT},E_0,\sigma_{0}}\to  \varinjlim_{\varphi_q}A_{\mathrm{LT},E_0,\sigma_{0}}\hookrightarrow W\left(\mathbb{F}(\negthinspace(T_{\mathrm{LT}}^{1/p^{\infty}})\negthinspace)\right),\]
	where $\mathbb{F}(\negthinspace(T_{\mathrm{LT}}^{1/p^{\infty}})\negthinspace)$ is the completed perfection of $\mathbb{F}(\negthinspace(T_{\mathrm{LT}})\negthinspace)$. We define $W_E\left(\mathbb{F}(\negthinspace(T_{\mathrm{LT}}^{1/p^{\infty}})\negthinspace)\right)\coloneq \mathcal{O}_E\otimes_{W(\mathbb{F})}W\left(\mathbb{F}(\negthinspace(T_{\mathrm{LT}}^{1/p^{\infty}})\negthinspace)\right)$. By tensoring with $\mathcal{O}_E$, we obtain a $(\varphi_q,\mathcal{O}_K^{\times})$-equivariant injection
	\begin{equation}\label{equivairant injection Lubin-Tate case}
		A_{\mathrm{LT},E,\sigma_{0}}\hookrightarrow W_E\left(\mathbb{F}(\negthinspace(T_{\mathrm{LT}}^{1/p^{\infty}})\negthinspace)\right).
	\end{equation}
	
	\section{Frobenius Descent for finitely presented étale $(\varphi_q,\mathcal{O}_K^{\times})$-modules over $W_E(A_{\infty})$}\label{section 3}
	
	In this section, we show that every finitely presented étale $(\varphi_q,\mathcal{O}_K^{\times})$-module over $W_E(A_{\infty})$ canonically descends to a finitely presented étale $(\varphi_q,\mathcal{O}_K^{\times})$-module over $A_{\mathrm{mv},E}$.
	
	\subsection{Finite projective étale $\varphi_q$-modules}
	
	In this section, we show that every finite projective étale $\varphi_q$-module over $W_E(A_{\infty})$ canonically descends to a finite projective étale $\varphi_q$-module over $A_{\mathrm{mv},E}$. We start with a lemma characterizing finite projective modules over $A_{\mathrm{mv},E}$ and $W_E(A_{\infty})$.
	
	\begin{lem}\label{description of finte projective modules}
		Let $R$ be one of the rings $\{A_{\mathrm{mv},E},A_{\mathrm{mv},E}/\varpi^n,W_E(A_{\infty}),W_E(A_{\infty})/\varpi^n\}$ where $n\geq 1$ is an integer. Any finite projective $R$-module is a free $R$-module.
	\end{lem}
	
	\begin{proof}
		Let $M_n$ be a finite projective $A_{\mathrm{mv},E}/\varpi^n$-module, locally free of rank $r$. For $n=1$, $A_{\mathrm{mv},E}/\varpi\cong A$, then by \cite[$\S$1 satz 3]{lutkebohmert1977vektorraumbundel}, $M_1$ is a free $A$-module of rank $r$. For $n\geq 2$, as $M_n/\varpi$ is a finite projective module over $A_{\mathrm{mv},E}/\varpi$ locally free of rank $r$, there exists $x_1,\dots,x_r\in M_n$ such that $\overline{x}_1,\dots,\overline{x}_r\in M_n/\varpi$ forms an $A_{\mathrm{mv},E}/\varpi$-basis of $M_n/\varpi$. By Nakayama's Lemma, $M_n=\sum_{i=1}^rA_{\mathrm{mv},E}/\varpi^nx_i$. Suppose that there exists $a_1,\dots,a_r\in A_{\mathrm{mv},E}/\varpi^n$ such that $\sum_{i=1}^ra_ix_i=0$. Since $M_n$ is locally free of rank $r$, for any prime ideal $\mathfrak{p}\subset A_{\mathrm{mv},E}/\varpi^n$, $(M_n)_{\mathfrak{p}}$ is a free $\left(A_{\mathrm{mv},E}/\varpi^n\right)_{\mathfrak{p}}$-module of rank $r$, and $x_1,\dots,x_r$ generate $(M_n)_{\mathfrak{p}}$, hence $x_1,\dots,x_r$ must be an $\left(A_{\mathrm{mv},E}/\varpi^n\right)_{\mathfrak{p}}$-basis of $(M_n)_{\mathfrak{p}}$, which implies that $a_1=\cdots=a_r=0$ in $\left(A_{\mathrm{mv},E}/\varpi^n\right)_{\mathfrak{p}}$. In other words, the localization of the $A_{\mathrm{mv},E}/\varpi^n$-module $a_iA_{\mathrm{mv},E}/\varpi^n$ at any prime ideal is $0$, which implies that $a_i=0$ for any $1\leq i\leq r$. Therefore, $M_n=\bigoplus_{i=1}^rA_{\mathrm{mv},E}/\varpi^nx_i$ is free.
		
		The conclusion for $A_{\infty}$ comes from \cite[Theorem 2.19]{DH21Projective}. For $W_E(A_{\infty})/\varpi^n$, $A_{\mathrm{mv},E}$ and $W_E(A_{\infty})$, the proof is similar, and we omit it.
	\end{proof}

	\begin{defn}
		Let $R$ be a ring with an endomorphism $\varphi$. Assume that $p$ is not invertible in $R$. We say that an $R$-module $M$ is a $\varphi$-\textit{module} over $R$ if there is an endomorphism $\varphi$ of $M$ such that, for any $r\in R$ and $m\in M$, $\varphi(rm)=\varphi(r)\varphi(m)$. For any $\varphi$-module $M$, there is a natural map, called the \textit{linearization map},
		\begin{equation}\label{linearization map}
			\begin{aligned}
				\mathrm{id}\otimes\varphi: R\otimes_{\varphi,R}M&\to M\\
				r\otimes m&\mapsto r\varphi(m).
			\end{aligned}
		\end{equation}
		We say that a $\varphi$-module $M$ is \textit{étale} if the linearization map (\ref{linearization map}) is an isomorphism. If, in addition, $R$ is a topological ring and $\varphi$ is a continuous endomorphism of $R$, we require $M$ to be a topological $R$-module and the endomorphism $\varphi$ of $M$ to be continuous.
	\end{defn}
	 
	 \begin{defn}
	 	Let $\Gamma$ be a topological group, and let $R$ be a topological ring with a continuous $(\varphi,\Gamma)$-action (Definition \ref{continuous (phi,Gamma)-action}). Assume that $p$ is not invertible in $R$. We say that an étale $\varphi$-module $M$ is an \textit{étale} $(\varphi,\Gamma)$-\textit{module} over $R$, if there is a continuous action of $\Gamma$ on $M$ commuting with the endomorphism $\varphi$ of $M$ such that
	 	\[\gamma(rm)=\gamma(r)\gamma(m), \quad \gamma\in \Gamma, r\in R,m\in M.\]
	 \end{defn}
	 
	For a finite module $M$ over a topological ring $R$, we endow $M$ with the quotient topology induced by some surjection $R^d\twoheadrightarrow M$, where $R^d$ is given the product topology. One can check that the topology on $M$ does not depend on the choice of the surjection $R^d\twoheadrightarrow M$.
	We denote by $\modet_{\varphi}(R)$ (resp. $\modetfp_{\varphi}(R)$, $\modet_{\varphi,\Gamma}(R)$, $\modetfp_{\varphi,\Gamma}(R)$) the category of finite projective étale $\varphi$-modules (resp. finitely presented étale $\varphi$-modules, finite projective étale $(\varphi,\Gamma)$-modules, finitely presented $(\varphi,\Gamma)$-modules) over $R$.

	For $n\geq 1$, let $W_{n}(\mathbb{F}_q)\coloneq W(\mathbb{F}_q)/p^n=\mathcal{O}_K/p^n$. Let $S$ be an affine $W_n(\mathbb{F}_q)$-scheme. Suppose that 
	\begin{align}\label{Katz's condition}
		\begin{cases}
				\text{(i) the affine scheme}\ S\ \text{is flat over}\ W_n(\mathbb{F}_q), \\
				\text{(ii) the special fibre}\ S\times_{W_n(\mathbb{F}_q)}\mathbb{F}_q\ \text{is normal and integral},\\
				\text{(iii) there is an endomorphism}\ \varphi_q\ \text{of}\ S\ \text{which induces the}\ q\text{-th power map on}\ S\times_{}\mathbb{F}_q
			\end{cases}
	\end{align}
	For example, $S\coloneq \mathrm{Spec}(A_{\mathrm{mv},K}/p^n)$, $\mathrm{Spec}(\varphi_q)$ satisfies (\ref{Katz's condition}): its special fibre is $\mathrm{Spec}A_q$, which is a normal integral scheme over $\mathbb{F}_q$ by \cite[Corollary 3.1.1.2]{breuil2023conjectures} and \cite[Theorem 19.4]{matsumura1989commutative}, and $S$ is flat over $W_n(\mathbb{F}_q)$ since $A_{\mathrm{mv},K}$ is flat over $\mathcal{O}_K$. 
	
	For a scheme $(S,\varphi_q)$ satisfying (\ref{Katz's condition}, (iii)), we denote by $\modet_{\varphi_q}(\mathcal{O}_S)$ the category of pairs $(\mathcal{V},\phi_{\mathcal{V}})$ where $\mathcal{V}$ is a Zariski-locally free $\mathcal{O}_S$-module of finite rank with an isomorphism $\phi_{\mathcal{V}}: \varphi_q^*\mathcal{V}\xrightarrow{\sim}\mathcal{V}$ of $\mathcal{O}_S$-modules. Fix a geometric point $\overline{s}$ of $S$, we denote by $\mathrm{Rep}^{\mathrm{free}}_{W_n(\mathbb{F}_q)}(\pi_1^{\mathrm{\acute{e}t}}(S,\overline{s}))$ the category of finite free $W_n(\mathbb{F}_q)$-modules with a continuous $\pi_1^{\mathrm{\acute{e}t}}(S,\overline{s})$-action, where $\pi_1^{\mathrm{\acute{e}t}}(S,\overline{s})$ is the étale fundamental group of $(S,\overline{s})$. For the definition and basic properties of $\pi_1^{\mathrm{\acute{e}t}}(S,\overline{s})$, see \cite[Exposé V]{raynaud1971revetements}.
	
	\begin{prop}[Katz]\label{Katz's method: descend using phi_q}
		Let $(S,\varphi_q)$ be a scheme satisfying (\ref{Katz's condition}). Then there is an equivalence of categories between $\mathrm{Rep}^{\mathrm{free}}_{W_n(\mathbb{F}_q)}(\pi_1^{\mathrm{\acute{e}t}}(S,\overline{s}))$ and $\modet_{\varphi_q}(\mathcal{O}_S)$. The functor is given as follows: for an object $\rho_n$ of $\mathrm{Rep}^{\mathrm{free}}_{W_n(\mathbb{F}_q)}(\pi_1^{\mathrm{\acute{e}t}}(S,\overline{s}))$, $\rho_n$ corresponds to an étale local system of $W_n(\mathbb{F}_q)$-modules $\mathcal{F}_n$ over $S$, then $(\nu_*(\mathcal{F}\otimes_{W_n(\mathbb{F}_q)}\mathcal{O}_S),\nu_*(\mathrm{id}\otimes \phi_{\mathcal{O}_S}))$ is an object of $\modet_{\varphi_q}(\mathcal{O}_S)$, where $\nu: S_{\mathrm{\acute{e}t}}\to S_{\mathrm{Zar}}$ is the restriction from the étale topos to the Zariski topos.
	\end{prop}
	
	\begin{proof}
		This is a \cite[Proposition 4.1.1]{katz1973padic}.
	\end{proof}

	\begin{prop}\label{perfect case of Katz's result}
		Let $R$ be a perfect $\mathbb{F}_q$-algebra, and suppose that $R$ is an integral domain. For every $n\geq 1$, let $W_n(R)\coloneq W(R)/p^n$, and let $\varphi$ be the unique endomorphism of $W_n(R)$ such that $\overline{\varphi}: R\to R$ is the $q$-th power map. Then there is a natural equivalence between the category $\modet_{\varphi}(W_n(R))$ and the category $\mathrm{Rep}^{\mathrm{free}}_{W_n(\mathbb{F}_q)}(\pi_1^{\mathrm{\acute{e}t}}(\mathrm{Spec}(R),\overline{s}))$, where $\overline{s}$ is a fixed geometric point of $\mathrm{Spec}(R)$.
	\end{prop}
	
	\begin{proof}
		By \cite[Proposition 3.2.7]{kedlaya2015relative}, there is a natural equivalence between the category $\modet_{\varphi}(W_n(R))$ and the category of étale shaves of flat $W_n(\mathbb{F}_q)/p^n$-modules over $\mathrm{Spec}(R)$ represented by finite étale $R$-schemes. Since $R$ is an integral domain, there is an equivalence of categories between the latter category and $\mathrm{Rep}^{\mathrm{free}}_{W_n(\mathbb{F}_q)}(\pi_1^{\mathrm{\acute{e}t}}(\mathrm{Spec}(R),\overline{s}))$ (see for example \cite[\href{https://stacks.math.columbia.edu/tag/03RV}{Tag 03RV}]{stacks-project}).
	\end{proof}
	
	
	\begin{prop}\label{mod p^n descent using Katz}
		For $n\geq 1$, the functor
		\begin{align*}
			\modet_{\varphi_q}(A_{\mathrm{mv},K}/p^n) & \to \modet_{\varphi_q}(W(A_{q,\infty})/p^n)\\
			M & \mapsto W(A_{q,\infty})/p^n\otimes_{A_{\mathrm{mv},K}/p^n} M
		\end{align*}
		is an equivalence of categories, where $A_{q,\infty}$ is the completed perfection of $A_q$.
	\end{prop}
	
	\begin{proof}
		There is a morphism
		\begin{align}\label{morphism 1}
			\mathrm{Spec}(A_{q,\infty})\to \mathrm{Spec}(A_q)\to \mathrm{Spec}(A_{\mathrm{mv},K}/p^n).
		\end{align}
		We fix a geometric point $\overline{x}_{\infty}$ of $\mathrm{Spec}(A_{q,\infty})$, and denote the image of $\overline{x}$ in $\mathrm{Spec}(A_q)$ (resp.  $\mathrm{Spec}(A_{\mathrm{mv},K}/p^n)$) by $\overline{x}$ (resp.  $\overline{x}'$). Recall that an $A_{\mathrm{mv},K}/p^n$-module $M$ is locally free if and only if $M$ is finite projective, hence there is an equivalence of categories
		\begin{align}\label{equi 3}
			\Gamma(\mathrm{Spec}(A_{\mathrm{mv},K}/p^n),-): \modet_{\varphi_q}(\mathcal{O}_{\mathrm{Spec}(A_{\mathrm{mv},K}/p^n)})\to \modet_{\varphi_q}(A_{\mathrm{mv},K}/p^n).
		\end{align}
		Combining with Proposition \ref{Katz's method: descend using phi_q}, there is an equivalence of categories
		\begin{align}\label{equi 1}
			\repfr_{W_n(\mathbb{F}_q)}\Big(\pi_1^{\mathrm{\acute{e}t}}(\mathrm{Spec}(A_{\mathrm{mv},K}/p^n),\overline{x}')\Big)\xrightarrow{\sim} \modet_{\varphi_q}(A_{\mathrm{mv},K}/p^n).
		\end{align}
		By Proposition \ref{perfect case of Katz's result}, there is an equivalence of categories
		\begin{align}\label{equi 2}
			\repfr_{W_n(\mathbb{F}_q)}\Big(\pi_1^{\mathrm{\acute{e}t}}(\mathrm{Spec}(A_{q,\infty}),\overline{x}_{\infty})\Big)\xrightarrow{\sim} \modet_{\varphi_q}(W(A_{q,\infty})/p^n).
		\end{align}
		Note that (\ref{morphism 1}) induces a natural map of profinite groups
		\begin{align}\label{map 1}
			\pi_1^{\mathrm{\acute{e}t}}(\mathrm{Spec}(A_{q,\infty}),\overline{x}_{\infty})\to \pi_1^{\mathrm{\acute{e}t}}(\mathrm{Spec}(A_q),\overline{x})\to \pi_1^{\mathrm{\acute{e}t}}(\mathrm{Spec}(A_{\mathrm{mv},K}/p^n),\overline{x}').
		\end{align}
		The map $\pi_1^{\mathrm{\acute{e}t}}(\mathrm{Spec}(A_q),\overline{x})\to \pi_1^{\mathrm{\acute{e}t}}(\mathrm{Spec}(A_{\mathrm{mv},K}/p^n),\overline{x}')$ is an isomorphism since it is induced by a universal homeomorphism of schemes (\cite[EXPOSÉ VIII, Théorème 1.1]{SGA4}), and $\pi_1^{\mathrm{\acute{e}t}}(\mathrm{Spec}(A_{q,\infty}),\overline{x}_{\infty})\to \pi_1^{\mathrm{\acute{e}t}}(\mathrm{Spec}(A_q),\overline{x})$ is also an isomorphism since $\mathrm{Spec}(A_{q,\infty})\to \mathrm{Spec}(A_q)$ factorizes as
		\[\mathrm{Spec}(A_{q,\infty})\to \mathrm{Spec}(A_q^{\mathrm{perf}})\to \mathrm{Spec}(A_q)\]
		where the first morphism induces an isomorphism of étale fundamental groups (\cite[Theorem 7.4.8]{WS2020berkeley}) and the second morphism is a universal homeomorphism of schemes. Therefore, the composition (\ref{map 1}) is an isomorphism.
		
		Combining with (\ref{equi 1}) and (\ref{equi 2}), there is a commutative diagram of categories:
		\begin{equation*}
			\begin{tikzcd}
				\repfr_{W_n(\mathbb{F}_q)}(\pi_1^{\mathrm{\acute{e}t}}(\mathrm{Spec}(A_{\mathrm{mv},K}/p^n),\overline{x}')) \arrow[rr,"\text{pullback}","\sim"'] \arrow[d,"\wr"] & & \repfr_{W_n(\mathbb{F}_q)}(\pi_1^{\mathrm{\acute{e}t}}(\mathrm{Spec}(A_{q,\infty}),\overline{x}_{\infty})) \arrow[d,"\wr"] \\ 
				\modet_{\varphi_q}(A_{\mathrm{mv},K}/p^n) \arrow[rr,"W(A_{q,\infty})/p^n\otimes_{A_{\mathrm{mv},K}/p^n}-"] & & \modet_{\varphi_q}(W(A_{q,\infty})/p^n).
			\end{tikzcd}
		\end{equation*}
		Then the conclusion follows.
	\end{proof}
	
	\begin{thm}\label{descent of phi_q-module}
		The functor
		\begin{align*}
			\modet_{\varphi_q}(A_{\mathrm{mv},E}) & \to \modet_{\varphi_q}(W_E(A_{\infty}))\\
			M & \mapsto W_E(A_{\infty})\otimes_{A_{\mathrm{mv},E}} M
		\end{align*}
		is an equivalence of categories.
	\end{thm}
	\begin{proof}
		First we assume that $E=K$. For fully faithfulness, let $M,N$ be finite projective étale $\varphi_q$-modules over $A_{\mathrm{mv},K}$, $M_{\infty}\coloneq  W(A_{q,\infty})\otimes_{A_{\mathrm{mv},K}}M$, $N_{\infty}\coloneq  W(A_{q,\infty})\otimes_{A_{\mathrm{mv},K}}N$. If $f: M\to N$ is a homomorphism of $\varphi_q$-modules, and $f_{\infty}\coloneq \id\otimes f: M_{\infty}\to N_{\infty}$ is the zero map, then the map $\overline{f}_{\infty}: M_{\infty}/p^n\to N_{\infty}/p^n$ induced by $f_{\infty}$ is zero. By Proposition \ref{mod p^n descent using Katz}, we deduce that, for any $n\geq 1$, the map $\overline{f}: M/p^n\to N/p^n$ induced by $f$ is zero, hence $f=0$. Moreover, let $h_{\infty}: M_{\infty}\to N_{\infty}$ be a homomorphism of $\varphi_q$-modules, by Proposition \ref{mod p^n descent using Katz}, for every $n$, there exists a unique homomorphism $h_n: M/p^n\to N/p^n$ of $\varphi_q$-modules such that $\id\otimes h_n: M_{\infty}/p^n\to N_{\infty}/p^n$ is induced by $h_{\infty}$. Then $\{h_n\}_n$ induces a homomorphism $h: M\to N$ of $\varphi_q$-modules, and $\id\otimes h=h_{\infty}$.
		
		For the essential surjectivity, let $M_{\infty}$ be a finite free étale $\varphi_q$-module over $W(A_{q,\infty})$ of rank $r$, by Proposition \ref{mod p^n descent using Katz}, for every $n$, there exists a finite free étale $\varphi_q$-module $M_n$ of rank $r$ over $A_{\mathrm{mv},K}/p^n$ such that $M_{\infty}/p^n\cong W(A_{q,\infty})/p^n\otimes_{A_{\mathrm{mv},K}/p^n} M_n$. Moreover, these $M_n$ are compatible, i.e. $M_{n+1}/p^n\cong M_n$ as finite projective étale $\varphi_q$-modules over $A_{\mathrm{mv},K}/p^n$. We put $M\coloneq \varprojlim_n M_n$, which is a $\varphi_q$-module over $A_{\mathrm{mv},K}$. Take $x_1,\dots,x_r\in M$ such that $\overline{x_1},\dots,\overline{x_r}$ forms an $A_q$-basis of $M_1\cong M/p$, then $\overline{x_1},\dots,\overline{x_r}$ forms a $A_{\mathrm{mv},K}/p^n$-basis of $M_n\cong M/p^n$ for any $n$. As $p$ is contained in the Jacobson radical of $A_{\mathrm{mv},K}$, using Nakayama's lemma, we deduce that $x_1,\dots,x_r$ generate $M$ as an $A_{\mathrm{mv},K}$-module. Moreover, if $\sum_{i=1}^ra_ix_i=0$, by reducing modulo $p^n$ for $n\geq 1$, we get $a_i\in \bigcap_{n\geq 1}p^nA_{\mathrm{mv},K}=0$ for any $0\leq i\leq r$, thus $M=\bigoplus_{i=1}^rA_{\mathrm{mv},K}x_i$ is free, and we have
		\[A_{\mathrm{mv},K}\otimes_{\varphi_q,A_{\mathrm{mv},K}}M\cong \varprojlim_n(A_{\mathrm{mv},K}\otimes_{\varphi_q,A_{\mathrm{mv},K}}M_n)\xrightarrow[\sim]{\varprojlim_n(\mathrm{id}\otimes\varphi_q)}\varprojlim_nM_n=M,\]
		hence $M$ is a finite projective étale $\varphi_q$-module over $A_{\mathrm{mv},K}$. Moreover, 
		\[M_{\infty} \cong \varprojlim_n \left( W(A_{q,\infty})\otimes_{A_{\mathrm{mv},K}/p^n} M_n\right)\cong W(A_{q,\infty}) \otimes_{A_{\mathrm{mv},K}}\varprojlim_nM_n= W(A_{q,\infty}) \otimes_{A_{\mathrm{mv},K}} M,\]
		which proves the essential surjectivity.
		
		Now we show the fully faithfulness for the general case. 
		The faithfulness can be deduced directly from the following commutative diagram:
		\begin{equation}\label{fatihfulness A and A_{infty}}
			\begin{tikzcd}
				\mathrm{Hom}_{\modet_{\varphi_q}(A_{\mathrm{mv},E})}(M_1,M_2) \arrow[d,hook]\arrow[r] & \mathrm{Hom}_{\modet_{\varphi_q}(W_E(A_{\infty}))}(M_{1,\infty},M_{2,\infty})\arrow[d,hook]\\
				\mathrm{Hom}_{\modet_{\varphi_q}(A_{\mathrm{mv},K})}(M_1,M_2) \arrow[r,"\sim"] & \mathrm{Hom}_{\modet_{\varphi_q}(W(A_{q,\infty}))}(M_{1,\infty},M_{2,\infty})
			\end{tikzcd}
		\end{equation}
		where $M_1, M_2$ are finite projective étale $\varphi_q$-modules over $A_{\mathrm{mv},E}$, $M_{i,\infty}\coloneq W_{E}(A_{\infty})\otimes_{A_{\mathrm{mv},E}}M_{i}$ for $i=1,2$, and in the bottom horizontal map, we use the identification
		\[W_{E}(A_{\infty})\otimes_{A_{\mathrm{mv},E}}M_i\cong W(A_{q,\infty})\otimes_{A_{\mathrm{mv},K}}M_i,\quad i=1,2.\]
		For $f_{\infty}\in \mathrm{Hom}_{\modet_{\varphi_q}(W_E(A_{\infty}))}(M_{1,\infty},M_{2,\infty})$, the commutative diagram (\ref{fatihfulness A and A_{infty}}) shows that $f_{\infty}=\mathrm{id}\otimes f$ for some $f\in\mathrm{Hom}_{\modet_{\varphi_q}(A_{\mathrm{mv},K})}(M_1,M_2)$. Since $A_{\mathrm{mv},E}=\mathcal{O}_E\otimes_{\sigma_{0},\mathcal{O}_K}A_{\mathrm{mv},K}$, if we can show that $f$ is $\mathcal{O}_E$-linear, then $f$ is $A_{\mathrm{mv},E}$-linear, which finishes the proof of the fully faithfulness. For $x\in \mathcal{O}_E$, consider the map
		\begin{align*}
			f_x: M_1 \to M_2,\
			m_1 \mapsto f(x\cdot m_1)-x\cdot f(m_1).
		\end{align*}
		Then $f_x$ is a homomorphism of $\varphi_q$-modules over $A_{\mathrm{mv},K}$. Moreover, for any $a\otimes m_1\in W(A_{q,\infty})\otimes_{A_{\mathrm{mv},K}}M_1\cong M_{1,\infty}$,
		\begin{align*}
			(\mathrm{id}\otimes f_x)(a\otimes m_1)=a\otimes f(xm_1)-a\otimes(xf(m_1))
		\end{align*}
		Using the identification $W(A_{q,\infty})\otimes_{A_{\mathrm{mv},K}}M_1\cong M_{1,\infty}$, since $\mathrm{id}\otimes f_x=f_{\infty}$ is $\mathcal{O}_E$-linear, we have 
		\begin{align*}
			a\otimes f(xm_1)=f_{\infty}(a\otimes xm_1)&=f_{\infty}(x\cdot(a\otimes m_1))\\
			&=x\cdot f_{\infty}(a\otimes m_1)=x\cdot(a\otimes f(m_1))=a\otimes(xf(m_1)).
		\end{align*}
		Thus $\mathrm{id} \otimes f_x=0$, hence $f_x=0$ by the fully faithfulness of $W(A_{q,\infty})\otimes_{A_{\mathrm{mv},K}}-$, i.e. $f$ is $\mathcal{O}_E$-linear.
		
		It remains to show the essential surjectivity.
		Let $M_{\infty}$ be a finite projective étale $\varphi_q$-module over $W_E(A_{\infty})$. The case $E=K$ implies that there is a finite projective étale $\varphi_q$-module $M$ over $A_{\mathrm{mv},K}$ such that 
		\[W(A_{q,\infty})\otimes_{A_{\mathrm{mv},K}}M\cong M_{\infty}. \]
		For $a\in \mathcal{O}_E$, the scalar multiplication $M_{\infty}\to M_{\infty}, m\mapsto a\cdot m$ (viewing $M_{\infty}$ as a $W_E(A_{\infty})$-module) is an endomorphism of the étale $\varphi_q$-module $M_{\infty}$ over $W(A_{q,\infty})$, hence descends to an endomorphism $g_a$ of the étale $\varphi_q$-module $M$ over $A_{\mathrm{mv},K}$, which gives an $\mathcal{O}_E$-module structure on $M$. Moreover, if $a\in \mathcal{O}_{K}$, then the scalar multiplication by $a$ on $M$ agrees with $g_{\sigma_{0}(a)}$, since they are the same on $M_{\infty}$. Hence there is a $A_{\mathrm{mv},E}=\mathcal{O}_E\otimes_{\sigma_0,\mathcal{O}_K}A_{\mathrm{mv},K}$-module structure on $M$. Since
		\[W_E(A_{\infty})\otimes_{A_{\mathrm{mv},E}}M\cong W(A_{q,\infty})\otimes_{A_{\mathrm{mv},K}}M\cong M_{\infty}\]
		is finite projective over $W_E(A_{\infty})$ and $A_{\mathrm{mv},E}\to W_E(A_{\infty})$ is faithfully flat (Proposition \ref{A_{mv,E} to W_E(A_{infty}) is faithfully flat}), by faithfully flat descent of projectivity (see for example \cite[\href{https://stacks.math.columbia.edu/tag/058S}{Tag 058S}]{stacks-project}), $M$ is a finite projective module over $A_{\mathrm{mv},E}$. Hence by Lemma \ref{description of finte projective modules}, $M$ is a finite free $A_{\mathrm{mv},E}$-module. Besides, it is easy to check that $M$ is a $\varphi_q$-module over $A_{\mathrm{mv},E}$.  Since
		\[A_{\mathrm{mv},E}\otimes_{\varphi_q,A_{\mathrm{mv},E}}M\cong A_{\mathrm{mv},K}\otimes_{\varphi_q,A_{\mathrm{mv},K}}M\cong M,\]
		we see that $M$ is a finite projective étale $\varphi_q$-module over $A_{\mathrm{mv},E}$.
	\end{proof}

	\subsection{Finitely presented étale $(\varphi_q,\mathcal{O}_K^{\times})$-modules}\label{Frobenius descent of finitely presented étale (varphi_q,mathcal{O}_K*)-modules}
	In this section, we use a standard lemma in algebraic $K$-theory (Lemma \ref{existence of equivariant surjection of phi_q-modules}) to extend Theorem \ref{descent of phi_q-module} to finitely presented étale $\varphi_q$-modules over $A_{\mathrm{mv},E}$, then we add the $\mathcal{O}_K^{\times}$-action to deduce that every finitely presented étale $(\varphi_q,\mathcal{O}_K^{\times})$-module over $W_E(A_{\infty})$ descends.
	
	Let $R$ be a ring, $M$ be an $R$-module, $r\in R$, we put $M[r]\coloneq\left\{m\in M: rm=0\right\}$ which is a submodule of $M$. We start with the structure of finitely presented étale $\varphi_q$-modules over $A_{\mathrm{mv},E}$. 
	
	\begin{lem}\label{phi_q invarinat ideals of A}
		If $I\subseteq A$ is an ideal of $A$, and $A\cdot\varphi_q(I)=I$, then $I=0$ or $I=A$.
	\end{lem}
	\begin{proof}
		Note that $\varphi_q^{f'}=\mathrm{Frob}_{q'}: A\to A$, where $q'=q^{f'}=\#\mathbb{F}$. In particular, 
		\[A\cdot \varphi_q^{nf'}(I)\subseteq I^{q'^n},\quad n\geq 1.\]
		Since $A\cdot\varphi_q(I)=I$, by induction on $m\geq 1$, we deduce $A\cdot\varphi_q^m(I)=I$ for any $m\geq 1$. As a consequence, we have
		\[I= A\cdot \varphi_q^{nf'}(I)\subseteq I^{q'^n}\subseteq I, \quad n\geq 1,\]
		which implies that $I^{q'^n}= I$ for any $n\geq 1$. Therefore,
		\[I=\bigcap_{n=1}^{+\infty}I^{q'^n}\subseteq \bigcap_{n=1}^{+\infty}I^{n}.\]
		On the other hand, $A$ is a Noetherian domain, hence if $I$ is a proper ideal of $A$, by Krull's intersection theorem, $\bigcap_{n=1}I^{n}=0$, hence $I=0$.
	\end{proof}
	
	\begin{cor}\label{finite etale phi_q-module over A}
		If $M$ is a finitely presented étale $\varphi_q$-module over $A$, then $M$ is a finite free $A$-module.
	\end{cor}
	\begin{proof}
		Let $\mathrm{Fit}_k(M)$ be the $k$-th fitting ideal of $M$, $k\geq0$ (for the definition, see \cite[\href{https://stacks.math.columbia.edu/tag/07Z8}{Tag 07Z8}]{stacks-project}). By \cite[\href{https://stacks.math.columbia.edu/tag/07ZA}{Tag 07ZA}]{stacks-project}, we have
		\[\mathrm{Fit}_k(M)=\mathrm{Fit}_k(A\otimes_{\varphi_q,A}M)=A\cdot \varphi_q(\mathrm{Fit}_k(M)).\]
		Then Lemma \ref{phi_q invarinat ideals of A} implies that $\mathrm{Fit}_k(M)$ is either $0$ or $A$. Thus \cite[\href{https://stacks.math.columbia.edu/tag/07ZD}{Tag 07ZD}]{stacks-project} implies that $M$ is a finite projective $A$-module, hence is finite free by \cite[$\S$1 satz 3]{lutkebohmert1977vektorraumbundel} (Lemma \ref{description of finte projective modules}).
	\end{proof}

	\begin{prop}\label{description of finitely presented etale phi_q-module over A_{mv,E}}
		Let $M$ be a non-zero finitely presented étale $\varphi_q$-module over  $A_{\mathrm{mv},E}$. Then there exists $1\leq n_1<\dots<n_r\leq \infty$ and $m_1,\dots,m_r\geq 1$ such that
		\[M\cong \bigoplus_{i=1}^r\left(A_{\mathrm{mv},E}/\varpi^{n_i}\right)^{\oplus m_i},\]
		here we use the convention that $A_{\mathrm{mv},E}/\varpi^{\infty}=A_{\mathrm{mv},E}$.
	\end{prop}
	
	\begin{proof}
		As $\varphi_q: A_{\mathrm{mv},E}\to A_{\mathrm{mv},E}$ is faithfully flat, we can deduce that $\varpi^sM/\varpi^{s+1}M$ is a finitely presented étale $\varphi_q$-module over $A$ for any $s\geq 0$, hence is finite free (Corollary \ref{finite etale phi_q-module over A}). 
		We apply \cite[Theorem 4.3]{marquis2024studyvariouscategoriesgravitating} to deduce that
		\[M\cong M_{\infty}\oplus \bigoplus_{i=1}^{r-1}M_{i}\]
		with $M_{\infty}$ is a finite projective $A_{\mathrm{mv},E}$-module, and $M_{i}$ is a finite projective $A_{\mathrm{mv},E}/\varpi^{n_i}$-module for some $1\leq n_i<+\infty$. By Lemma \ref{description of finte projective modules}, every $M_i$ is a finite free $A_{\mathrm{mv},E}/\varpi^{n_i}$-module, and $M_{\infty}$ is a finite free $A_{\mathrm{mv},E}$-module.
	\end{proof}

	Since $A_{\mathrm{mv},E}$ is a Noetherian domain, and $\varphi_q: A_{\mathrm{mv},E}\to A_{\mathrm{mv},E}$ is flat (Lemma \ref{A_{mathrm{m.v.},K} is noetherian} and Lemma \ref{phi_q is flat}), one can easily deduce that
	\begin{lem}\label{f.p. etale (phi_q,O_K*)-modules over A is an abelian category}
		The categories $\modetfp_{\varphi_q}(A_{\mathrm{mv},E})$ and $\modetfp_{\varphi_q,\mathcal{O}_K^{\times}}(A_{\mathrm{mv},E})$ are abelian categories.
	\end{lem}
	
	We quote an important and technical lemma:
	\begin{lem}\label{existence of equivariant surjection of phi_q-modules}
		Let $R$ be a ring with an endomorphism $\varphi_q$. Let $M$ be a finitely generated étale $\varphi_q$-module over $R$. Then there exists a finite free étale $\varphi_q$-module $F$ and a $\varphi_q$-equivariant surjection $F\twoheadrightarrow M$.
	\end{lem}
	\begin{proof}
		This is \cite[Lemma 1.5.2]{kedlaya2015relative}.
	\end{proof}
	
	\begin{cor}\label{f.p. phi_q-modules}
		Let $R$ be a ring with a flat endomorphism $\varphi_q$. Let $M$ be a finitely presented étale $\varphi_q$-module over $R$. Then there exists an exact sequence $F'\to F\to M\to 0$ of étale $\varphi_q$-modules over $R$, where $F,F'$ are finite free étale $\varphi_q$-modules over $R$.
	\end{cor}
	\begin{proof}
		By Lemma \ref{existence of equivariant surjection of phi_q-modules}, there exists an exact sequence $F\xrightarrow{f} M\to 0$ of étale $\varphi_q$-modules over $R$. Since $M$ is finitely presented, $\ker f$ is finitely generated. Since $M$, $F$ are étale and $\varphi_q$ is flat, $\ker f$ is étale, then we apply Lemma \ref{existence of equivariant surjection of phi_q-modules} to $\ker f$ to conclude.
	\end{proof}
	
	Let $M_{\infty}$ be a finitely presented étale $\varphi_q$-module over $W_E(A_{\infty})$. By Lemma \ref{existence of equivariant surjection of phi_q-modules}, we can take a finite free étale $\varphi_q$-module $F_{\infty}$ over $W_E(A_{\infty})$ and a $\varphi_q$-equivariant surjection $f_{\infty}: F_{\infty}\twoheadrightarrow M_{\infty}$. By Theorem \ref{descent of phi_q-module}, there exists a unique finite free étale $\varphi_q$-module $F$ over $A_{\mathrm{mv},E}$ such that $F_{\infty}\cong W_E(A_{\infty})\otimes_{A_{\mathrm{mv},E}}F$. We use the injection $F\hookrightarrow W_E(A_{\infty})\otimes_{A_{\mathrm{mv},E}}F, x\mapsto 1\otimes x$ to identify $F$ with a $A_{\mathrm{mv},E}$-submodule of $F_{\infty}$, then we put
	\begin{align}\label{definition of Dec}
		\mathrm{Dec}(M_{\infty})\coloneq f_{\infty}(F)\subset M_{\infty}.
	\end{align}
	Then $\mathrm{Dec}(M_{\infty})$ is a finite $A_{\mathrm{mv},E}$-module with an endomorphism $\varphi_q$ induced by the endomorphism of $F$.
	
	Suppose that there is another finite free étale $\varphi_q$-module $F_{\infty}'$ over $W_E(A_{\infty})$ and a $\varphi_q$-equivariant homomorphism $f_{\infty}': F_{\infty}'\to M_{\infty}$ (not necessarily surjective). Let $F'$ be a finite free étale $\varphi_q$-module over $A_{\mathrm{mv},E}$ such that $F_{\infty}'\cong W_E(A_{\infty})\otimes_{A_{\mathrm{mv},E}}F'$, we have the following lemma:
	
	\begin{lem}\label{Dec(M_{infty}) is well-defined}
		With notations above, we have $f_{\infty}'(F')\subseteq f_{\infty}(F)$.
	\end{lem}
	
	\begin{proof}
		Let $L$ be the fibre product of the diagram $F_{\infty}'\to M_{\infty}\twoheadleftarrow F_{\infty}$ in the category of $W_E(A_{\infty})$-modules, 
		\[L=\left\{(x,x')\in F_{\infty}\oplus F_{\infty}': f_{\infty}(x)=f_{\infty}'(x')\right\},\]
		then $L$ is a $\varphi_q$-module, and by the surjectivity of $f_{\infty}$, there is a short exact sequence of $\varphi_q$-modules over $W_E(A_{\infty})$:
		\[0\to \ker f_{\infty}\xrightarrow{x\mapsto (x,0)}L\xrightarrow{(x,x')\mapsto x'}F_{\infty}'\to 0.\]
		As $M_{\infty}$ is finitely presented, $\ker f_{\infty}$ must be a finite $W_E(A_{\infty})$-module, hence $L$ is a finite $W_E(A_{\infty})$-module. Moreover, $ M_{\infty}$ and $F_{\infty}$ are finitely presented étale $\varphi_q$-modules, hence a standard argument using the five lemma implies that $\ker f_{\infty}$ is étale, hence $L$ is étale. Therefore, we can apply Lemma \ref{existence of equivariant surjection of phi_q-modules} to deduce that there exists a finite free étale $\varphi_q$-module $F_{\infty}''$ with a $\varphi_q$-equivariant surjection $F_{\infty}''\twoheadrightarrow L$. Let $h_{1,\infty}$ be the composition $F_{\infty}''\twoheadrightarrow L\to F_{\infty}$, and let $h_{2,\infty}$ be the composition $F_{\infty}''\twoheadrightarrow L\twoheadrightarrow F_{\infty}'$, there is a commutative diagram of finitely presented étale $\varphi_q$-modules
		\begin{equation*}
			\begin{tikzcd}
				F_{\infty}'' \arrow[d,two heads,"h_{2,\infty}"]\arrow[r,"h_{1,\infty}"] & F_{\infty}\arrow[d,two heads,"f_{\infty}"]\\
				F_{\infty}'\arrow[r,"f_{\infty}'"] & M_{\infty}.
			\end{tikzcd}
		\end{equation*}
		Let $F''$ be the finite free étale $\varphi_q$-module over $A_{\mathrm{mv},E}$ such that $F''_{\infty}\cong W_E(A_{\infty})\otimes F''$, then $h_{1,\infty}$ and $h_{2,\infty}$ descend to $F''\xrightarrow{h_1} F$ and $F''\xrightarrow{h_2} F'$ (Theorem \ref{descent of phi_q-module}). Moreover, by Proposition \ref{A_{mv,E} to W_E(A_{infty}) is faithfully flat}, $A_{\mathrm{mv},E}\hookrightarrow W_E(A_{\infty})$ is faithfully flat, hence $h_2$ is surjective. Therefore,
		\[f_{\infty}'(F')=f_{\infty}'(h_{2,\infty}(F''))=f_{\infty}(h_{1,\infty}(F''))\subseteq f_{\infty}(F).\]
	\end{proof}
	
	In particular, $\mathrm{Dec}(M_{\infty})$ is independent of the choice of $F_{\infty}$ and the surjection $f_{\infty}: F_{\infty}\twoheadrightarrow M_{\infty}$. 
	
	\begin{defn}\label{defn of decompletion}
		Let $M_{\infty}$ be a finitely presented étale $\varphi_q$-module over $W_E(A_{\infty})$. We call $\mathrm{Dec}(M_{\infty})$ the \textit{decompletion and deperfection} of $M_{\infty}$. Moreover, any homomorphism $h_{\infty}: M_{\infty}\to N_{\infty}$ between finitely presented étale $\varphi_q$-modules over $W_E(A_{\infty})$ restricts to a homomorphism $h: \mathrm{Dec}(M_{\infty})\to \mathrm{Dec}(N_{\infty})$ of finite $A_{\mathrm{mv},E}$-modules which commutes with $\varphi_q$. Thus $\mathrm{Dec}$ is a functor sending finitely presented étale $\varphi_q$-modules over $W_E(A_{\infty})$ to finite $A_{\mathrm{mv},E}$-modules with a semi-linear endomorphism $\varphi_q$.
	\end{defn}

	\begin{lem}\label{the intersection operation behaves very well}
		Let $f: F'\to F$ be a homomorphism of finite free étale $\varphi_q$-module over $A_{\mathrm{mv},E}$ which induces a homomorphism $f_{\infty}:  W_E(A_{\infty})\otimes_{A_{\mathrm{mv},E}} F'\to W_E(A_{\infty})\otimes_{A_{\mathrm{mv},E}} F$ of étale $\varphi_q$-module over $W_E(A_{\infty})$. Then
		\[F\cap \mathrm{Im}(f_{\infty})=\mathrm{Im}(f).\]
	\end{lem}
	\begin{proof}
		Obviously, $\mathrm{Im}(f)\subseteq F\cap \mathrm{Im}(f_{\infty})$. We put $F'_{\infty}\coloneq W_E(A_{\infty})\otimes F'$ and $F_{\infty}\coloneq W_E(A_{\infty})\otimes F$. Let $e_1',\dots,e_r'$ be an $A_{\mathrm{mv},E}$-basis of $F'$, then $\mathrm{Im}(f)=\sum_{i=1}^{r}A_{\mathrm{mv},E}f(e_i')$, and $\mathrm{Im}(f_{\infty})=\sum_{i=1}^{r}W_E(A_{\infty})f(e_i')$. As a consequence, $\mathrm{Im}(f_{\infty})$ is the image of the natural map
		\begin{align}\label{natural map induced by tensoring}
			W_E(A_{\infty})\otimes_{A_{\mathrm{mv},E}}\mathrm{Im}(f)\to F_{\infty}.
		\end{align}
		On the other hand, $A_{\mathrm{mv},E}\to W_E(A_{\infty})$ is flat, hence (\ref{natural map induced by tensoring}) is also injective as it is the base change of $\mathrm{Im}(f)\hookrightarrow F$, hence we may identify $\mathrm{Im}(f_{\infty})$ with $W_E(A_{\infty})\otimes_{A_{\mathrm{mv},E}}\mathrm{Im}(f)$. Now consider the following commutative diagram with exact rows
		\begin{equation*}
			\begin{tikzcd}
				0\arrow[r] &\mathrm{Im}(f)\arrow[d]\arrow[r]&F\arrow[d,hook]\arrow[r,"\pi"]& F/\mathrm{Im}(f)\arrow[r]\arrow[d]&0\\
				0\arrow[r] &\mathrm{Im}(f_{\infty})\arrow[r]&F_{\infty}\arrow[r,"\pi_{\infty}"]& F_{\infty}/\mathrm{Im}(f_{\infty})\arrow[r]&0
			\end{tikzcd}
		\end{equation*} 
		Note that $F_{\infty}/\mathrm{Im}(f_{\infty})\cong W_E(A_{\infty})\otimes_{A_{\mathrm{mv},E}}F/\mathrm{Im}(f)$, and by \cite[\href{https://stacks.math.columbia.edu/tag/05CK}{Tag 05CK}]{stacks-project} and Lemma \ref{A_{mv,E} to W_E(A_{infty}) is faithfully flat}, the right vertical map is injective. Hence if $x\in F\cap \mathrm{Im}(f_{\infty})$, we have $\pi_{\infty}(x)=0$, which implies that $\pi(x)=0$, hence $x$ lies in $\mathrm{Im}(f)$. Therefore, $\mathrm{Im}(f)=F\cap \mathrm{Im}(f_{\infty})$.
	\end{proof}
	
	As a consequence, we can deduce the following result:
	
	\begin{prop}\label{Dec(M_{infty}) is etale}
		For any finitely presented étale $\varphi_q$-module $M_{\infty}$ over $W_E(A_{\infty})$, $\mathrm{Dec}(M_{\infty})$ is a finitely presented étale $\varphi_q$-module over $A_{\mathrm{mv},E}$. 
	\end{prop}
	
	\begin{proof}
		We only need to show that the linearization map (\ref{linearization map}) $A_{\mathrm{mv},E}\otimes_{\varphi_q}\mathrm{Dec}(M_{\infty})\to \mathrm{Dec}(M_{\infty})$ is an isomorphism.
		By Corollary \ref{f.p. phi_q-modules}, there exist a short exact sequence of étale $\varphi_q$-modules
		\[F_{\infty}'\xrightarrow{g_{\infty}}F_{\infty}\xrightarrow{f_{\infty}}M_{\infty}\to 0,\]
		where $F_{\infty}, F_{\infty}'$ are finite free étale $\varphi_q$-modules over $W_E(A_{\infty})$, and this induces a sequence of finite $\varphi_q$-modules over $A_{\mathrm{mv},E}$ by Theorem \ref{descent of phi_q-module} and Definition \ref{defn of decompletion}:
		\begin{align}\label{a sequence of phi_q-modules}
			F'\xrightarrow{g}F\xrightarrow{f}\mathrm{Dec}(M_{\infty})\to 0,
		\end{align}
		where $F, F'$ are finite free étale $\varphi_q$-modules over $A_{\mathrm{mv},E}$ such that $F_{\infty}\cong W_E(A_{\infty})\otimes_{A_{\mathrm{mv},E}}F$, $F_{\infty}'\cong W_E(A_{\infty})\otimes_{A_{\mathrm{mv},E}}F'$.
		By (\ref{definition of Dec}), $f$ is surjective and $f\circ g=0$. Moreover, if $x\in F$, $f(x)=0$, then $x\in F\cap\mathrm{Im}(g_{\infty})=\mathrm{Im}(g)$ by Lemma \ref{the intersection operation behaves very well}. Thus (\ref{a sequence of phi_q-modules}) is an exact sequence. Then the étaleness of $F'$ and $F$ implies that $\mathrm{Dec}(M_{\infty})$ is étale. Hence $\mathrm{Dec}(M_{\infty})$ is a finitely presented étale $\varphi_q$-module over $A_{\mathrm{mv},E}$.
	\end{proof}
	
	\begin{thm}\label{descent of fin pre etale phi_q-mod}
		The functor
		\begin{align*}
			\mathrm{Dec}:\ \modetfp_{\varphi_q}(W_E(A_{\infty}))&\to \modetfp_{\varphi_q}(A_{\mathrm{mv},E})\\
			M_{\infty}&\mapsto \mathrm{Dec}(M_{\infty})
		\end{align*}
		is an equivalence of categories with a quasi-inverse given by $W_E(A_{\infty})\otimes_{A_{\mathrm{mv},E}}-$. Moreover, $\mathrm{Dec}$ preserves short exact sequences.
	\end{thm}
	
	\begin{proof}
		We need to show that there are natural isomorphisms
		\begin{align*}
			\mathrm{Dec}\circ \left(W_E(A_{\infty})\otimes_{A_{\mathrm{mv},E}}-\right)\cong \mathrm{Id}, W_E(A_{\infty})\otimes_{A_{\mathrm{mv},E}}\left(\mathrm{Dec}(-)\right)\cong \mathrm{Id}. 
		\end{align*}
		
		For a finitely presented étale $\varphi_q$-module $M$ over $A_{\mathrm{mv},E}$, by Lemma \ref{existence of equivariant surjection of phi_q-modules}, there exists a $\varphi_q$-equivariant surjection $f: F\twoheadrightarrow M$, where $F$ is a finite free étale $\varphi_q$-module over $A_{\mathrm{mv},E}$. Then applying the functor $W_E(A_{\infty})\otimes_{A_{\mathrm{mv},E}}-$, we get a $\varphi_q$-equivariant surjection $f_{\infty}: W_E(A_{\infty})\otimes_{A_{\mathrm{mv},E}}F\twoheadrightarrow W_E(A_{\infty})\otimes_{A_{\mathrm{mv},E}}M$. By Definition \ref{defn of decompletion}, we have
		\[\mathrm{Dec}\left(W_E(A_{\infty})\otimes_{A_{\mathrm{mv},E}}M\right)=f_{\infty}(F)=M.\]
		
		Conversely, for a finitely presented étale $\varphi_q$-module $M_{\infty}$ over $W_E(A_{\infty})$, By the proof of Proposition \ref{Dec(M_{infty}) is etale}, there is an exact sequence of étale $\varphi_q$-modules $F_{\infty}'\xrightarrow{g_{\infty}} F_{\infty}\xrightarrow{f_{\infty}}M_{\infty}\to 0$, where $F_{\infty}'$, $F_{\infty}$ are as in \textit{loc. cit.}, and the functor $\mathrm{Dec}$ induces an exact sequence $F'\xrightarrow{g}F\xrightarrow{f}\mathrm{Dec}(M_{\infty})\to 0$. Then there is a commutative diagram with exact rows
		\begin{equation*}
			\begin{tikzcd}
				W_E(A_{\infty})\otimes_{A_{\mathrm{mv},E}}F'\arrow[r,"\mathrm{id}\otimes g"]\arrow[d,"\wr"]& W_E(A_{\infty})\otimes_{A_{\mathrm{mv},E}}F\arrow[r,"\mathrm{id}\otimes f"]\arrow[d,"\wr"]& W_E(A_{\infty})\otimes_{A_{\mathrm{mv},E}}\mathrm{Dec}(M_{\infty})\arrow[r]\arrow[d]&0\\
				F_{\infty}'\arrow[r,"g_{\infty}"]&F_{\infty}\arrow[r,"f_{\infty}"]&M_{\infty}\arrow[r]&0
			\end{tikzcd}
		\end{equation*}
		where the right vertical map is the natural homomorphism
		\[W_E(A_{\infty})\otimes_{A_{\mathrm{mv},E}}\mathrm{Dec}(M_{\infty})\to M_{\infty},\ x\otimes m\mapsto xm.\]
		This implies that the natural homomorphism $W_E(A_{\infty})\otimes_{A_{\mathrm{mv},E}}\mathrm{Dec}(M_{\infty})\to M_{\infty}$ is an isomorphism.
		
		It remains to show that $\mathrm{Dec}$ preserves short exact sequences, which is a direct consequence of the fact that  $A_{\mathrm{mv},E}\hookrightarrow W_E(A_{\infty})$ is faithfully flat (Lemma \ref{A_{mv,E} to W_E(A_{infty}) is faithfully flat}) and $\mathrm{Dec}$ is a quasi-inverse of $W_E(A_{\infty})\otimes_{A_{\mathrm{mv},E}}-$.
	\end{proof}
	
	Now we add the action of $\mathcal{O}_K^{\times}$.
	
	\begin{cor}\label{descent of f.p. phi_q,O_K*-modules}
		The functor
		\begin{align*}
			\mathrm{Dec}:\ \modetfp_{\varphi_q,\mathcal{O}_K^{\times}}(W_E(A_{\infty}))&\to \modetfp_{\varphi_q,\mathcal{O}_K^{\times}}(A_{\mathrm{mv},E})\\
			M_{\infty}&\mapsto \mathrm{Dec}(M_{\infty})
		\end{align*}
		is an equivalence of categories with a quasi-inverse given by $W_E(A_{\infty})\otimes_{A_{\mathrm{mv},E}}-$. Moreover, the functor $\mathrm{Dec}$ preserves short exact sequences
	\end{cor}
	\begin{proof}
		Let $M_{\infty}$ be a finitely presented étale $(\varphi_q,\mathcal{O}_K^{\times})$-module over $W_E(A_{\infty})$. Let $a\in\mathcal{O}_K^{\times}$, since $W_E(A_{\infty})\otimes_{a,W_E(A_{\infty})}M_{\infty}$ is also a finitely presented étale $\varphi_q$-module over $W_E(A_{\infty})$, by Theorem \ref{descent of fin pre etale phi_q-mod}, the natural isomorphism
		\[W_E(A_{\infty})\otimes_{a,W_E(A_{\infty})}M_{\infty}\xrightarrow{\sim} M_{\infty},\ x\otimes m\mapsto xa(m)\]
		descends to an isomorphism of $\varphi_q$-modules
		\[\mathrm{Dec}\Big(W_E(A_{\infty})\otimes_{a,W_E(A_{\infty})}M_{\infty}\Big)\xrightarrow{\sim} \mathrm{Dec}(M_{\infty}).\]
		Let $f_{\infty}: F_{\infty}\twoheadrightarrow M_{\infty}$ be a surjection of finitely presented étale $\varphi_q$-modules, where $F_{\infty}=W_E(A_{\infty})\otimes_{A_{\mathrm{mv},E}} F$ for some finite free étale $\varphi_q$-module $F$ over $A_{\mathrm{mv},E}$. Then we get a surjection
		\begin{align*}
			W_E(A_{\infty})\otimes_{A_{\mathrm{mv},E}}\left(A_{\mathrm{mv},E}\otimes_{a,A_{\mathrm{mv},E}} F\right)
			&\cong W_E(A_{\infty})\otimes_{a,W_E(A_{\infty})}F_{\infty}\xdbheadrightarrow{\mathrm{id}\otimes f_{\infty}}W_E(A_{\infty})\otimes_{a,W_E(A_{\infty})}M_{\infty},
		\end{align*}
		which easily implies that
		\[\mathrm{Dec}\left(W_E(A_{\infty})\otimes_{a,W_E(A_{\infty})}M_{\infty}\right)\cong A_{\mathrm{mv},E}\otimes_{a,A_{\mathrm{mv},E}}\mathrm{Dec}(M_{\infty}),\]
		hence we get an endomorphism 
		\[a: \mathrm{Dec}(M_{\infty})\xrightarrow{m\mapsto 1\otimes m}A_{\mathrm{mv},E}\otimes_{a,A_{\mathrm{mv},E}}\mathrm{Dec}(M_{\infty})\cong \mathrm{Dec}(W_E(A_{\infty})\otimes_{a}M_{\infty})\xrightarrow{\sim} \mathrm{Dec}(M_{\infty}).\]
		This gives the semi-linear action of $\mathcal{O}_K^{\times}$ on $\mathrm{Dec}(M_{\infty})$ which makes $\mathrm{Dec}(M_{\infty})$ into a finitely presented étale $(\varphi_q,\mathcal{O}_K^{\times})$-module over $A_{\mathrm{mv},E}$. By the same argument in the proof of Theorem \ref{descent of fin pre etale phi_q-mod}, we have
		\begin{align*}
			\mathrm{Dec}\circ \left(W_E(A_{\infty})\otimes_{A_{\mathrm{mv},E}}-\right)\cong \mathrm{Id},\ W_E(A_{\infty})\otimes_{A_{\mathrm{mv},E}}\left(\mathrm{Dec}(-)\right)\cong \mathrm{Id}, 
		\end{align*}
		which finishes the proof.
	\end{proof}
	
	When $K=\qp$, we have $ \mathbb{F}(\negthinspace(T_{\mathrm{cycl}}^{1/p^{\infty}})\negthinspace)=A_{\infty}$ and $A_{\mathrm{cycl},E}=A_{\mathrm{mv},E}$, hence
	
	\begin{lem}\label{descent of f.p. etale modules from W_E(T_cycl) to A_cycl}
		The functor 
		\begin{align*}
			W_E\left( \mathbb{F}(\negthinspace(T_{\mathrm{cycl}}^{1/p^{\infty}})\negthinspace)\right) \otimes_{A_{\mathrm{cycl},E}}-: \modetfp_{\varphi_q,\mathbb{Z}_p^{\times}}\left(A_{\mathrm{cycl},E}\right)& \to \modetfp_{\varphi_q,\mathbb{Z}_p^{\times}}\left(W_E\left( \mathbb{F}(\negthinspace(T_{\mathrm{cycl}}^{1/p^{\infty}})\negthinspace)\right)\right) \\
			M &\mapsto W_E\left( \mathbb{F}(\negthinspace(T_{\mathrm{cycl}}^{1/p^{\infty}})\negthinspace)\right) \otimes_{A_{\mathrm{cycl},E}}M
		\end{align*}
		is an equivalence of categories.
	\end{lem}

	As a consequence of Proposition \ref{description of finitely presented etale phi_q-module over A_{mv,E}} and Theorem \ref{descent of fin pre etale phi_q-mod}, we have
	\begin{prop}\label{description of finitely presented etale phi_q-module over W_E(A_{infty})}
		Let $M$ be a non-zero finitely presented étale $\varphi_q$-module over  $W_E(A_{\infty})$. Then there exists $1\leq n_1<\dots<n_r\leq \infty$ and $m_1,\dots,m_r\geq 1$ such that
		\[M\cong \bigoplus_{i=1}^r\left(W_E(A_{\infty})/\varpi^{n_i}\right)^{\oplus m_i},\]
		here we use the convention that $W_E(A_{\infty})/\varpi^{\infty}=W_E(A_{\infty})$.
	\end{prop}
	
	\section{Descent from $W_E(A_{\infty}')$ to $W_E(A_{\infty})$}\label{section 4}
	
	The goal of this section is to show that every finitely presented étale $(\varphi_q,(\mathcal{O}_K^{\times})^f)$-module over $W_E(A_{\infty}')$ canonically descends to a finitely presented étale $(\varphi_q,\mathcal{O}_K^{\times})$-module over $W_E(A_{\infty})$.
	
	\subsection{A vanishing result on certain continuous cohomology groups}
	
	Throughout this section, for a topological space $X$ and a topological ring $R$, we denote by $C(X,R)$ the set of continuous maps from $X$ to $R$.
	
	Recall that the perfectoid ring $A_{\infty}'\coloneq\mathbb{F}\left(\negthinspace\left(T_{\mathrm{LT},0}^{1/p^{\infty}}\right)\negthinspace\right)\left\langle\left(\frac{T_{\mathrm{LT},i}}{T_{\mathrm{LT},0}^{p^{i}}}\right)^{\pm 1/p^{\infty}}: 1\leq i\leq f-1\right\rangle$ admits an $\mathbb{F}$-linear continuous $(\varphi_q,(\mathcal{O}_K^{\times})^f)$-action (see \cite[$\S$2.3 (28), $\S$2.5 (47)]{breuil2023multivariable}, where the variable $T_{K,i}$ of \textit{loc. cit.} is denoted by $T_{\mathrm{LT},i}$ here) given by \begin{align}\label{phi_q,O_K*-action on A_{infty}'}
			\begin{cases}
				\varphi_q(T_{\mathrm{LT},i})=T_{\mathrm{LT},i}^q, & 0\leq i\leq f-1\\
				\underline{a}(T_{\mathrm{LT},i})=a_{i,\mathrm{LT}}(T_{\mathrm{LT},i}), & \underline{a}=(a_0,\dots,a_{f-1})\in (\mathcal{O}_K^{\times})^f, 0\leq i\leq f-1.
				\end{cases}
	\end{align}
	The topology on $A_{\infty}'$ is defined by the multiplicative norm $|\cdot|: A_{\infty}'\to\mathbb{R}_{\geq 0}$ sending $T_{\mathrm{LT},i}$ to $p^{-1/p^{i}}$.
	
	Let $M$ be a finite projective étale $(\varphi_q,(\mathcal{O}_{K}^{\times})^f)$-module over $A_{\infty}'$, and let $\Delta_1\coloneq\{(a_0,\dots,a_{f-1})\in \mathcal{O}_K^{\times}: \prod_i a_i=1\}$ be a subgroup of $\left(\mathcal{O}_K^{\times}\right)^f$. In this section, we prove that 
	\[H^k_{\mathrm{cont}}(\Delta_1,M)=0, k\geq 1.\]

	We recall the definition of continuous cohomology groups. Let $G$ be a topological group, and let $M$ be a topological abelian group with a continuous $G$-action. The $i$-th \textit{continuous cohomology group} $H^i_{\mathrm{cont}}(G,M)$ is computed by the complex of abelian groups
	\[C_{\mathrm{cont}}^0(G,M)\xrightarrow{d^0} C_{\mathrm{cont}}^1(G,M)\xrightarrow{d^1} C_{\mathrm{cont}}^2(G,M)\to\cdots\]
	where $C_{\mathrm{cont}}^0(G,M)=M$,  $C_{\mathrm{cont}}^i(G,M)=C(G^{i},M)$ for $i\geq 1$, $d^0: M\to C_{\mathrm{cont}}^1(G,M)$ is the map sending $m$ to the function $ g_1\mapsto g_1(m)-m$ and for $i\geq 1$
	\begin{equation}\label{defn of d_n for continuous coh}
		\begin{aligned}
			d^i(f)(g_1,\dots,g_{i+1}):=g_1\cdot f(g_2,\dots,g_{i+1})&+\sum_{j=1}^i(-1)^jf(g_1,\dots,g_jg_{j+1},\dots,g_{i+1})\\
			&+(-1)^{i+1}f(g_1,\dots,g_i).
		\end{aligned}
	\end{equation}
	
	\begin{lem}\label{continuous cohomology induces a long exact sequence}
		Let $G$ be a profinite group. Suppose that $0\to M'\xrightarrow{f'} M\xrightarrow{f''} M''\to 0$ is an exact sequence of continuous $G$-modules. If $f''$ admits local sections, i.e., for every element $m''$ of $ M''$, there exists an open neighbourhood $U$ of $m''$ in $M''$ and a continuous map $s: U\to M$ such that $f''\circ s=\mathrm{id}_{U}$, then there is a long exact sequence of abelian groups
		\[0\to M'^G\to M^G\to \cdots \to H^{i-1}_{\mathrm{cont}}(G,M'')\to H^{i}_{\mathrm{cont}}(G,M')\to H^{i}_{\mathrm{cont}}(G,M)\to H^{i}_{\mathrm{cont}}(G,M'')\to\cdots.\]
	\end{lem}
	\begin{proof}
		This is a direct consequence of \cite[Proposition 1.3, Corollary 2.4]{lichtenbaum2009weil}.
	\end{proof}
	
	Recall the following theorem in \cite[$\S$2.5 Theorem 2.5.1]{breuil2023multivariable}:
	
	\begin{thm}\label{descent in char p, using Delta_1}
		The functor $A_{\infty}'\otimes_{m}-: \modet_{\varphi_q,\mathcal{O}_K^{\times}}(A_{\infty})\to \modet_{\varphi_q,\left(\mathcal{O}_K^{\times}\right)^f}(A_{\infty}')$ is an equivalence of categories which is exact and preserves the rank. A quasi-inverse is given by $\left(-\right)^{\Delta_1}$.
	\end{thm}
	
	Let $M$ be a finite projective étale $(\varphi_q,\left(\mathcal{O}_K^{\times}\right)^f)$-module over $A_{\infty}'$ of rank $r$. By Theorem \ref{descent in char p, using Delta_1}, as a $\Delta_1$-module, it is isomorphic to $\left(A_{\infty}'\right)^{\oplus r}$. Therefore, for any $m\geq 1$, there is an isomorphism of abelian groups
	\[H^k_{\mathrm{cont}}(\Delta_1,M)\cong H^k_{\mathrm{cont}}(\Delta_1,A_{\infty}')^{\oplus r},\quad k\geq 1.\]
	Thus to show $H^k_{\mathrm{cont}}(\Delta_1,M)=0$, it suffices to show:
	\begin{thm}\label{vanishing of continuous cohomology of A_{infty}'}
		For $k\geq 1$, we have
		\[H^k_{\mathrm{cont}}(\Delta_1,A_{\infty}')=0.\]
	\end{thm}
	
	\begin{proof}
		Let $\zokgen\coloneq \mathrm{Spa}(A_{\infty},A_{\infty}^{\circ})$ and $U_{\underline{n}_0}\coloneq \mathrm{Spa}(A_{\infty}',\left(A_{\infty}'\right)^{\circ})$ which are affinoid perfectoid spaces. By \cite[Proposition 2.4.4]{breuil2023multivariable}, there is a morphism $U_{\underline{n}_0}\xrightarrow{m}\zokgen$ which is a pro-étale $\Delta_1$-torsor. In particular, it is a covering in the pro-étale site $\left(\zokgen\right)_{\mathrm{pro\acute{e}t}}$. We also denote the ring homomorphism $A_{\infty}\hookrightarrow A_{\infty}'$ by $m$.

		Recall that for any perfectoid space $X$ and any profinite group $G$, the fibre product $\underline{G}\times X$ of pro-étale sheaves is represented by the perfectoid space $X \times G\coloneq\varprojlim\limits_{H=\text{ open subgroups of } G} X\times G/H$, where $\underline{G}$ is the pro-étale sheaf sending $Y$ to the set of continuous maps from $Y$ to $G$, and $X\times G/H$ is a finite disjoint union of $G/H$ copies of $X$ (\cite[$\S$9.3]{WS2020berkeley}, \cite[Lemma 10.13]{scholze2022etale}). In particular, for $n\geq 1$ we have the perfectoid space $ U_{\underline{n_0}}\times\Delta_1^n$. In fact, we have
		\begin{equation}\label{affinoid ring of Delta_1^n times U_n}
			\begin{aligned}
				U_{\underline{n_0}}\times\Delta_1^n
				&=\mathrm{Spa}\left(C(\Delta_1^n,A_{\infty}'),C(\Delta_1^n,\left(A_{\infty}'\right)^{\circ})\right),\quad n\geq 1.
			\end{aligned}
		\end{equation}
		
		Claim: for $n\geq 1$, there is an isomorphism of affinoid perfectoid spaces.
		\begin{align*}
			U_{\underline{n_0}}\times\Delta_1^n&\xrightarrow{\sim} U_{\underline{n_0}}^{\times (n+1)}\\
			\left(x,g_1,\dots,g_n\right)&\mapsto \left(x,g_1(x),\dots,\left(\prod_{i=1}^{n-1}g_i\right) (x),\left(\prod_{i=1}^ng_i\right) (x)\right),
		\end{align*}
		 where $U_{\underline{n_0}}^{\times (n+1)}\coloneq U_{\underline{n_0}}\times_{\zokgen}U_{\underline{n_0}}\times_{\zokgen}\cdots U_{\underline{n_0}}$ is the fibre product of $n+1$ copies of $U_{\underline{n_0}}$ for $n\geq 0$. We shall prove the claim by induction on $n\geq 1$. The case $n=1$ is proved by the first paragraph of the proof of \cite[Proposition 4.3.2]{weinstein2017gal}. For $n\geq 2$, then there is an isomorphism of perfectoid spaces
		\[U_{\underline{n_0}}\times\Delta_1^n\cong \left( U_{\underline{n_0}}\times\Delta_1^{n-1}\right)\times\Delta_1\cong U_{\underline{n_0}}^{\times n}\times\Delta_1\cong U_{\underline{n_0}}^{\times (n-1)}\times_{\zokgen}\left( U_{\underline{n_0}}\times\Delta_1\right)\cong U_{\underline{n_0}}^{\times (n+1)},\]
		which proves the claim. As a consequence, using (\ref{affinoid ring of Delta_1^n times U_n}), for $n\geq 1$, there is an isomorphism 
		\begin{equation}\label{terms of complex of continuous cohomology}
			\begin{aligned}
				t^n: \mathcal{O}_{U_{\underline{n_0}}^{\times (n+1)}}\left(U_{\underline{n_0}}^{\times (n+1)}\right)&\xrightarrow{\sim}C(\Delta_1^n,A_{\infty}')\\
				a_0\otimes\cdots \otimes a_n&\mapsto \left((g_1,\dots,g_n)\mapsto a_0\cdot g_1(a_1)\cdot (g_1g_2)(a_2)\cdot\cdots\cdot \left(\prod_{i=1}^{n}g_i\right)(a_n)\right),
			\end{aligned}
		\end{equation}
		where we use the isomorphism
		\[\mathcal{O}_{U_{\underline{n_0}}^{\times (n+1)}}(U_{\underline{n_0}}^{\times (n+1)})\cong A_{\infty}'\widehat{\otimes}_{A_{\infty}}\cdots \widehat{\otimes}_{A_{\infty}}A_{\infty}'\eqcolon \widehat{\otimes}^{n+1}_{A_{\infty}}A_{\infty}',\quad n\geq 1.\]
		We let $t^0$ be the isomorphism $ \mathcal{O}_{U_{\underline{n_0}}^{\times 1}}(U_{\underline{n_0}}^{\times 1})=\mathcal{O}_{U_{\underline{n_0}}}(U_{\underline{n_0}})\cong A_{\infty}'\eqcolon \widehat{\otimes}^{1}_{A_{\infty}}A_{\infty}'$.
		
		By \cite[Proposition 8.2.8]{WS2020berkeley}, the presheaf $\mathcal{O}: X\mapsto \mathcal{O}_X(X)$ is a sheaf on $\left(\zokgen\right)_{\mathrm{pro\acute{e}t}}$, and for $i\geq 1$ $H^i\left(\left(\zokgen\right)_{\mathrm{pro\acute{e}t}},\mathcal{O}\right)=0$. Since $U_{\underline{n_0}}\to \zokgen$ is a pro-étale covering, by \cite[\href{https://stacks.math.columbia.edu/tag/03OW}{Tag 03OW}]{stacks-project}, there is a spectral sequence
		\[E_2^{k,l}=\check{H}^k\left(\mathcal{U},\underline{H}^l(\mathcal{O})\right)\Longrightarrow H^{k+l}\left(\left(\zokgen\right)_{\mathrm{pro\acute{e}t}},\mathcal{O}\right),\]
		where $\mathcal{U}=\{U_{\underline{n_0}}\to \zokgen\}$, $\underline{H}^l(\mathcal{O})$ is the pre-sheaf sending $X$ to $H^l(X_{\mathrm{pro\acute{e}t}},\mathcal{O})$ on $\left(\zokgen\right)_{\mathrm{pro\acute{e}t}}$, and $\check{H}^k\left(\mathcal{U},\underline{H}^l(\mathcal{O})\right)$ is the $k$-th \v{C}ech cohomology group associated to the covering $\mathcal{U}$ and the presheaf $\underline{H}^l(\mathcal{O})$. Since $U_{\underline{n_0}}^{\times n}$ is affinoid perfectoid for every $n\geq 1$, by \cite[Proposition 8.2.8]{WS2020berkeley}, we have
		\[\underline{H}^l(\mathcal{O})(U_{\underline{n}_0}^{\times n})=H^l\left(\left(U_{\underline{n}_0}^{\times n}\right)_{\mathrm{pro\acute{e}t}},\mathcal{O}\right)=0,\quad l\geq 1\]
		hence
		\[\check{H}^k\left(\mathcal{U},\underline{H}^l(\mathcal{O})\right)=0,\quad l\geq 1, k\geq 0.\]
		Therefore, the spectral sequence collapses at the second page, and induces an isomorphism
		\[\check{H}^k\left(\mathcal{U},\mathcal{O}\right)\cong H^k\left(\left(\zokgen\right)_{\mathrm{pro\acute{e}t}},\mathcal{O}\right)=0,\quad k\geq 1.\]
		On the other hand, by the definition of \v{C}ech cohomology groups, $\check{H}^k\left(\mathcal{U},\mathcal{O}\right)$ is computed by the complex
		\begin{align}\label{complex of Cech cohomology coincides with continuous cohomology}
			\mathcal{O}_{U_{\underline{n_0}}^{\times 1}}(U_{\underline{n_0}}^{\times 1})\xrightarrow{\partial^{0}}\mathcal{O}_{U_{\underline{n_0}}^{\times 2}}(U_{\underline{n_0}}^{\times 2})\xrightarrow{\partial^{1}}\mathcal{O}_{U_{\underline{n_0}}^{\times 3}}(U_{\underline{n_0}
			}^{\times 3})\xrightarrow{\partial^{2}}\cdots.
		\end{align}
		To be explicit, for $n\geq 0$, using the isomorphism $\mathcal{O}_{U_{\underline{n_0}}^{\times (n+1)}}(U_{\underline{n_0}
		}^{\times (n+1)})\cong \widehat{\otimes}^{n+1}_{A_{\infty}}A_{\infty}'$, $\partial^{n}$ is given by
		\begin{align*}
			\partial^{n}:\quad  \widehat{\otimes}^{n+1}_{A_{\infty}}A_{\infty}'&\to \widehat{\otimes}^{n+2}_{A_{\infty}}A_{\infty}'\\
			a_0\otimes\cdots\otimes a_n&\mapsto \sum_{i=0}^{n+1}(-1)^{i}a_0\otimes\cdots\otimes a_{i-1}\otimes 1\otimes a_i\otimes\cdots\otimes a_{n}.
		\end{align*}
		Then for $n\geq 0$, $a_0\otimes\cdots\otimes a_n\in \widehat{\otimes}^{n+1}_{A_{\infty}}A_{\infty}'$ and $(g_1,\dots,g_{n+1})\in\Delta_1^{n+1}$, using (\ref{defn of d_n for continuous coh}), (\ref{terms of complex of continuous cohomology}) and (\ref{complex of Cech cohomology coincides with continuous cohomology}), we have
		\begin{align*}
			d^{n}\Big( t^{n}(a_0\otimes\cdots\otimes a_n)\Big)\big(g_1,\dots,g_{n+1}\big)
			&=g_1\Big(a_0\cdot g_2(a_1)\cdot(g_2g_3)(a_2)\cdot\cdots\cdot\left(g_2\cdots g_{n+1}\right)(a_n)\Big)\\
			&\quad +\sum_{l=1}^{n}(-1)^la_0\cdot g_1(a_1)\cdot\cdots\cdot (g_1g_2\cdots g_{i-1})(a_{i-1})\\
			&\qquad\cdot (g_1g_2\cdots g_{i}g_{i+1})(a_{i})\cdot\cdots\cdot(g_1\cdots g_{n+1})(a_n)\\
			&\quad +(-1)^{n+1}a_0\cdot g_1(a_1)\cdot\cdots\cdot \left(g_1\cdots g_n\right)(a_n)\\
			&=\prod_{i=0}^{n}\left(\left(\prod_{j=1}^{i+1}g_j\right)(a_i)\right)\\
			&\quad+\sum_{l=1}^{n}(-1)^la_0\prod_{i=1}^{l-1}\left(\left(\prod_{j=1}^{i}g_j\right)(a_i)\right)\prod_{i=l}^{n}\left(\left(\prod_{j=1}^{i+1}g_j\right)(a_i)\right)\\
			&\quad +(-1)^{n+1}a_0\prod_{i=1}^{n}\left(\left(\prod_{j=1}^{i}g_j\right)(a_i)\right)
		\end{align*}
		and
		\begin{align*}
			t^{n+1}\Big(\partial^{n}(a_0\otimes\cdots\otimes a_n)\Big)\big(g_1,\dots,g_{n+1}\big)&=\sum_{l=0}^{n+1}(-1)^{l}t^{n+1}\left(a_0\otimes\cdots\otimes 1\otimes a_{l}\otimes \cdots \otimes a_n\right)(g_1,\dots,g_{n+1})\\
			&=t^{n+1}\left(1\otimes a_0\otimes \cdots \otimes a_n\right)(g_1,\dots,g_{n+1})\\
			&\quad+\sum_{l=1}^{n}(-1)^{l}t^{n+1}\left(a_0\otimes\cdots\otimes 1\otimes a_{l}\otimes \cdots \otimes a_n\right)(g_1,\dots,g_{n+1})\\
			&\quad +(-1)^{n+1}t^{n+1}\left(a_0\otimes\cdots\otimes a_n\otimes 1\right)(g_1,\dots,g_{n+1})\\
			&=\prod_{i=0}^{n}\left(\left(\prod_{j=1}^{i+1}g_j\right)(a_i)\right)\\
			&\quad+\sum_{l=1}^{n}(-1)^la_0\prod_{i=1}^{l-1}\left(\left(\prod_{j=1}^{i}g_j\right)(a_i)\right)\prod_{i=l}^{n}\left(\left(\prod_{j=1}^{i+1}g_j\right)(a_i)\right)\\
			&\quad +(-1)^{n+1}a_0\prod_{i=1}^{n}\left(\left(\prod_{j=1}^{i}g_j\right)(a_i)\right).
		\end{align*}
		We see that $t^{n+1}\circ\partial^{n}=d^{n}\circ t^{n}$, hence there is a commutative diagram of complexes 
		\begin{equation*}
			\begin{tikzcd}
				\dots \arrow[r] & \widehat{\otimes}^{n}_{A_{\infty}}A_{\infty}' \arrow[r,"\partial^{n-1}"]\arrow[d,"t^{n-1}","\wr"'] & \widehat{\otimes}^{n+1}_{A_{\infty}}A_{\infty}' \arrow[r,"\partial^{n}"]\arrow[d,"t^{n}","\wr"'] &\widehat{\otimes}^{n+2}_{A_{\infty}}A_{\infty}' \arrow[r,"\partial^{n+1}"]\arrow[d,"t^{n+1}","\wr"']&\dots\\
				\dots\arrow[r]& C_{\mathrm{cont}}^{n-1}(\Delta_1,A_{\infty}')\arrow[r,"d^{n-1}"]&C_{\mathrm{cont}}^{n}(\Delta_1,A_{\infty}')\arrow[r,"d^{n}"]&C_{\mathrm{cont}}^{n+1}(\Delta_1,A_{\infty}')\arrow[r,"d^{n+1}"]&\dots.
			\end{tikzcd}
		\end{equation*}
		Thus $H^k_{\mathrm{cont}}(\Delta_1,A_{\infty}')\cong \check{H}^k\left(\mathcal{U},\mathcal{O}\right)=0$ for $k\geq 1$.
	\end{proof}
	
	\begin{cor}\label{vanishing of continuous cohomology of etale modules}
		Let $M$ be a finite projective étale $(\varphi_q,\left(\mathcal{O}_K^{\times}\right)^f)$-module over $A_{\infty}'$. Then for $k\geq 1$,
		\[H^k_{\mathrm{cont}}(\Delta_1,M)=0.\]
	\end{cor}
	
	As $W_E(A_{\infty})$, we equip $W_E(A_{\infty}')\coloneq \mathcal{O}_E\otimes_{W(\mathbb{F})}W(A_{\infty}')$ with the weak topology so that the $(\varphi_q,(\mathcal{O}_K^{\times})^f)$-action on $W_E(A_{\infty}')$ is continuous. The ring homomorphism $m:A_{\infty}\hookrightarrow A_{\infty}'$ lifts to a map $m:W_E(A_{\infty})\hookrightarrow W_E(A_{\infty}')$ commuting with the continuous $(\varphi_q,(\mathcal{O}_K^{\times})^f)$-actions, where $(\mathcal{O}_K^{\times})^f$ acts on $W_E(A_{\infty})$ via the map $(\mathcal{O}_K^{\times})^f\to \mathcal{O}_K^{\times}, (a_0,\dots,a_{f-1})\mapsto \prod_{i=0}^{f-1}a_i$.
	
	\begin{cor}\label{vanishing of continuous cohomology of W_E(A_{infty}')}
		For $k\geq 1$, we have
		\[H^k_{\mathrm{cont}}(\Delta_1,W_E(A_{\infty}'))=0.\]
	\end{cor}
	
	\begin{proof}
		Consider the short exact sequence of continuous $\Delta_1$-modules
		\[0\to W_E(A_{\infty}')\xrightarrow{\varpi}W_E(A_{\infty}')\to A_{\infty}'\to 0,\]
		since Teichmüller lifts give a continuous section of the projection $W_E(A_{\infty}')\to A_{\infty}'$, Lemma \ref{continuous cohomology induces a long exact sequence} and Theorem \ref{vanishing of continuous cohomology of A_{infty}'} implies that the map
		\begin{align}\label{multiplication by pi}
			H^k_{\mathrm{cont}}(\Delta_1,W_E(A_{\infty}'))\xrightarrow{\varpi}H^k_{\mathrm{cont}}(\Delta_1,W_E(A_{\infty}'))
		\end{align}
		is surjective for every $k\geq 1$. For any continuous cocycle $c\in Z^k_{\mathrm{cont}}(\Delta_1,W_E(A_{\infty}'))$, the surjectivity of (\ref{multiplication by pi}) implies that there exists $c_0'\in C^{k-1}_{\mathrm{cont}}(\Delta_1,W_E(A_{\infty}'))$ and $c_0\in Z^{k}_{\mathrm{cont}}(\Delta_1,W_E(A_{\infty}'))$ such that
		\[c=d^{k-1}(c_0')+\varpi\cdot c_0.\]
		By induction on $n\geq 0$, we can show that for any cocycle $c_n\in Z^{k}_{\mathrm{cont}}(\Delta_1,W_E(A_{\infty}'))$, there exists $c_{n+1}'\in C^{k-1}_{\mathrm{cont}}(\Delta_1,W_E(A_{\infty}'))$ and $c_{n+1}\in Z^{k}_{\mathrm{cont}}(\Delta_1,W_E(A_{\infty}'))$ such that
		\[c_n=d^{k-1}(c_{n+1}')+\varpi\cdot c_{n+1}, \quad n\geq 0.\]
		It is easy to check that $c'\coloneq \sum_{i=0}^{\infty}\varpi^{i}c_i'$ is a continuous cocycle, and $c=d^{k-1}(c')$.
		Therefore,
		\[H^k_{\mathrm{cont}}(\Delta_1,W_E(A_{\infty}'))=0,\quad k\geq 1.\]
	\end{proof}
	
	\subsection{Descent for finite projective modules}
	
	The goal of this section is to show that every finite projective étale $(\varphi_q,(\mathcal{O}_K^{\times})^f)$-module over $W_E(A_{\infty}')$ canonically descends to a finite projective étale $(\varphi_q,\mathcal{O}_K^{\times})$-module over $W_E(A_{\infty})$. 
	
	Note that $A_{\infty}$ and $A_{\infty}'$ are isomorphic as perfectoid algebras, hence the topological rings $W_E(A_{\infty})$ and $W_E(A_{\infty}')$ are isomorphic, therefore, by Lemma \ref{description of finte projective modules}, every finite projective module over $W_E(A_{\infty}')/\varpi^n$ is a finite free $W_E(A_{\infty}')/\varpi^n$-module for $1\leq n\leq+\infty$.

	Recall that every element of the ring $W_E(A_{\infty}')$ can be uniquely written as
	\[\sum_{i=0}^{\infty}[x_i]\varpi^{i},\quad x_i\in A_{\infty}'.\]
	Thus for $n\geq 2$, there exists a continuous map of topological spaces
	\begin{equation}\label{s_n}
		\begin{aligned}
			s_n: W_E(A_{\infty}')/\varpi^{n-1}\to W_E(A_{\infty}')/\varpi^{n},\ 
			\sum_{i=0}^{n-2}[x_i]\varpi^{i}\mapsto \sum_{i=0}^{n-2}[x_i]\varpi^{i},
		\end{aligned}
	\end{equation}
	where $W_E(A_{\infty}')/\varpi^{n-1}$ and $W_E(A_{\infty}')/\varpi^{n}$ are endowed with the product topology induced by
	\[\left(A_{\infty}'\right)^{n}\xrightarrow{\sim}W_E(A_{\infty}')/\varpi^{n},\ (x_i)_{0\leq i\leq n-1}\mapsto \sum_{i=0}^{n-1}[x_i]\varpi^{i}.\]
	
	\begin{lem}\label{A{infty} to A{infty}'}
		Let $M$ be a finite projective étale $(\varphi_q,(\mathcal{O}_K^{\times})^f)$-module over $W_E(A_{\infty}')/\varpi^n$ of rank $r$, where $n\geq 1$ is an integer. Then $M^{\Delta_1}$ is a finite projective étale $(\varphi_q,\mathcal{O}_K^{\times})$-module over $W_E(A_{\infty})/\varpi^n$ of rank $r$, and there is a natural isomorphism
		\[W_E(A_{\infty}')/\varpi^n\otimes_{W_E(A_{\infty})/\varpi^n}M^{\Delta_1}\cong M.\]
	\end{lem}
	\begin{proof}
		We proceed by induction on $n$. The case $n=1$ is Theorem \ref{descent in char p, using Delta_1}. For $n\geq 2$, consider the following short exact sequence of topological $\Delta_1$-modules
		\[0\to \varpi^{n-1}M\to M \to M/\varpi^{n-1}\to 0.\]
		Then $\varpi^{n-1}M$ and $M/\varpi^{n-1}$ are finite projective étale $(\varphi_q,(\mathcal{O}_K^{\times})^f)$-modules over $A_{\infty}'$ and $W_E(A_{\infty}')/\varpi^{n-1}$ of rank $r$ respectively. With a chosen basis of $M$, we can use (\ref{s_n}) to construct a topological section of the quotient map $M\twoheadrightarrow M/\varpi^{n-1}$. Hence we can apply Lemma \ref{continuous cohomology induces a long exact sequence} to deduce the following long exact sequence
		\begin{align}\label{long exact sequence attached to continuous coh grps}
			0\to \left(\varpi^{n-1}M\right)^{\Delta_1}\to M^{\Delta_1}\to \left(M/\varpi^{n-1}\right)^{\Delta_1}\to H^1_{\mathrm{cont}}(\Delta_1,\varpi^{n-1}M)\to\cdots.
		\end{align}
		Corollary \ref{vanishing of continuous cohomology of etale modules} ensures that $H^1_{\mathrm{cont}}(\Delta_1,\varpi^{n-1}M)=0$. By the induction hypothesis, $\left(M/\varpi^{n-1}\right)^{\Delta_1}$ is a finite projective étale $(\varphi_q,\mathcal{O}_K^{\times})$-module over $W_E(A_{\infty})/\varpi^{n-1}$. In particular, we can take $e_1',\dots,e_r'\in M^{\Delta_1}\subset M$ such that the image of $e_1',\dots,e_r'$ in $\left(M/\varpi^{n-1}\right)^{\Delta_1}$ forms a $W_E(A_{\infty})/\varpi^{n-1}$-basis of $\left(M/\varpi^{n-1}\right)^{\Delta_1}$. By the induction hypothesis, the image of $e_1',\dots,e_r'$ under the surjection $M\twoheadrightarrow M/\varpi^{n-1}$ generates $M/\varpi^{n-1}$ as a $W_E(A_{\infty}')/\varpi^{n-1}$-module. Hence $M=\sum_{i=1}^rW_E(A_{\infty}')/\varpi^n\cdot e_i'+\varpi^{n-1}M.$ By Nakayama's lemma, $e_1',\dots,e_r'$ generate $M$ over $W_E(A_{\infty}')/\varpi^n$. As $M$ is a free $W_E(A_{\infty}')/\varpi^n$-module of rank $r$, $\{e_1',\dots,e_r'\}$ must be a $W_E(A_{\infty}')/\varpi^n$-basis of $M$, i.e.  $M=\bigoplus_{i=1}^rW_E(A_{\infty}')/\varpi^n\cdot e_i'$. By our choice for $e_i'$, we have
		\[M^{\Delta_1}=\bigoplus_{i=1}^r\left(W_E(A_{\infty}')/\varpi^n\right)^{\Delta_1}\cdot e_i'=\bigoplus_{i=1}^rW_E(A_{\infty})/\varpi^n\cdot e_i'.\]
		Thus $M^{\Delta_1}$ is a free $W_E(A_{\infty})/\varpi^n$-module of rank $r$. Moreover, $M\cong W_E(A_{\infty}')/\varpi^n\otimes_{W_E(A_{\infty})/\varpi^n}M^{\Delta_1}$. 
		
		The endomorphism $\varphi_q$ of $M$ (resp. $\varpi^{n-1}M$, $M/\varpi^{n-1}$) commutes with the action of $(\mathcal{O}_K^{\times})^{f}$, hence the restriction of $\varphi_q$ on $M^{\Delta_1}\subseteq M$ (resp. $\left(\varpi^{n-1}M\right)^{\Delta_1}\subseteq \varpi^{n-1}M$, $\left(M/\varpi^{n-1}\right)^{\Delta_1}\subseteq M/\varpi^{n-1}$) is an endomorphism of $M^{\Delta_1}$ (resp. $\left(\varpi^{n-1}M\right)^{\Delta_1}$, $\left(M/\varpi^{n-1}\right)^{\Delta_1}$), which makes $M^{\Delta_1}$ (resp. $\left(\varpi^{n-1}M\right)^{\Delta_1}$, $\left(M/\varpi^{n-1}\right)^{\Delta_1}$) into a $\varphi_q$-module over $W_E(A_{\infty})/\varpi^{n}$ (resp. $A_{\infty}$, $W_E(A_{\infty})/\varpi^{n-1}$).
		By the induction hypothesis, $\left(\varpi^{n-1}M\right)^{\Delta_1}$ and $\left(M/\varpi^{n-1}\right)^{\Delta_1}$ are étale $\varphi_q$-modules over $A_{\infty}$ and $W_E(A_{\infty})/\varpi^{n-1}$ respectively, hence are étale over $W_E(A_{\infty})/\varpi^{n}$. Then by a standard argument using five lemma, $M^{\Delta_1}$ is étale.
	\end{proof}
	
	
	
	
	Now we are able to prove the following theorem:
	\begin{thm}\label{descent from W_E(A{infty}') to W_E(A{infty})}
		The functor
		\begin{align*} \modet_{\varphi_q,\mathcal{O}_K^{\times}}(W_E(A_{\infty}))&\to\modet_{\varphi_q,(\mathcal{O}_K^{\times})^f}(W_E(A_{\infty}'))\\
			M& \mapsto W_E(A_{\infty}')\otimes_{W_E(A_{\infty})}M
		\end{align*}
		is an equivalence of categories with a quasi-inverse given by $(-)^{\Delta_1}$.
	\end{thm}
	\begin{proof}
		It is easy to check that $\left(-\right)^{\Delta_1}\circ \left(W_E(A_{\infty}')\otimes_{W_E(A_{\infty})}-\right)\cong \mathrm{Id}_{\modet_{\varphi_q,\mathcal{O}_K^{\times}}(W_E(A_{\infty}))}$. We need to show that for any finite projective étale $(\varphi_q,(\mathcal{O}_K^{\times})^f)$-module $M$ over $W_E(A_{\infty}')$ of rank $r$, $M^{\Delta_1}$ is a finite projective étale $(\varphi_q,\mathcal{O}_K^{\times})$-module over $W_E(A_{\infty})$ of rank $r$, and there is a canonical isomorphism
		\[W_E(A_{\infty}')\otimes_{W_E(A_{\infty})}M^{\Delta_1}\cong M.\]
		Since $M$ is a finite projective étale $(\varphi_q,(\mathcal{O}_K^{\times})^f)$-module over $W_E(A_{\infty}')$, we have $M\cong \varprojlim_n\left(M/\varpi^n\right)$,
		and $M/\varpi^n$ is a finite projective étale $(\varphi_q,\mathcal{O}_K^{\times})$-module over $W_E(A_{\infty}')/\varpi^n$ of rank $r$ for every $n$. Thus by Lemma \ref{A{infty} to A{infty}'}, $\left(M/\varpi^n\right)^{\Delta_1}$ is a finite projective étale $(\varphi_q,\mathcal{O}_K^{\times})$-module over $W_E(A_{\infty})/\varpi^n$ of rank $r$. Moreover, there is a natural injection of abelian groups $\varprojlim_n\left(M/\varpi^n\right)^{\Delta_1}\hookrightarrow \varprojlim_n M/\varpi^n\cong M$, and the image of $\varprojlim_n\left(M/\varpi^n\right)^{\Delta_1}$ is exactly $M^{\Delta_1}$. Hence we can identify  $\varprojlim_n\left(M/\varpi^n\right)^{\Delta_1}$ with $M^{\Delta_1}$. 
		
		The proof of Lemma \ref{A{infty} to A{infty}'} shows that the transition map $\left(M/\varpi^n\right)^{\Delta_1}\to \left(M/\varpi^{n-1}\right)^{\Delta_1}$
		is surjective and can be identified with $\left(M/\varpi^{n}\right)^{\Delta_1}\twoheadrightarrow\left(M/\varpi^{n}\right)^{\Delta_1}/\varpi^{n-1}$. Since $\left(M/\varpi^n\right)^{\Delta_1}$ is a free $W_E(A_{\infty})/\varpi^n$-module of rank $r$ for every $n\geq 1$, the inverse limit $\varprojlim_n\left(M/\varpi^n\right)^{\Delta_1}$ is a free $W_E(A_{\infty})$-module of rank $r$. 
		
		The endomorphism $\varphi_q: W_E(A_{\infty})\to W_E(A_{\infty})$ is an isomorphism, hence
		\begin{align*}
			W_E(A_{\infty})\otimes_{\varphi_q,W_E(A_{\infty})}\varprojlim_n\left(M/\varpi^n\right)^{\Delta_1}&\cong \varprojlim_n\left(W_E(A_{\infty})\otimes_{\varphi_q,W_E(A_{\infty})}\left(M/\varpi^n\right)^{\Delta_1}\right)\\
			&\cong \varprojlim_n\left(W_E(A_{\infty})/\varpi^n\otimes_{\varphi_q,W_E(A_{\infty})/\varpi^n}\left(M/\varpi^n\right)^{\Delta_1}\right)\\
			&\cong \varprojlim_n\left(M/\varpi^n\right)^{\Delta_1},
		\end{align*}
		where the last isomorphism uses the étaleness of $\left(M/\varpi^n\right)^{\Delta_1}$ for $n\geq 1$.
		Hence $\varprojlim_n\left(M/\varpi^n\right)^{\Delta_1}=M^{\Delta_1}$ is a finite projective étale $(\varphi_q,\mathcal{O}_K^{\times})$-module over $W_E(A_{\infty})$ of rank $r$.
		
		It remains to show that the natural injection $W_{E}(A_{\infty}')\otimes_{W_E(A_{\infty})}M^{\Delta_1}\hookrightarrow M$ is an isomorphism. For $m_1\in \left(M/\varpi\right)^{\Delta_1}$, since every transition map of the inverse system $\left\{\left(M/\varpi^n\right)^{\Delta_1}: n\geq 1\right\}$ is surjective, 
		the natural projection 
		\[\mathrm{pr}_1: M^{\Delta_1}= \varprojlim_n\left(M/\varpi^n\right)^{\Delta_1}\to  \left(M/\varpi\right)^{\Delta_1}\]
		must be surjective.
		On the other hand, applying $\left(-\right)^{\Delta_1}$ to the short exact sequence $0\to \varpi M\to M\to M/\varpi\to 0$, we get an exact sequence
		\[0\to \varpi M^{\Delta_1}\to M^{\Delta_1}\xrightarrow{\mathrm{pr}_1} \left(M/\varpi \right)^{\Delta_1},\]
		thus the map $\mathrm{pr}_1$ induces an isomorphism of $A_{\infty}$-modules
		\[M^{\Delta_1}/\varpi\cong \left(M/\varpi\right)^{\Delta_1}.\]
		
		By Lemma \ref{A{infty} to A{infty}'}, $A_{\infty}'\otimes_{A_{\infty}}\left(M/\varpi\right)^{\Delta_1}\cong M/\varpi$, thus the natural map $W_{E}(A_{\infty}')\otimes_{W_E(A_{\infty})}M^{\Delta_1}\hookrightarrow M$ induces an isomorphism
		\[\left(W_{E}(A_{\infty}')\otimes_{W_E(A_{\infty})}M^{\Delta_1}\right)/\varpi\cong A_{\infty}'\otimes_{A_{\infty}}M^{\Delta_1}/\varpi\cong A_{\infty}'\otimes_{A_{\infty}}\left(M/\varpi\right)^{\Delta_1}\cong M/\varpi.\]
		Since $1+\varpi x$ is invertible for every $x\in W_E(A_{\infty}')$,  we see that $\varpi$ is contained in the Jacobson radical of $W_E(A_{\infty}')$, hence by Nakayama's lemma, the natural injection $W_{E}(A_{\infty}')\otimes_{W_E(A_{\infty})}M^{\Delta_1}\hookrightarrow M$ is surjective, thus it is an isomorphism.
	\end{proof}
	
	\begin{cor}\label{vanishing of continuous cohomology of M}
		Let $M$ be a finite projective étale $(\varphi_q,(\mathcal{O}_K^{\times})^{f})$-module over $W_E(A_{\infty}')$, then
		\[H^k_{\mathrm{cont}}\left(\Delta_1,M\right)=0,\quad k\geq 1.\]
	\end{cor}
	\begin{proof}
		By Theorem \ref{descent from W_E(A{infty}') to W_E(A{infty})}, as a $\Delta_1$-module, $M\cong W_E(A_{\infty}')^{\oplus r}$ for some $r\geq 0$. Then the conclusion follows from Corollary \ref{vanishing of continuous cohomology of W_E(A_{infty}')}. 
	\end{proof}
	
	\subsection{Descent for finitely presented modules}
	
	In this section, we show that the functor $(-)^{\Delta_1}$ is an equivalence of categories from $\modetfp_{\varphi_q,(\mathcal{O}_K^{\times})^f}(W_E(A_{\infty}'))$ to $\modetfp_{\varphi_q,\mathcal{O}_K^{\times}}(W_E(A_{\infty}))$ with a quasi-inverse given by $W_E(A_{\infty}')\otimes_{m,W_E(A_{\infty})}-$.
	
	Note that the map
	\begin{align*}
		A_{\infty}=\mathbb{F}\left(\negthinspace\left(T_{\sigma_{0}}^{p^{-\infty}}\right)\negthinspace\right)\left\langle\left(\frac{T_{\sigma_i}}{T_{\sigma_{0}}}\right)^{\pm 1/ p^{\infty}}: i\right\rangle&\longrightarrow A_{\infty}'=\mathbb{F}\left(\negthinspace\left(T_{\mathrm{LT},0}^{1/p^{\infty}}\right)\negthinspace\right)\left\langle\left(\frac{T_{\mathrm{LT},i}}{T_{\mathrm{LT},0}^{p^{i}}}\right)^{\pm 1/p^{\infty}}:  i\right\rangle\\
		T_{\sigma_{i}}&\mapsto T_{\mathrm{LT},i}^{1/p^{i}}
	\end{align*}
	is an isomorphism of perfectoid algebras commuting with $\varphi_q$, and it lifts to an isomorphism of topological rings $W_E(A_{\infty})\xrightarrow{\sim}W_E(A_{\infty}')$ commuting with $\varphi_q$, hence Proposition \ref{description of finitely presented etale phi_q-module over W_E(A_{infty})} implies that
	\begin{prop}\label{description of finitely presented etale phi_q-module over W_E(A_{infty}')}
		Let $M$ be a non-zero finitely presented étale $\varphi_q$-module over  $W_E(A_{\infty}')$. Then there exists $1\leq n_1<\dots<n_r\leq \infty$ and $m_1,\dots,m_r\geq 1$ such that
		\[M\cong \bigoplus_{i=1}^r\left(W_E(A_{\infty}')/\varpi^{n_i}\right)^{\oplus m_i},\]
		here we use the convention that $W_E(A_{\infty}')/\varpi^{\infty}=W_E(A_{\infty}')$.
	\end{prop}
	As a consequence, we can show that
	
	\begin{prop}\label{vanishing of cohomology groups of f.p. etale (phi_q,O_K*)-modules}
		Let $M$ be a finitely presented étale $(\varphi_q,(\mathcal{O}_K^{\times})^f)$-module over  $W_E(A_{\infty}')$. Then
		\[H^k_{\mathrm{cont}}(\Delta_1,M)=0,\quad k\geq 1.\]
	\end{prop}
	\begin{proof}
		If $M$ is zero, there is nothing to prove, so we assume that $M$ is non-zero. We first deal with torsion modules, i.e. modules $M$ such that $\varpi^n M=0$ and $\varpi^{n-1} M\neq 0$ for some $n\geq 1$. We proceed by induction on $n\geq 1$. If $n=1$, Proposition \ref{description of finitely presented etale phi_q-module over W_E(A_{infty}')} implies that $M$ is a finite projective étale $(\varphi_q,(\mathcal{O}_K^{\times})^{f})$-module over $A_{\infty}'$, then applying Corollary \ref{vanishing of continuous cohomology of etale modules}, the conclusion follows. For $n\geq 2$, consider the short exact sequence
		\[0\to M[\varpi^{n-1}]\to M\to M/M[\varpi^{n-1}]\to 0,\]
		By Proposition \ref{description of finitely presented etale phi_q-module over W_E(A_{infty}')}, we can construct a continuous section (of topological spaces) $M/M[\varpi^{n-1}]\to M$ using (\ref{s_n}). Applying Lemma \ref{continuous cohomology induces a long exact sequence}, we get a long exact sequence
		\begin{align*}
			\cdots&\to H^k_{\mathrm{cont}}\left(\Delta_1,M[\varpi^{n-1}]\right)\to H^k_{\mathrm{cont}}\left(\Delta_1,M\right)\to H^k_{\mathrm{cont}}\left(\Delta_1, M/M[\varpi^{n-1}]\right)\to \cdots .
		\end{align*}
		By the induction hypothesis, 
		\[H^k_{\mathrm{cont}}\left(\Delta_1,M[\varpi^{n-1}]\right)= H^k_{\mathrm{cont}}\left(\Delta_1, M/M[\varpi^{n-1}]\right)= 0,\]
		hence for any torsion module $M$,
		\[H^k_{\mathrm{cont}}\left(\Delta_1,M\right)=0,\quad k\geq 1.\]
		
		If $M$ is torsion-free, Proposition \ref{description of finitely presented etale phi_q-module over W_E(A_{infty}')} implies $M$ is finite free, and we apply Corollary \ref{vanishing of continuous cohomology of M} to conclude.
		
		In general, let $M_{\mathrm{tor}}\coloneq \bigcup_{n=1}^{\infty}M[\varpi^n]$ be the torsion submodule of $M$ which is also finitely presented by Proposition \ref{description of finitely presented etale phi_q-module over W_E(A_{infty}')}, then $M/M_{\mathrm{tor}}$ is torsion-free, and there is an exact sequence of finitely presented étale $(\varphi_q,(\mathcal{O}_K^{\times})^f)$-modules over  $W_E(A_{\infty}')$
		\[0\to M_{\mathrm{tor}}\to M\to M/M_{\mathrm{tor}}\to 0.\]
		Again,  it is easy to see that there is a continuous section of $M\twoheadrightarrow M/M_{\mathrm{tor}}$, hence we may apply Lemma \ref{continuous cohomology induces a long exact sequence} to deduce that there is a long exact sequence
		\begin{align*}
			\dots \to H^k_{\mathrm{cont}}\left(\Delta_1,M_{\mathrm{tor}}\right)\to H^k_{\mathrm{cont}}\left(\Delta_1,M\right)\to H^k_{\mathrm{cont}}\left(\Delta_1,M/M_{\mathrm{tor}}\right)\to \dots .
		\end{align*}
		Therefore, for $k\geq 1$, $H^k_{\mathrm{cont}}\left(\Delta_1,M\right)=0$.
	\end{proof}
	
	\begin{thm}\label{descent of f.p. from W_E(A{infty}') to W_E(A{infty})}
		The functor
		\begin{align*} \modetfp_{\varphi_q,\mathcal{O}_K^{\times}}(W_E(A_{\infty}))&\to\modetfp_{\varphi_q,(\mathcal{O}_K^{\times})^f}(W_E(A_{\infty}'))\\
			M& \mapsto W_E(A_{\infty}')\otimes_{W_E(A_{\infty})}M
		\end{align*}
		is an equivalence of categories with a quasi-inverse given by $(-)^{\Delta_1}$.
	\end{thm}
	\begin{proof}
		Using Proposition \ref{description of finitely presented etale phi_q-module over W_E(A_{infty})}, one can easily deduce that for every finitely presented étale $(\varphi_q,\mathcal{O}_K^{\times})$-module $M$ over $W_E(A_{\infty})$, the natural map $M\to \left(W_E(A_{\infty}')\otimes_{m} M\right)^{\Delta_1}$ is an isomorphism. 
		Let $M$ be a finitely presented étale $(\varphi_q,(\mathcal{O}_K^{\times})^{f})$-module over $W_E(A_{\infty}')$, it remains to show that $M^{\Delta_1}$ is a finitely presented étale $(\varphi_q,\mathcal{O}_K^{\times})$-module over $W_E(A_{\infty})$, and the natural homomorphism $W_E(A_{\infty}')\otimes_{m}M^{\Delta_1}\to M$ is an isomorphism.
		
		If $M$ is torsion, suppose that $\varpi^{n-1}M\neq 0$, $\varpi^nM=0$, we prove that $M^{\Delta_1}$ is a finitely presented étale $(\varphi_q,\mathcal{O}_K^{\times})$-module by induction on $n\geq 1$. The case $n=1$ is Theorem \ref{descent in char p, using Delta_1}. For $n\geq 2$, there is a short exact sequence of finitely presented étale $(\varphi_q,\mathcal{O}_K^{\times})$-modules over $W_E(A_{\infty}')$ by Proposition \ref{description of finitely presented etale phi_q-module over W_E(A_{infty}')}
		\[0\to M[\varpi^{n-1}]\to M\to M/M[\varpi^{n-1}]\to 0,\]
		and taking $\Delta_1$-invariants yields an exact sequence
		\[0\to \left(M[\varpi^{n-1}]\right)^{\Delta_1}\to M^{\Delta_1}\to \left(M/M[\varpi^{n-1}]\right)^{\Delta_1}\to H^1_{\mathrm{cont}}\left(\Delta_1,M[\varpi^{n-1}]\right)\]
		Proposition \ref{vanishing of cohomology groups of f.p. etale (phi_q,O_K*)-modules} implies that $H^1_{\mathrm{cont}}\left(\Delta_1,M[\varpi^{n-1}]\right)=0$, hence there is an exact sequence of $W_E(A_{\infty})$-modules
		\begin{align}\label{short exact sequence of Delta_1-invariants}
			0\to \left(M[\varpi^{n-1}]\right)^{\Delta_1}\to M^{\Delta_1}\to \left(M/M[\varpi^{n-1}]\right)^{\Delta_1}\to0.
		\end{align}
		The induction hypothesis implies that $\left(M[\varpi^{n-1}]\right)^{\Delta_1}$ and $\left(M/M[\varpi^{n-1}]\right)^{\Delta_1}$ are finitely presented, hence $M^{\Delta_1}$ is also finitely presented. By a standard argument using five lemma, $M^{\Delta_1}$ is étale. Thus $M^{\Delta_1}$ is a finitely presented étale $(\varphi_q,\mathcal{O}_K^{\times})$-module. Moreover, applying $W_E(A_{\infty}')\otimes_m-$ to the short exact sequence (\ref{short exact sequence of Delta_1-invariants}), there is a commutative diagram with exact rows
		\begin{equation*}
			\begin{tikzcd}[column sep=small]
				& W_E(A_{\infty}')\otimes_m \left(M[\varpi^{n-1}]\right)^{\Delta_1}\arrow[r]\arrow[d,"\wr"]&W_E(A_{\infty}')\otimes_mM^{\Delta_1}\arrow[r]\arrow[d]& W_E(A_{\infty}')\otimes_m \left(M/M[\varpi^{n-1}]\right)^{\Delta_1}\arrow[r]\arrow[d,"\wr"]&0\\
				0\arrow[r]& M[\varpi^{n-1}]\arrow[r]&M\arrow[r]& M/M[\varpi^{n-1}]\arrow[r]&0.
			\end{tikzcd}
		\end{equation*}
		The left and right vertical maps are isomorphisms by the induction hypothesis, then the snake lemma implies that the middle vertical map is an isomorphism.
		
		If $M$ is torsion-free, by Proposition \ref{description of finitely presented etale phi_q-module over W_E(A_{infty}')}, $M$ is free. Then the conclusion follows from Theorem \ref{descent from W_E(A{infty}') to W_E(A{infty})}.
		
		For general $M$, by Proposition \ref{description of finitely presented etale phi_q-module over W_E(A_{infty}')}, there is a short exact sequence of finitely presented étale $(\varphi_q,\mathcal{O}_K^{\times})$-modules over $W_E(A_{\infty}')$
		\[0\to M_{\mathrm{tor}}\to M\to M/M_{\mathrm{tor}}\to 0.\]
		The same arguments as for torsion representations applies here as well, allowing us to conclude. 
	\end{proof}
	
	\section{Review of classical $(\varphi,\Gamma)$-modules}\label{section 5}
	
	For a finite $\mathbb{Z}_p$-algebra $R$ with the $p$-adic topology, we denote the category of finite $R$-modules with a continuous $R$-linear $\gal(\overline{K}/K)$-action by $\rep_{R}(\gal(\overline{K}/K))$, and let $\repfr_{R}(\gal(\overline{K}/K))$ be the full subcategory consisting of continuous $R$-representations of $\gal(\overline{K}/K)$ as finite free $R$-modules. In this section, we recall some results on classical étale $(\varphi,\Gamma)$-modules. The main reference is \cite{Fontaine1990}, \cite{schneider2017galois} and \cite{chiarellotto2014note}.
	
	\subsection{Lubin-Tate and cyclotomic $(\varphi,\Gamma)$-modules}
	
	Recall that in Section \ref{injection in L-T case} we define the ring $A_{\mathrm{LT},K}$ as the $p$-adic completion of $\mathcal{O}_K[\negthinspace[T_{\mathrm{LT}}]\negthinspace][\frac{1}{T_{\mathrm{LT}}}]$, and there is a continuous $(\varphi_q,\mathcal{O}_K^{\times})$-action on $A_{\mathrm{LT},K}$ given by (\ref{phi_q,O_k^* action ass. to Lubin-Tate formal groups}).
	Let $\left(A_{\mathrm{LT},K}^{\mathrm{ur}}\right)^{\wedge}$ be the $p$-adic completion of the maximal unramified extension of $A_{\mathrm{LT},K}$, which admits an action by $\gal(\overline{K}/K)$. The ring $\left(A_{\mathrm{LT},K}^{\mathrm{ur}}\right)^{\wedge}$ is a discrete valuation ring with residue field isomorphic to $\mathbb{F}_q(\negthinspace(T_{\mathrm{LT}})\negthinspace)^{\mathrm{sep}}$ (\cite[Lemma 6.3]{chiarellotto2014note}). Let $\rho$ be a finite type $\mathcal{O}_K$-representation of $\gal(\overline{K}/K)$, the abelian group
	\[D_{\mathrm{LT}}(\rho)\coloneq\left(\left(A_{\mathrm{LT},K}^{\mathrm{ur}}\right)^{\wedge}\otimes_{\mathcal{O}_K}\rho\right)^{\gal(\overline{K}/K_{\infty})}\]
	is a finitely presented étale $(\varphi_q,\mathcal{O}_K^{\times})$-module over $A_{\mathrm{LT},K}$, and
	\[D_{\mathrm{LT}}: \rep_{\mathcal{O}_K}(\gal(\overline{K}/K))\to \modetfp_{\varphi_q,\mathcal{O}_K^{\times}}\left(A_{\mathrm{LT},K}\right)\] 
	is an equivalence of categories (see for instance \cite[Theorem 3.3.10]{schneider2017galois}). Moreover, the restriction $D_{\mathrm{LT}}: \repfr_{\mathcal{O}_K}(\gal(\overline{K}/K))\to \modet_{\varphi_q,\mathcal{O}_K^{\times}}\left(A_{\mathrm{LT},K}\right)$ is also an equivalence of categories which preserves the rank. As a consequence, we have
	\begin{lem}
		The functor
		\begin{align*}
			D_{\mathrm{LT},\sigma_{0}}: \rep_{\mathcal{O}_E}(\gal(\overline{K}/K)) & \to \modetfp_{\varphi_q,\mathcal{O}_K^{\times}} (A_{\mathrm{LT},E,\sigma_{0}})\\
			\rho &\mapsto \left(\left(\mathcal{O}_E\otimes_{\sigma_{0},\mathcal{O}_K}\left(A_{\mathrm{LT},K}^{\mathrm{ur}}\right)^{\wedge}\right)\otimes_{\mathcal{O}_E}\rho\right)^{\gal(\overline{K}/K_{\infty})}
		\end{align*}
		is an equivalence of categories, and its restriction $D_{\mathrm{LT},\sigma_{0}}: \repfr_{\mathcal{O}_E}(\gal(\overline{K}/K)) \to \modet_{\varphi_q,\mathcal{O}_K^{\times}} (A_{\mathrm{LT},E,\sigma_{0}})$ is also an equivalence of categories which preserves the rank.
	\end{lem}

	Let $K(\mu_{p^\infty})$ be the abelian extension of $K$ contained in $\overline{K}$ generated by all $p^n$-th roots of unity for all $n\geq 1$. Since $K/\qp$ is unramified, we have $\gal(K(\mu_{p^\infty})/K)\cong \gal(\qp(\mu_{p^\infty})/\qp)\cong\mathbb{Z}_p^{\times}$.
	
	Recall that there is a continuous $(\varphi,\mathbb{Z}_p^{\times})$-action on $A_{\mathrm{cycl},K}$ given by (\ref{cycltomic phi,Gmma action}). Let $\left(A_{\mathrm{cycl},K}^{\mathrm{ur}}\right)^{\wedge}$ be the $p$-adic completion of the maximal unramified extension of $A_{\mathrm{cycl},K}$, which admits an action by $\gal(\overline{K}/K)$. For $0\leq i\leq f-1$ and any embedding $\sigma_{i}:\mathcal{O}_K\hookrightarrow \mathcal{O}_E$, we put $A_{\mathrm{cycl},E,\sigma_{i}}\coloneq \mathcal{O}_E\otimes_{\sigma_{i},\mathcal{O}_K}A_{\mathrm{cycl},K}$, then
	\[\mathcal{O}_E\otimes_{\mathbb{Z}_p}A_{\mathrm{cycl},K}\cong \prod_{i=0}^{f-1}A_{\mathrm{cycl},E,\sigma_{i}},\]
	and the endomorphism $\varphi$ of $\mathcal{O}_E\otimes_{\mathbb{Z}_p}A_{\mathrm{cycl},K}$ induces a homomorphism $\varphi: A_{\mathrm{cycl},E,\sigma_{i}}\to A_{\mathrm{cycl},E,\sigma_{i-1}}$ for every $i$. Let $\varphi_q\coloneq\varphi^f$ be the endomorphism of $A_{\mathrm{cycl},E,\sigma_{i}}$.
	
	For any finite type $\mathcal{O}_E$-representation $\rho$ of $\gal(\overline{K}/K)$ and $0\leq i\leq f-1$, by \cite{Fontaine1990}, the abelian group
	\[D_{\mathrm{cycl},\sigma_{i}}(\rho)\coloneq A_{\mathrm{cycl},E,\sigma_{i}}\otimes_{\mathcal{O}_E\otimes_{\mathbb{Z}_p}A_{\mathrm{cycl},K}}\left(\left(A_{\mathrm{cycl},K}^{\mathrm{ur}}\right)^{\wedge}\otimes_{\mathbb{Z}_p}\rho\right)^{\gal(\overline{K}/K(\mu_{p^\infty}))}\]
	is a finitely presented étale $(\varphi_q,\mathbb{Z}_p^{\times})$-module over $A_{\mathrm{cycl},E,\sigma_{i}}$. In fact, by \cite{Fontaine1990}, we have
	\begin{lem}\label{D cycl E sigma i}
		For $0\leq i\leq f-1$, the functor
		$D_{\mathrm{cycl},\sigma_{i}}: \rep_{\mathcal{O}_E}(\gal(\overline{K}/K))  \to \modetfp_{\varphi_q,\mathbb{Z}_p^{\times}}\left(A_{\mathrm{cycl},E,\sigma_{i}}\right)$
		is an equivalence of categories, and its restriction $D_{\mathrm{cycl},\sigma_{i}}: \repfr_{\mathcal{O}_E}(\gal(\overline{K}/K)) \to \modet_{\varphi_q,\mathbb{Z}_p^{\times}}\left(A_{\mathrm{cycl},E,\sigma_{i}}\right)$ is also an equivalence of categories which preserves the rank. 
	\end{lem}
	
	Since the trivial map $A_{\mathrm{cycl},E,\sigma_{i}}\to A_{\mathrm{cycl},E,\sigma_{0}}, \sum_{n\in\mathbb{Z}} a_n T_{\mathrm{cycl}}^n\mapsto \sum_{n\in\mathbb{Z}} a_n T_{\mathrm{cycl}}^n$ is an isomorphism commuting with the continuous $(\varphi_q,\mathbb{Z}_p^{\times})$-actions for $1\leq i\leq f-1$, we may simply denote $A_{\mathrm{cycl},E,\sigma_{0}}$ by $A_{\mathrm{cycl},E}=\left(\mathcal{O}_E[\negthinspace[T_{\mathrm{cycl}}]\negthinspace][\frac{1}{T_{\mathrm{cycl}}}]\right)^{\wedge}$ where the completion is the $p$-adic completion, and Lemma \ref{D cycl E sigma i} gives an equivalence of categories
	\[D_{\mathrm{cycl},\sigma_{i}}: \rep_{\mathcal{O}_E}(\gal(\overline{K}/K)) \to \modetfp_{\varphi_q,\mathbb{Z}_p^{\times}}\left(A_{\mathrm{cycl},E}\right)\] 
	for every $0\leq i\leq f-1$. Moreover, let $\varphi$ be the $\mathcal{O}_E$-linear endomorphism of $A_{\mathrm{cycl},E}$ which sends $T_{\mathrm{cycl}}$ to $(1+T_{\mathrm{cycl}})^p-1$, then it is easy to check that
	\begin{align}\label{cycltomic phi,Z_p* module are isomorphic by base change by phi}
		D_{\mathrm{cycl},\sigma_{0}}(\rho)\cong A_{\mathrm{cycl},E}\otimes_{\varphi^i,A_{\mathrm{cycl},E}} D_{\mathrm{cycl},\sigma_{i}}(\rho),\quad i=1,\dots,f-1.
	\end{align}
	
	\subsection{Comparison between $D_{\mathrm{LT},\sigma_{0}}(\rho)$ and $D_{\mathrm{cycl},\sigma_{0}}(\rho)$}
	
	In this section, we show how $D_{\mathrm{LT},\sigma_{0}}(\rho)$ and $D_{\mathrm{cycl},\sigma_{0}}(\rho)$ are related. Let $\mathbb{C}_p$ be the $p$-adic completion of the algebraic closure $\overline{K}$ of $K$, and for any perfectoid field $M$, let $M^{\flat}$ be the tilt of $M$ (\cite[Lemma 3.4]{scholze2012perfectoid}). Let $\widehat{K_{\infty}}$ (resp. $\widehat{K(\mu_{p^{\infty}})}$) be the $p$-adic completion of $K_{\infty}$ (resp. $K(\mu_{p^{\infty}})$), then $\widehat{K_{\infty}}$ and $\widehat{K(\mu_{p^{\infty}})}$ are perfectoid fields. By \cite[4.2.1.Proposition \& 4.3.4.Corollaire]{wintenberger1983corps} and \cite[Theorem 3.7.(ii)]{scholze2012perfectoid}, we have
	\begin{align}\label{tilt is completed perfection}
		\begin{cases}
			\mathbb{F}_q(\negthinspace(T_{\mathrm{LT}}^{p^{-\infty}})\negthinspace)\cong\widehat{K_{\infty}}^{\flat}\cong(\mathbb{C}_p^{\flat})^{\gal(\overline{K}/K_{\infty})}, \\
			\mathbb{F}_q(\negthinspace(T_{\mathrm{cycl}}^{p^{-\infty}})\negthinspace)\cong\widehat{K(\mu_{p^{\infty}})}^{\flat}\cong(\mathbb{C}_p^{\flat})^{\gal(\overline{K}/K(\mu_{p^{\infty}}))}.
		\end{cases}
	\end{align}
	These isomorphisms are compatible with the actions of $\gal(\overline{K}/K)$, where $\gal(\overline{K}/K)$ acts on $\mathbb{F}_q(\negthinspace(T_{\mathrm{LT}}^{p^{-\infty}})\negthinspace)$ (resp. $\mathbb{F}_q(\negthinspace(T_{\mathrm{cycl}}^{p^{-\infty}})\negthinspace)$) via $\gal(\overline{K}/K)\twoheadrightarrow \gal(K_{\infty}/K)\cong \mathcal{O}_K^{\times}$ (resp. $\gal(\overline{K}/K)\twoheadrightarrow \gal(K(\mu_{p^{\infty}})/K)\cong \mathbb{Z}_p^{\times}$). Thus these isomorphisms are compatible with the actions of $\mathcal{O}_K^{\times}$ and of $\mathbb{Z}_p^{\times}$ respectively. As $\gal(\overline{K}/K_{\infty})$ is a closed subgroup of $\gal(\overline{K}/K(\mu_{p^{\infty}}))$, there is an embedding which commutes with the $\mathcal{O}_K^{\times}$-action:
	\begin{align}\label{injection induced by tilting functor}
		\mathbb{F}_q(\negthinspace(T_{\mathrm{cycl}})\negthinspace)\hookrightarrow\mathbb{F}_q(\negthinspace(T_{\mathrm{cycl}}^{p^{-\infty}})\negthinspace)\xrightarrow{\sim}(\mathbb{C}_p^{\flat})^{\gal(\overline{K}/K(\mu_{p^{\infty}}))}\hookrightarrow(\mathbb{C}_p^{\flat})^{\gal(\overline{K}/K_{\infty})}\xrightarrow{\sim}\mathbb{F}_q(\negthinspace(T_{\mathrm{LT}}^{p^{-\infty}})\negthinspace).
	\end{align}
	
	

	We endow $\left(A_{\mathrm{LT},K}^{\mathrm{ur}}\right)^{\wedge}$ (resp. $\left(A_{\mathrm{cycl},K}^{\mathrm{ur}}\right)^{\wedge}$) with the $p$-adic topology, then the action of $\gal(\overline{K}/K_{\infty})$ (resp. $\gal(\overline{K}/K(\mu_{p^\infty}))$) on $\left(A_{\mathrm{LT},K}^{\mathrm{ur}}\right)^{\wedge}$ (resp. $\left(A_{\mathrm{cycl},K}^{\mathrm{ur}}\right)^{\wedge}$) is continuous.
	
	\begin{lem}\label{vanishing result, classical phi,Gamma modules}
		For $n\geq 1$, we have 
		\begin{enumerate}[label=(\alph*), ref=\ref{vanishing result, classical phi,Gamma modules}.(\alph*)]
			\item [(a)]
			$H^1_{\mathrm{cont}}\left(\gal(\overline{K}/K_{\infty}),\left(A_{\mathrm{LT},K}^{\mathrm{ur}}\right)^{\wedge}\right)=0$.\label{vanishing of continuous coh of witt vectors}
			\item [(b)]
			$H^1_{\mathrm{cont}}\left(\gal(\overline{K}/K_{\infty}),\mathrm{GL}_n\left(\left(A_{\mathrm{LT},K}^{\mathrm{ur}}\right)^{\wedge}\right)\right)=0$.\label{vanishing of continuous coh of GL_n of witt vectors}
			\item [(c)]
			$H^1_{\mathrm{cont}}\left(\gal(\overline{K}/K(\mu_{p^{\infty}})),\mathrm{GL}_n\left(\left(A_{\mathrm{cycl},K}^{\mathrm{ur}}\right)^{\wedge}\right)\right)=0$.\label{vanishing of continuous coh of GL_n of witt vectors cycl}
		\end{enumerate}
	\end{lem}
	
	\begin{proof}
		The same strategy of the proof of Corollary \ref{vanishing of continuous cohomology of W_E(A_{infty}')} applies to (a), and (b) is a direct consequence of (a) and the fact that $H^1_{\mathrm{cont}}\left(\gal(\overline{K}/K_{\infty}),\mathrm{GL}_n\left(\mathbb{F}_q(\negthinspace(T_{\mathrm{LT}})\negthinspace)^{\mathrm{sep}}\right)\right)=0$
		(Hilbert satz 90, see for example \cite[p.151, Proposition 3]{serre1979local}), and the proof of (c) is exactly the same as (b), hence we omit the details.
	\end{proof}

	We define $ W_E\left(\mathbb{F}_q(\negthinspace(T_{\mathrm{LT}}^{p^{-\infty}})\negthinspace)\right)\coloneq\mathcal{O}_E\otimes_{\sigma_{0},\mathcal{O}_K}W\left(\mathbb{F}_q(\negthinspace(T_{\mathrm{LT}}^{p^{-\infty}})\negthinspace)\right)$. The following result reveals the relation between Fontaine's cyclotomic $(\varphi_q,\mathbb{Z}_p^{\times})$-modules and Lubin-Tate $(\varphi_q,\mathcal{O}_K^{\times})$-modules.
	
	\begin{cor}\label{comparison of LT and cycl modules}
		Let $\rho$ be a finite type $\mathcal{O}_E$-representation of $\gal(\overline{K}/K)$. There is a canonical isomorphism of $W_E\left(\mathbb{F}_q(\negthinspace(T_{\mathrm{LT}}^{p^{-\infty}})\negthinspace)\right)$-modules which commutes with the $(\varphi_q,\mathcal{O}_K^{\times})$-actions:
		\begin{equation}\label{isomorphism reveals the relation}
			W\left(\mathbb{F}_q(\negthinspace(T_{\mathrm{LT}}^{p^{-\infty}})\negthinspace)\right)\otimes_{A_{\mathrm{LT},K}}D_{\mathrm{LT},\sigma_{0}}(\rho)\xrightarrow{\sim}W\left(\mathbb{F}_q(\negthinspace(T_{\mathrm{LT}}^{p^{-\infty}})\negthinspace)\right)\otimes_{A_{\mathrm{cycl},K}}D_{\mathrm{cycl},\sigma_{0}}(\rho),
		\end{equation}
		where $\mathcal{O}_K^{\times}$ acts on $D_{\mathrm{cycl},\sigma_{0}}(\rho)$ via the norm map $\mathrm{Norm}_{K/\qp}: \mathcal{O}_K^{\times}\to \mathbb{Z}_p^{\times}$.
	\end{cor}
	
	\begin{proof}
	First we show that for any finite continuous $\mathcal{O}_E$-representation $\rho$ of $\gal(\overline{K}/K)$,
	there are natural isomorphisms of $\mathcal{O}_E\otimes_{\sigma_{0},\mathcal{O}_K}W\left(\mathbb{C}_p^{\flat}\right)$-modules commuting with the actions of $\gal(\overline{K}/K)$:
	\begin{align}\label{isomorphism of different phi,Gamma modules}
		W\big(\mathbb{C}_p^{\flat}\big)\otimes_{A_{\mathrm{LT},K}}D_{\mathrm{LT},\sigma_{0}}(\rho)\xrightarrow{\sim}W\big(\mathbb{C}_p^{\flat}\big)\otimes_{\sigma_{0},\mathcal{O}_K}\rho\xleftarrow{\sim}W\big(\mathbb{C}_p^{\flat}\big)\otimes_{A_{\mathrm{cycl},K}}D_{\mathrm{cycl},\sigma_{0}}(\rho).
	\end{align}
	
	If $\rho$ is a finite free continuous $\mathcal{O}_E$-representation of $\gal(\overline{K}/K)$, we can directly apply Lemma \ref{vanishing result, classical phi,Gamma modules} and the construction of $D_{\mathrm{LT},\sigma_{0}}$ and $D_{\mathrm{cycl},\sigma_{0}}$ to deduce that the natural homomorphisms
	\begin{align}\label{isomorphism of LT and cycl phi,Gamma modules}
		\begin{cases}
			\left(A_{\mathrm{LT},K}^{\mathrm{ur}}\right)^{\wedge}\otimes_{A_{\mathrm{LT},K}}D_{\mathrm{LT},\sigma_{0}}(\rho)\longrightarrow \left(A_{\mathrm{LT},K}^{\mathrm{ur}}\right)^{\wedge}\otimes_{\sigma_{0},\mathcal{O}_K}\rho, &\text{as}\ \mathcal{O}_E\otimes_{\sigma_{0},\mathcal{O}_K}\left(A_{\mathrm{LT},K}^{\mathrm{ur}}\right)^{\wedge}\text{-modules},\\
			\left(A_{\mathrm{cycl},K}^{\mathrm{ur}}\right)^{\wedge}\otimes_{A_{\mathrm{cycl},K}}D_{\mathrm{cycl},\sigma_{0}}(\rho)\longrightarrow\left(A_{\mathrm{cycl},K}^{\mathrm{ur}}\right)^{\wedge}\otimes_{\sigma_{0},\mathcal{O}_K}\rho, &\text{as}\ \mathcal{O}_E\otimes_{\sigma_{0},\mathcal{O}_K}\left(A_{\mathrm{cycl},K}^{\mathrm{ur}}\right)^{\wedge}\text{-modules},
		\end{cases}
	\end{align}
	are isomorphisms.
	Then (\ref{isomorphism of different phi,Gamma modules}) follows.

	If $\rho=\bigcup_{n=1}^{\infty}\rho[\varpi^n]$ is a torsion representation, since $\rho$ is finitely generated, there exists $N\geq 1$ such that $\varpi^N\cdot\rho =0$. We will show that (\ref{isomorphism of different phi,Gamma modules}) holds by induction on $N\geq 1$. If $\rho=\rho[\varpi]$, i.e. $\rho$ is a finite dimensional continuous $\mathbb{F}$-representation of $\gal(\overline{K}/K)$, then (\ref{isomorphism of different phi,Gamma modules}) follows from the proof of \cite[Proposition 2.1.2]{breuil2023multivariable}. If $N\geq 2$, as $D_{\mathrm{LT},\sigma_{0}}(-)$ and $D_{\mathrm{cycl},\sigma_{0}}(-)$ are equivalences of categories, we have exact sequences
	\[0\to D_{\mathrm{LT},\sigma_{0}}(\rho[\varpi^{N-1}])\to D_{\mathrm{LT},\sigma_{0}}(\rho)\to D_{\mathrm{LT},\sigma_{0}}(\rho/\rho[\varpi^{N-1}])\to0\]
	and
	\[0\to D_{\mathrm{cycl},\sigma_{0}}(\rho[\varpi^{N-1}])\to D_{\mathrm{cycl},\sigma_{0}}(\rho)\to D_{\mathrm{cycl},\sigma_{0}}(\rho/\rho[\varpi^{N-1}])\to0.\]
	Since $A_{\mathrm{LT},K}\hookrightarrow W\left(\mathbb{C}_p^{\flat}\right)$, $\mathcal{O}_K\hookrightarrow W\left(\mathbb{C}_p^{\flat}\right)$ and $A_{\mathrm{cycl},K}\hookrightarrow W\left(\mathbb{C}_p^{\flat}\right)$ are flat, there is a commutative of diagram of $\mathcal{O}_E\otimes_{\sigma_{0},\mathcal{O}_K}W\left(\mathbb{C}_p^{\flat}\right)$-modules with exact rows
	\begin{equation*}
		\begin{tikzcd}[column sep=tiny]
			W\left(\mathbb{C}_p^{\flat}\right)\otimes_{A_{\mathrm{LT},K}}D_{\mathrm{LT},\sigma_{0}}(\rho[\varpi^{N-1}]) \arrow[r,hook]\arrow[d,"\wr"] & W\left(\mathbb{C}_p^{\flat}\right)\otimes_{A_{\mathrm{LT},K}}D_{\mathrm{LT},\sigma_{0}}(\rho)\arrow[r,two heads]\arrow[d] & W\left(\mathbb{C}_p^{\flat}\right)\otimes_{A_{\mathrm{LT},K}}D_{\mathrm{LT},\sigma_{0}}(\rho/\rho[\varpi^{N-1}])\arrow[d,"\wr"]\\
			W\left(\mathbb{C}_p^{\flat}\right)\otimes_{\sigma_{0},\mathcal{O}_K}\rho[\varpi^{N-1}]\arrow[r,hook]&W\left(\mathbb{C}_p^{\flat}\right)\otimes_{\sigma_{0},\mathcal{O}_K}\rho\arrow[r,two heads]&W\left(\mathbb{C}_p^{\flat}\right)\otimes_{\sigma_{0},\mathcal{O}_K}\left(\rho/\rho[\varpi^{N-1}]\right)\\
			W\left(\mathbb{C}_p^{\flat}\right)\otimes_{A_{\mathrm{cycl},K}}D_{\mathrm{cycl},\sigma_{0}}(\rho[\varpi^{N-1}]) \arrow[r,hook]\arrow[u,"\wr"'] & W\left(\mathbb{C}_p^{\flat}\right)\otimes_{A_{\mathrm{cycl},K}}D_{\mathrm{cycl},\sigma_{0}}(\rho)\arrow[r,two heads]\arrow[u] & W\left(\mathbb{C}_p^{\flat}\right)\otimes_{A_{\mathrm{cycl},K}}D_{\mathrm{cycl},\sigma_{0}}(\rho/\rho[\varpi^{N-1}])\arrow[u,"\wr"']
		\end{tikzcd}
	\end{equation*}
	where the vertical maps are induced by (\ref{isomorphism of LT and cycl phi,Gamma modules}). Then (\ref{isomorphism of different phi,Gamma modules}) is a direct consequence of the induction hypothesis and the five lemma.
	
	For a general finite $\mathcal{O}_E$-representation $\rho$, let $\rho_{\mathrm{tor}}$ be the torsion subrepresentation of $\rho$, then there is a short exact sequence
	\[0\to \rho_{\mathrm{tor}}\to \rho\to \rho/\rho_{\mathrm{tor}}\to0.\]
	The same arguments as for torsion representations applies here as well, which implies that (\ref{isomorphism of different phi,Gamma modules}) holds for any finite $\mathcal{O}_E$-representation $\rho$. Then taking $\gal(\overline{K}/K_{\infty})$-invariants, the conclusion follows.
	\end{proof}
	
	Since 
	\[W\left(\mathbb{F}_q(\negthinspace(T_{\mathrm{LT}}^{p^{-\infty}})\negthinspace)\right)\otimes_{A_{\mathrm{LT},K}}D_{\mathrm{LT},\sigma_{0}}(\rho)\cong W_E\left(\mathbb{F}_q(\negthinspace(T_{\mathrm{LT}}^{p^{-\infty}})\negthinspace)\right)\otimes_{A_{\mathrm{LT},E,\sigma_{0}}}D_{\mathrm{LT},\sigma_{0}}(\rho)\]
	and
	\[W\left(\mathbb{F}_q(\negthinspace(T_{\mathrm{LT}}^{p^{-\infty}})\negthinspace)\right)\otimes_{A_{\mathrm{cycl},K}}D_{\mathrm{cycl},\sigma_{0}}(\rho)\cong W_E\left(\mathbb{F}_q(\negthinspace(T_{\mathrm{LT}}^{p^{-\infty}})\negthinspace)\right)\otimes_{A_{\mathrm{cycl},E}}D_{\mathrm{cycl},\sigma_{0}}(\rho),\]
	the isomorphism (\ref{isomorphism reveals the relation}) can be rewritten as 
	\begin{equation}\label{isomorphism reveals the relation ver2}
		W_E\left(\mathbb{F}_q(\negthinspace(T_{\mathrm{LT}}^{p^{-\infty}})\negthinspace)\right)\otimes_{A_{\mathrm{LT},E,\sigma_{0}}}D_{\mathrm{LT},\sigma_{0}}(\rho)\cong W_E\left(\mathbb{F}_q(\negthinspace(T_{\mathrm{LT}}^{p^{-\infty}})\negthinspace)\right)\otimes_{A_{\mathrm{cycl},E}}D_{\mathrm{cycl},\sigma_{0}}(\rho).
	\end{equation}
	
	\section{Étale $(\varphi_q,\mathcal{O}_K^{\times})$-modules attached to Galois representations}\label{section 6}
	
	In this section, we associate étale $(\varphi_q,\mathcal{O}_K^{\times})$-modules to finite type continuous $\mathcal{O}_E$-representations and finite dimensional continuous $E$-representations of $\gal(\overline{K}/K)$, and prove some basic properties of these functors.
	
	\subsection{Finitely presented étale $(\varphi_q,\mathcal{O}_K^{\times})$-modules attached to finite type continuous $\mathcal{O}_E$-representations of $\gal(\overline{K}/K)$}
	
	In this section, we construct the functor $D_{A_{\mathrm{mv},E}}^{(i)}: \rep_{\mathcal{O}_E}(\gal(\overline{K}/K))\to \modetfp_{\varphi_q,\mathcal{O}_K^{\times}}\left(A_{\mathrm{mv},E}\right)$ and show the fully faithfulness of $D_{A_{\mathrm{mv},E}}^{(i)}$ for $0\leq i\leq f-1$.
	
	For $0\leq i\leq f-1$, consider the following diagram of topological rings:
	\begin{equation}
		\begin{tikzcd}
			A_{\mathrm{mv},E}\arrow[r,hook]& W_E(A_{\infty})\arrow[r,hook,"m"]& W_E(A_{\infty}')\\
			 & & A_{\mathrm{LT},E,\sigma_{0}}\arrow[u,hook,"\mathrm{pr}_{i}"],
		\end{tikzcd}
	\end{equation}
	where the map $\mathrm{pr}_i: A_{\mathrm{LT},E,\sigma_{0}}\hookrightarrow W_E(A_{\infty})$ is the composition of (\ref{equivairant injection Lubin-Tate case}) and $W_E\left(\mathbb{F}_q(\negthinspace(T_{\mathrm{LT}}^{1/p^{\infty}})\negthinspace)\right)\hookrightarrow W_E(A_{\infty}'),\ [T_{\mathrm{LT}}]\mapsto [T_{\mathrm{LT},i}]$.
	The map $\mathrm{pr}_i$ commutes with the continuous $(\varphi_q,\mathcal{O}_K^{\times})$-actions (where $\mathcal{O}_K^{\times}$ acts on $W_E(A_{\infty}')$ via the $i$-th inclusion $\mathcal{O}_K^{\times}\xhookrightarrow{j_i}(\mathcal{O}_K^{\times})^f$).
	
	For a finite type continuous $\mathcal{O}_E$-representation $\rho$ of $\gal(\overline{K}/K)$, $0\leq i\leq f-1$, we put
	\begin{align}
		D_{W_E(A_{\infty})}^{(i)}(\rho)\coloneq\left(W_E(A_{\infty}')\otimes_{\mathrm{pr}_i,A_{\mathrm{LT},E,\sigma_{0}}}D_{\mathrm{LT},\sigma_0}(\rho)\right)^{\Delta_1},
	\end{align}
	which is a finitely presented étale $(\varphi_q,\mathcal{O}_K^{\times})$-module over $W_E(A_{\infty})$ by Theorem \ref{descent of f.p. from W_E(A{infty}') to W_E(A{infty})}. Then we define 
	\begin{align*}
		D_{A_{\mathrm{mv},E}}^{(i)}\coloneq \mathrm{Dec}\circ D_{W_E(A_{\infty})}^{(i)}: \rep_{\mathcal{O}_E}(\gal(\overline{K}/K)) &\to \modetfp_{\varphi_q,\mathcal{O}_K^{\times}}(A_{\mathrm{mv},E})\\
		\rho& \mapsto \mathrm{Dec}\left(D_{W_E(A_{\infty})}^{(i)}(\rho)\right),
	\end{align*}
	where $\mathrm{Dec}$ is the decompletion and deperfection functor (Definition \ref{defn of decompletion}).
	If $\rho$ is a finite free $\mathcal{O}_E$-representation, then $D_{A_{\mathrm{mv},E}}^{(i)}(\rho)$ is a finite free $A_{\mathrm{mv},E}$-module of the same rank. Moreover, for different $i$, these functors $D_{A_{\mathrm{mv},E}}^{(i)}$ are related by the following lemma:
	
	\begin{lem}\label{ralation of different f.p. D_{A_E}^i}
		For $0\leq i\leq f-1$ and any finite type $\mathcal{O}_E$-representation $\rho$ of $\gal(\overline{K}/K)$, there is an isomorphism of finitely presented étale $(\varphi_q,\mathcal{O}_K^{\times})$-modules over $A_{\mathrm{mv},E}$ which is functorial in $\rho$:
		\begin{align}\label{tensoring with phi changes i}
			\phi_i: A_{\mathrm{mv},E}\otimes_{\varphi,A_{\mathrm{mv},E}}D_{A_{\mathrm{mv},E}}^{(i)}(\rho)\xrightarrow{\sim} D_{A_{\mathrm{mv},E}}^{(i+1)}(\rho).
		\end{align}
		Moreover, $\phi_{f-1}\circ \phi_{f-2}\circ\cdots\circ \phi_0: A_{\mathrm{mv},E}\otimes_{\varphi_q,A_{\mathrm{mv},E}}D_{A_{\mathrm{mv},E}}^{(0)}(\rho)\to D_{A_{\mathrm{mv},E}}^{(0)}(\rho)$ is the linearization map $\mathrm{id}\otimes{\varphi_q}$.
	\end{lem}
	\begin{proof}
		Recall that there is an automorphism of $A_{\infty}=\mathbb{F}(\negthinspace(T_{\sigma_{0}}^{1/p^{\infty}})\negthinspace)\left\langle\left(\frac{T_{\sigma_{i}}}{T_{\sigma_{0}}}\right)^{\pm1/p^{\infty}}:1\leq i\leq f-1\right\rangle$
		\begin{align*}
			\varphi: A_{\infty} \to A_{\infty},\
			T_{\sigma_{i}}\mapsto T_{\sigma_{i-1}}^p,
		\end{align*}
		and an $\mathbb{F}$-linear automorphism of $A_{\infty}'=\mathbb{F}\left(\negthinspace\left(T_{\mathrm{LT},0}^{1/p^{\infty}}\right)\negthinspace\right)\left\langle\left(\frac{T_{\mathrm{LT},i}}{T_{\mathrm{LT},0}^{p^{i}}}\right)^{\pm 1/p^{\infty}}: 1\leq i\leq f-1\right\rangle$:
		\begin{align*}
			\varphi: A_{\infty}' \to A_{\infty}',\
			T_{\mathrm{LT},i}\mapsto \begin{cases}
				T_{\mathrm{LT},i+1},&  i\neq f-1\\
				T_{\mathrm{LT},0}^q,& i=f-1.
			\end{cases}
		\end{align*} 
		(See \cite[$\S$2.4 (43), $\S$2.5 (47)]{breuil2023multivariable}). Moreover, the homomorphism $m: A_{\infty}\hookrightarrow A_{\infty}'$ satisfies $m\circ\varphi=\varphi\circ m$. These automorphisms induce $\mathcal{O}_E$-linear automorphisms $\varphi: W_E(A_{\infty})\to W_E(A_{\infty})$ and $\varphi:W_E(A_{\infty}')\to W_E(A_{\infty}')$. One can easily check that 
		\begin{align*}
			\varphi\circ \mathrm{pr}_i=\begin{cases}
				\mathrm{pr}_{i+1},& i\neq f-1\\
				\mathrm{pr}_0\circ\varphi_q, &i=f-1.
			\end{cases}
		\end{align*}
		By Corollary \ref{descent of f.p. phi_q,O_K*-modules}, it suffices to construct a natural isomorphism of finitely presented étale $(\varphi_q,\mathcal{O}_K^{\times})$-modules
		\[\phi_i: W_E(A_{\infty})\otimes_{\varphi,W_E(A_{\infty})}D_{W_E(A_{\infty})}^{(i)}(\rho)\xrightarrow{\sim} D_{W_E(A_{\infty})}^{(i+1)}(\rho).\]
		Using the definition $D_{W_E(A_{\infty})}^{(i)}(\rho)=\left(W_E(A_{\infty}')\otimes_{\mathrm{pr}_i,A_{\mathrm{LT},E,\sigma_{0}}}D_{\mathrm{LT},\sigma_0}(\rho)\right)^{\Delta_1}$, we can construct $\phi_i$ as follows:
		\begin{align*}
			\phi_i: a\otimes(x\otimes v)\mapsto \begin{cases}
				a\cdot\varphi(x)\otimes v,& 0\leq i\leq f-2,\\
				a\cdot\varphi(x)\otimes \varphi_q(v), &i=f-1,
			\end{cases}
		\end{align*}
		where $a\in W_E(A_{\infty})$, $x\in W_E(A_{\infty}')$, $v\in D_{\mathrm{LT},\sigma_{0}}(\rho)$. A direct computation shows that $\phi_i$ is well-defined for every $i$ and commutes with $\varphi_q$ and the action of $\mathcal{O}_K^{\times}$. Since the composition $\phi_{i+f-1}\circ\phi_{i+f-2}\circ\cdots\circ \phi_{i+1}\circ\phi_i$ is the linearization map
		\[\mathrm{id}\otimes{\varphi_q}:W_E(A_{\infty})\otimes_{\varphi_q,W_E(A_{\infty})}D_{W_E(A_{\infty})}^{(i)}(\rho)\xrightarrow{\sim} D_{W_E(A_{\infty})}^{(i)}(\rho)\]
		which is an isomorphism, we see that $\phi_i$ is an isomorphism for every $i$.
	\end{proof}
	
	Since $W_E(A_{\infty}')$ is an integral domain, and $\mathrm{pr}_i: A_{\mathrm{LT},E,\sigma_{0}}\to W_E(A_{\infty}')$ is injective, we see that $W_E(A_{\infty}')$ is torsion-free as an $A_{\mathrm{LT},E,\sigma_{0}}$-module. Since $A_{\mathrm{LT},E,\sigma_{0}}$ is a discrete valuation ring, in particular, $A_{\mathrm{LT},E,\sigma_{0}}$ is a principal ideal domain and $W_E(A_{\infty}')$ is a flat $A_{\mathrm{LT},E,\sigma_{0}}$-algebra. Moreover, the morphism $\mathrm{Spec}(W_E(A_{\infty}'))\to \mathrm{Spec}(A_{\mathrm{LT},E,\sigma_{0}})$ is surjective, we conclude that the ring map $A_{\mathrm{LT},E,\sigma_{0}}\xrightarrow{\mathrm{pr}_i} W_E(A_{\infty}')$ is faithfully flat, hence $W_E(A_{\infty}')\otimes_{\mathrm{pr}_i,A_{\mathrm{LT},E,\sigma_{0}}}-$ is faithful and exact. As $(-)^{\Delta_1}$ and $\mathrm{Dec}$ are equivalences of categories, we see that the functor $D_{A_{\mathrm{mv},E}}^{(i)}=\mathrm{Dec}\circ\left(W_E(A_{\infty}')\otimes_{\mathrm{pr}_i,A_{\mathrm{LT},E,\sigma_{0}}}-\right)^{\Delta_1}$ is also faithful and exact. We will show that the functor $D_{A_{\mathrm{mv},E}}^{(i)}$ is full. As in \cite[$\S$2.8]{breuil2023multivariable}, the key is to find out how $D_{A_{\mathrm{mv},E}}^{(i)}$ and $D_{\mathrm{cycl},\sigma_{i}}$ are related.
	
	Since $K/\qp$ is unramified, the map $\mathcal{O}_K\xrightarrow{\mathrm{tr}}\mathbb{Z}_p$
	is a surjective map of $\mathbb{Z}_p$-modules. It induces a surjective ring map of Iwasawa algebras $\mathcal{O}_K[\negthinspace[\mathcal{O}_K]\negthinspace]\xrightarrow{\mathrm{tr}}\mathcal{O}_K[\negthinspace[\mathbb{Z}_p]\negthinspace]$ which commutes with the $(\varphi_q,\mathbb{Z}_p^{\times})$-actions. We choose variables as in (\ref{choice of variables}) for $\mathcal{O}_K[\negthinspace[\mathcal{O}_K]\negthinspace]$ so that
	\[\mathcal{O}_K[\negthinspace[\mathcal{O}_K]\negthinspace]=\mathcal{O}_K[\negthinspace[T_{0},\dots,T_{f-1}]\negthinspace],\]
	and let $T_{\mathrm{cycl}}\coloneq [1]-1\in \mathcal{O}_K[\negthinspace[\mathbb{Z}_p]\negthinspace]$, then $\mathcal{O}_K[\negthinspace[\mathbb{Z}_p]\negthinspace]=\mathcal{O}_K[\negthinspace[T_{\mathrm{cycl}}]\negthinspace]$. By the third paragraph of \cite[$\S$3.1.3, p.116]{breuil2023conjectures}, for $0\leq i\leq f-1$, the map $\mathrm{tr}$ satisfies
	\[\mathrm{tr}(T_i)\equiv -T_{\mathrm{cycl}} \mod (p,T_{\mathrm{cycl}}^2),\]
	hence it induces a continuous surjection $\mathrm{tr}: A_{\mathrm{mv},E}\twoheadrightarrow A_{\mathrm{cycl},E}$ which commutes with the $(\varphi_q,\mathbb{Z}_p^{\times})$-actions.
	
	Similarly, the trace map also induces a surjection $\mathrm{tr}: A_{\infty}\twoheadrightarrow \mathbb{F}(\negthinspace(T_{\mathrm{cycl}}^{1/p^{\infty}})\negthinspace)$ which commutes with the $(\varphi_q,\mathbb{Z}_p^{\times})$-actions (see \cite[$\S$3.1.3, p.116]{breuil2023conjectures}). It uniquely lifts to an $\mathcal{O}_E$-linear ring homomorphism, also denoted by $\mathrm{tr}: W_E(A_{\infty})\twoheadrightarrow W_E\left( \mathbb{F}(\negthinspace(T_{\mathrm{cycl}}^{1/p^{\infty}})\negthinspace)\right)$.
	\begin{lem}
		There is a commutative diagram \begin{equation}\label{comm diag 1}
			\begin{tikzcd}
				A_{\mathrm{mv},E} \arrow[r,two heads,"\mathrm{tr}"] \arrow[d,hook] & A_{\mathrm{cycl},E} \arrow[d,hook]\\
				W_E(A_{\infty}) \arrow[r,two heads,"\mathrm{tr}"] & W_E\left( \mathbb{F}(\negthinspace(T_{\mathrm{cycl}}^{1/p^{\infty}})\negthinspace)\right)
			\end{tikzcd}
		\end{equation}
		where the left vertical maps is (\ref{(varphi_q,mathcal{O}_K^{*})-equivariant injection from A_{mv,E} to W_E(A_{infty})})
		and the right vertical map is (\ref{(varphi_q,Z_p^*)-equivariant injection from A_{mv,E} to W_E(A_{infty})})
	\end{lem}
	
	\begin{proof}
		Since $\mathrm{tr}: A_{\mathrm{mv},E}\twoheadrightarrow A_{\mathrm{cycl},E}$ commutes with $\varphi_q$,
		it induces a continuous surjection
		\[f: \varinjlim_{\varphi_q}A_{\mathrm{mv},E}\twoheadrightarrow \varinjlim_{\varphi_q}A_{\mathrm{cycl},E}.\]
		By taking topological completions, we get a ring homomorphism
		\[f: W_E(A_{\infty})\cong \left(\varinjlim_{\varphi_q}A_{\mathrm{mv},E}\right)^{\wedge}\to \left(\varinjlim_{\varphi_q}A_{\mathrm{cycl},E}\right)^{\wedge} \cong W_E\left( \mathbb{F}(\negthinspace(T_{\mathrm{cycl}}^{1/p^{\infty}})\negthinspace)\right)\]
		which makes the following diagram commute:
		\begin{equation*}
			\begin{tikzcd}
				A_{\mathrm{mv},E} \arrow[r,"\mathrm{tr}"] \arrow[d] & A_{\mathrm{cycl},E}\arrow[d]\\
				W_E(A_{\infty}) \arrow[r,"f"] & W_E\left( \mathbb{F}(\negthinspace(T_{\mathrm{cycl}}^{1/p^{\infty}})\negthinspace)\right).
			\end{tikzcd}
		\end{equation*}
		Thus it suffices to show that the two maps $f:W_E(A_{\infty})\to W_E\left( \mathbb{F}(\negthinspace(T_{\mathrm{cycl}}^{1/p^{\infty}})\negthinspace)\right)$ and $\mathrm{tr}: W_E(A_{\infty})\to W_E\left( \mathbb{F}(\negthinspace(T_{\mathrm{cycl}}^{1/p^{\infty}})\negthinspace)\right)$ are identical. Hence it suffices to verify that the two homomorphisms of strict $p$-rings $f:W(A_{\infty})\to W\left( \mathbb{F}(\negthinspace(T_{\mathrm{cycl}}^{1/p^{\infty}})\negthinspace)\right)$ and $\mathrm{tr}: W(A_{\infty})\to W\left( \mathbb{F}(\negthinspace(T_{\mathrm{cycl}}^{1/p^{\infty}})\negthinspace)\right)$ are equal.
		By the theory of strict $p$-rings (\cite[Chapter II, $\S$5 Proposition 10]{serre1979local}), it suffices to verify that $f$ and $\mathrm{tr}$ are equal after reducing modulo $p$, which follows straightforwardly.
	\end{proof}
	
	\begin{prop}\label{relation to classical cycltomic (phi,Gamma)-modules}
		There exists an embedding $\sigma_d\in\{\sigma_{0},\dots,\sigma_{f-1}\}$ such that for any finite type $\mathcal{O}_E$-representation $\rho$ of $\gal(\overline{K}/K)$, there is a canonical isomorphism of finitely presented étale $(\varphi_q,\mathbb{Z}_p^{\times})$-modules over $A_{\mathrm{cycl},E}$:
		\[A_{\mathrm{cycl},E}\otimes_{\mathrm{tr},A_{\mathrm{mv},E}}D_{A_{\mathrm{mv},E}}^{(0)}(\rho)\cong D_{\mathrm{cycl},\sigma_d}(\rho).\]
	\end{prop}
	
	\begin{proof}
		By Lemma \ref{descent of f.p. etale modules from W_E(T_cycl) to A_cycl}, it suffices to show that there exists an embedding $\sigma_d:\mathcal{O}_K\hookrightarrow \mathcal{O}_E$ such that
		\[W_E\left(\mathbb{F}(\negthinspace(T_{\mathrm{cycl}}^{1/p^{\infty}})\negthinspace)\right)\otimes_{A_{\mathrm{mv},E}}D_{A_{\mathrm{mv},E}}^{(0)}(\rho)\cong W_E\left( \mathbb{F}(\negthinspace(T_{\mathrm{cycl}}^{1/p^{\infty}})\negthinspace)\right) \otimes_{A_{\mathrm{cycl},E}} D_{\mathrm{cycl},\sigma_d}(\rho).\]
		Using (\ref{comm diag 1}), the left hand side is
		\[W_E\left(\mathbb{F}(\negthinspace(T_{\mathrm{cycl}}^{1/p^{\infty}})\negthinspace)\right)\otimes_{A_{\mathrm{mv},E}}D_{A_{\mathrm{mv},E}}^{(0)}(\rho)\cong W_E\left(\mathbb{F}(\negthinspace(T_{\mathrm{cycl}}^{1/p^{\infty}})\negthinspace)\right)\otimes_{\mathrm{tr},W_E(A_{\infty})}D_{W_E(A_{\infty})}^{(0)}(\rho).\]
		Hence we need to show that there is an embedding $\sigma_{d}$ and a canonical isomorphism of finitely presented étale $(\varphi_q,\mathbb{Z}_p^{\times})$-modules over $W_E\left(\mathbb{F}(\negthinspace(T_{\mathrm{cycl}}^{1/p^{\infty}})\negthinspace)\right)$:
		\begin{align}\label{required isomorphism}
			W_E\left(\mathbb{F}(\negthinspace(T_{\mathrm{cycl}}^{1/p^{\infty}})\negthinspace)\right)\otimes_{\mathrm{tr},W_E(A_{\infty})}D_{W_E(A_{\infty})}^{(0)}(\rho)\cong W_E\left( \mathbb{F}(\negthinspace(T_{\mathrm{cycl}}^{1/p^{\infty}})\negthinspace)\right) \otimes_{A_{\mathrm{cycl},E}} D_{\mathrm{cycl},\sigma_d}(\rho).
		\end{align}
		Recall that there is a commutative diagram of perfectoid rings over $\mathbb{F}$ 	(\cite[$\S$2.8 (54)]{breuil2023multivariable}) which commutes with the $(\varphi_q,\mathcal{O}_K^{\times})$-action
		\begin{equation}\label{comm diag 2, mod p version}
			\begin{tikzcd}
				A_{\infty} \arrow[rrr,"\mathrm{tr}",two heads] \arrow[d,"m",hook] & & & \mathbb{F}(\negthinspace(T_{\mathrm{cycl}}^{1/p^{\infty}})\negthinspace) \arrow[d,"\iota",hook]\\
				A_{\infty}' \arrow[rrr,"T_{\mathrm{LT},i}^{p^{-n}}\mapsto T_{\mathrm{LT},\sigma_0}^{p^{i-n}}",two heads] & & & \mathbb{F}(\negthinspace(T_{\mathrm{LT}}^{1/p^{\infty}})\negthinspace)
			\end{tikzcd}
		\end{equation}
		where the right vertical map is induced by the left vertical map of \cite[$\S$2.8 (53)]{breuil2023multivariable}, and $\mathcal{O}_K^{\times}$ acts on $\mathbb{F}(\negthinspace(T_{\mathrm{cycl}}^{1/p^{\infty}})\negthinspace)$ via the norm map $\mathrm{Norm}_{K/\qp}: \mathcal{O}_K^{\times}\twoheadrightarrow \mathbb{Z}_p^{\times}$. Since $\gal(K_{\infty}/K(\mu_{p^{\infty}}))\cong\mathrm{ker}(\mathcal{O}_K^{\times}\xrightarrow{\mathrm{Norm}_{K/\qp}} \mathbb{Z}_p^{\times})$, the right vertical map $\iota$ decomposes as
		\[\iota: \mathbb{F}(\negthinspace(T_{\mathrm{cycl}}^{1/p^{\infty}})\negthinspace)\to\mathbb{F}(\negthinspace(T_{\mathrm{LT}}^{1/p^{\infty}})\negthinspace)^{\gal(K_{\infty}/K(\mu_{p^{\infty}}))}\hookrightarrow \mathbb{F}(\negthinspace(T_{\mathrm{LT}}^{1/p^{\infty}})\negthinspace).\]
		Recall that tilting functor gives an isomorphism (\ref{injection induced by tilting functor})
		\begin{align*}
			\mathbb{F}(\negthinspace(T_{\mathrm{cycl}}^{1/p^{\infty}})\negthinspace)\cong\mathbb{F}(\negthinspace(T_{\mathrm{LT}}^{1/p^{\infty}})\negthinspace)^{\gal(K_{\infty}/K(\mu_{p^{\infty}}))},
		\end{align*}
		thus the first part of $\iota$ is a continuous endomorphism of $\mathbb{F}(\negthinspace(T_{\mathrm{cycl}}^{1/p^{\infty}})\negthinspace)$ which commutes with the action of $\mathbb{Z}_p^{\times}$. Hence
		by \cite[Theorem 3.1]{berger2022decompletion}, there exists $\gamma\in\qp^{\times}$ such that $\iota$ is given by the composition
		\begin{align}\label{mod p case}
			\iota: \mathbb{F}(\negthinspace(T_{\mathrm{cycl}}^{1/p^{\infty}})\negthinspace)\xrightarrow[\sim]{T_{\mathrm{cycl}}\mapsto (1+T_{\mathrm{cycl}})^{\gamma}-1}\mathbb{F}(\negthinspace(T_{\mathrm{cycl}}^{1/p^{\infty}})\negthinspace)\hookrightarrow \mathbb{F}(\negthinspace(T_{\mathrm{LT}}^{1/p^{\infty}})\negthinspace)
		\end{align}
		where the last injection is (\ref{injection induced by tilting functor}).
		
		The commutative diagram (\ref{comm diag 2, mod p version}) lifts to a commutative diagram which commutes with the $(\varphi_q,\mathcal{O}_K^{\times})$-actions:
		\begin{equation}\label{comm diag 3}
			\begin{tikzcd}[column sep=huge]
				W_E\left(A_{\infty}\right) \arrow[rrr,"\mathrm{tr}",two heads] \arrow[d,"m",hook] & & & W_E\left(\mathbb{F}(\negthinspace(T_{\mathrm{cycl}}^{1/p^{\infty}})\negthinspace)\right) \arrow[d,hook,"\iota"]\\
				W_E\left(A_{\infty}'\right) \arrow[rrr,"{[T_{\mathrm{LT},i}]}^{p^{-n}}\mapsto{[T_{\mathrm{LT},\sigma_0}]} ^{p^{i-n}}",two heads] & & & W_E\left(\mathbb{F}(\negthinspace(T_{\mathrm{LT}}^{1/p^{\infty}})\negthinspace)\right)
			\end{tikzcd}
		\end{equation}
		As a consequence of (\ref{mod p case}), the right vertical map $\iota$ decomposes as
		\[\iota: W_E\left( \mathbb{F}(\negthinspace(T_{\mathrm{cycl}}^{1/p^{\infty}})\negthinspace)\right)\xrightarrow[\sim]{[T_{\mathrm{cycl}}]\mapsto \left[(1+T_{\mathrm{cycl}})^{\gamma}-1\right]}W_E\left(\mathbb{F}(\negthinspace(T_{\mathrm{cycl}}^{1/p^{\infty}})\negthinspace)\right)\hookrightarrow W_E\left( \mathbb{F}(\negthinspace(T_{\mathrm{LT}}^{1/p^{\infty}})\negthinspace)\right),\]
		where $[T_{\mathrm{cycl}}]$ and $\left[(1+T_{\mathrm{cycl}})^{\gamma}-1\right]$ are the Teichmüller lifts of $T_{\mathrm{cycl}}$ and $(1+T_{\mathrm{cycl}})^{\gamma}-1$ respectively. We write $\gamma=p^{mf+d}\cdot a$ for some $a\in \mathbb{Z}_p^{\times}$, $m\in\mathbb{Z}$ and $d\in\{0,\dots,f-1\}$, since the $\mathbb{F}$-linear endomorphism $T_{\mathrm{cycl}}\mapsto (1+T_{\mathrm{cycl}})^p-1$ of $\mathbb{F}(\negthinspace(T_{\mathrm{cycl}}^{1/p^{\infty}})\negthinspace)$ is just the endomorphism $\varphi$, the isomorphism $W_E\left( \mathbb{F}(\negthinspace(T_{\mathrm{cycl}}^{1/p^{\infty}})\negthinspace)\right)\xrightarrow[\sim]{[T_{\mathrm{cycl}}]\mapsto \left[(1+T_{\mathrm{cycl}})^{\gamma}-1\right]}W_E\left(\mathbb{F}(\negthinspace(T_{\mathrm{cycl}}^{1/p^{\infty}})\negthinspace)\right)$, which we denote by $\gamma$, decomposes as
		\begin{align}\label{decomposition of iota}
			\gamma: W_E\left( \mathbb{F}(\negthinspace(T_{\mathrm{cycl}}^{1/p^{\infty}})\negthinspace)\right)\xrightarrow{a}W_E\left(\mathbb{F}(\negthinspace(T_{\mathrm{cycl}}^{1/p^{\infty}})\negthinspace)\right)\xrightarrow{\varphi^{mf+d}}W_E\left( \mathbb{F}(\negthinspace(T_{\mathrm{cycl}}^{1/p^{\infty}})\negthinspace)\right).
		\end{align}
		where the first map is the action by $a\in \mathbb{Z}_p^{\times}$.
		
		Since $W_E\left(\mathbb{F}(\negthinspace(T_{\mathrm{LT}}^{1/p^{\infty}})\negthinspace)\right)\xrightarrow{[T_{\mathrm{LT}}]\mapsto [T_{\mathrm{LT},0}]}W_E\left(A_{\infty}'\right)$ is a section of the bottom horizontal surjection of (\ref{comm diag 3}), using Theorem \ref{descent of f.p. from W_E(A{infty}') to W_E(A{infty})} and (\ref{comm diag 3}), we have
		\begin{equation}\label{isomorphism on the one hand}
			\begin{aligned}
				W_E\left(\mathbb{F}(\negthinspace(T_{\mathrm{LT}}^{1/p^{\infty}})\negthinspace)\right)\otimes_{A_{\mathrm{LT},E,\sigma_{0}}}D_{\mathrm{LT},\sigma_{0}}(\rho)& \cong W_E\left(\mathbb{F}(\negthinspace(T_{\mathrm{LT}}^{1/p^{\infty}})\negthinspace)\right)\otimes_{W_E(A_{\infty}')}W_E(A_{\infty}')\otimes_{\mathrm{pr}_0,A_{\mathrm{LT},E,\sigma_{0}}}D_{\mathrm{LT},\sigma_{0}}(\rho)\\
				&\cong W_E\left(\mathbb{F}(\negthinspace(T_{\mathrm{LT}}^{1/p^{\infty}})\negthinspace)\right)\otimes_{W_E(A_{\infty}')}W_E(A_{\infty}')\otimes_{W_E(A_{\infty})}D_{W_E(A_{\infty})}^{(0)}(\rho)\\
				&\cong W_E\left(\mathbb{F}(\negthinspace(T_{\mathrm{LT}}^{1/p^{\infty}})\negthinspace)\right)\otimes_{\iota,W_E\left(\mathbb{F}(\negthinspace(T_{\mathrm{cycl}}^{1/p^{\infty}})\negthinspace)\right)}W_E\left(\mathbb{F}(\negthinspace(T_{\mathrm{cycl}}^{1/p^{\infty}})\negthinspace)\right)\\
				&\qquad \otimes_{\mathrm{tr},W_E(A_{\infty})}D_{W_E(A_{\infty})}^{(0)}(\rho).
			\end{aligned}
		\end{equation}
		On the other hand, the isomorphism (\ref{isomorphism reveals the relation ver2}) gives
		\begin{equation}\label{isomorphism on the other hand}
			\begin{aligned}
				W_E\left(\mathbb{F}(\negthinspace(T_{\mathrm{LT}}^{p^{-\infty}})\negthinspace)\right)\otimes_{A_{\mathrm{LT},E,\sigma_{0}}}D_{\mathrm{LT},\sigma_{0}}(\rho)&\cong W_E\left(\mathbb{F}(\negthinspace(T_{\mathrm{LT}}^{p^{-\infty}})\negthinspace)\right)\otimes_{A_{\mathrm{cycl},E}}D_{\mathrm{cycl},\sigma_{0}}(\rho)\\
				&\cong W_E\left(\mathbb{F}(\negthinspace(T_{\mathrm{LT}}^{p^{-\infty}})\negthinspace)\right)\otimes_{W_E\left(\mathbb{F}(\negthinspace(T_{\mathrm{cycl}}^{1/p^{\infty}})\negthinspace)\right)}W_E\left(\mathbb{F}(\negthinspace(T_{\mathrm{cycl}}^{1/p^{\infty}})\negthinspace)\right)\\
				&\qquad\qquad\qquad\qquad\qquad\otimes_{A_{\mathrm{cycl},E}}D_{\mathrm{cycl},\sigma_{0}}(\rho)\\
				&\cong W_E\left(\mathbb{F}(\negthinspace(T_{\mathrm{LT}}^{p^{-\infty}})\negthinspace)\right)\otimes_{\iota,W_E\left(\mathbb{F}(\negthinspace(T_{\mathrm{cycl}}^{1/p^{\infty}})\negthinspace)\right)}W_E\left(\mathbb{F}(\negthinspace(T_{\mathrm{cycl}}^{1/p^{\infty}})\negthinspace)\right)\\
				& \qquad\otimes_{\gamma^{-1},W_E\left(\mathbb{F}(\negthinspace(T_{\mathrm{cycl}}^{1/p^{\infty}})\negthinspace)\right)}W_E\left(\mathbb{F}(\negthinspace(T_{\mathrm{cycl}}^{1/p^{\infty}})\negthinspace)\right)\otimes_{A_{\mathrm{cycl},E}}D_{\mathrm{cycl},\sigma_{0}}(\rho),
			\end{aligned}
		\end{equation}
		where the second isomorphism is induced by the injection (\ref{injection induced by tilting functor}). To simplify the notation, we put
		\[D(\rho)\coloneq W_E\left(\mathbb{F}(\negthinspace(T_{\mathrm{cycl}}^{1/p^{\infty}})\negthinspace)\right) \otimes_{\gamma^{-1},W_E\left(\mathbb{F}(\negthinspace(T_{\mathrm{cycl}}^{1/p^{\infty}})\negthinspace)\right)}W_E\left(\mathbb{F}(\negthinspace(T_{\mathrm{cycl}}^{1/p^{\infty}})\negthinspace)\right)\otimes_{A_{\mathrm{cycl},E}}D_{\mathrm{cycl},\sigma_{0}}(\rho)\]
		which is a finitely presented étale $(\varphi_q,\mathbb{Z}_p^{\times})$-module over $W_E\left(\mathbb{F}(\negthinspace(T_{\mathrm{cycl}}^{1/p^{\infty}})\negthinspace)\right)$. Combining (\ref{isomorphism on the one hand}) and (\ref{isomorphism on the other hand}), we get an isomorphism 
		\begin{equation}\label{isormorphism 1}
			\begin{aligned}
				&W_E\left(\mathbb{F}(\negthinspace(T_{\mathrm{LT}}^{1/p^{\infty}})\negthinspace)\right)\otimes_{\iota,W_E\left(\mathbb{F}(\negthinspace(T_{\mathrm{cycl}}^{1/p^{\infty}})\negthinspace)\right)}\left(W_E\left(\mathbb{F}(\negthinspace(T_{\mathrm{cycl}}^{1/p^{\infty}})\negthinspace)\right) \otimes_{\mathrm{tr},W_E(A_{\infty})}D_{W_E(A_{\infty})}^{(0)}(\rho)\right)\\
				\cong & W_E\left(\mathbb{F}(\negthinspace(T_{\mathrm{LT}}^{p^{-\infty}})\negthinspace)\right)\otimes_{\iota,W_E\left(\mathbb{F}(\negthinspace(T_{\mathrm{cycl}}^{1/p^{\infty}})\negthinspace)\right)}D(\rho).
			\end{aligned}
		\end{equation}
		Since the actions of $\gal(K_{\infty}/K(\mu_{p^{\infty}}))$ on $W_E\left(\mathbb{F}(\negthinspace(T_{\mathrm{cycl}}^{1/p^{\infty}})\negthinspace)\right) \otimes_{\mathrm{tr},W_E(A_{\infty})}D_{W_E(A_{\infty})}^{(0)}(\rho)$ and $D(\rho)$ are trivial, by taking $\gal(K_{\infty}/K(\mu_{p^{\infty}}))$-invariants and using the fact that (\ref{decomposition of iota}) is an isomorphism, we deduce that (\ref{isormorphism 1}) descends to an isomorphism:
		\begin{equation}\label{isomorphism 3}
			\begin{aligned}
				W_E\left(\mathbb{F}(\negthinspace(T_{\mathrm{cycl}}^{1/p^{\infty}})\negthinspace)\right) \otimes_{\mathrm{tr},W_E(A_{\infty})}D_{W_E(A_{\infty})}^{(0)}(\rho)
				\cong D(\rho).
			\end{aligned}
		\end{equation}
		On the other hand, using (\ref{decomposition of iota}), we can compute $D(\rho)$ explicitly:
		\begin{align*}
			D(\rho)\cong& W_E\left(\mathbb{F}(\negthinspace(T_{\mathrm{cycl}}^{1/p^{\infty}})\negthinspace)\right) \otimes_{a^{-1},W_E\left(\mathbb{F}(\negthinspace(T_{\mathrm{cycl}}^{1/p^{\infty}})\negthinspace)\right)}W_E\left(\mathbb{F}(\negthinspace(T_{\mathrm{cycl}}^{1/p^{\infty}})\negthinspace)\right)\\
			&\quad \otimes_{\varphi^{-nf-d},W_E\left(\mathbb{F}(\negthinspace(T_{\mathrm{cycl}}^{1/p^{\infty}})\negthinspace)\right)}W_E\left(\mathbb{F}(\negthinspace(T_{\mathrm{cycl}}^{1/p^{\infty}})\negthinspace)\right)\otimes_{A_{\mathrm{cycl},E}}D_{\mathrm{cycl},\sigma_{0}}(\rho) \\
			\cong& W_E\left(\mathbb{F}(\negthinspace(T_{\mathrm{cycl}}^{1/p^{\infty}})\negthinspace)\right) \otimes_{a^{-1},A_{\mathrm{cycl},E}}A_{\mathrm{cycl},E}\otimes_{\varphi^{-nf-d},A_{\mathrm{cycl},E}}D_{\mathrm{cycl},\sigma_{0}}(\rho)\\
			\cong& W_E\left(\mathbb{F}(\negthinspace(T_{\mathrm{cycl}}^{1/p^{\infty}})\negthinspace)\right) \otimes_{a^{-1},A_{\mathrm{cycl},E}}D_{\mathrm{cycl},\sigma_{d}}(\rho)\\
			\cong& W_E\left(\mathbb{F}(\negthinspace(T_{\mathrm{cycl}}^{1/p^{\infty}})\negthinspace)\right) \otimes_{A_{\mathrm{cycl},E}}D_{\mathrm{cycl},\sigma_{d}}(\rho)
		\end{align*}
		where the third isomorphism uses (\ref{cycltomic phi,Z_p* module are isomorphic by base change by phi}) and the last isomorphism uses the fact that the endomorphism of $D_{\mathrm{cycl},\sigma_{d}}(\rho)$ given by $a\in\mathbb{Z}_p^{\times}$ is an automorphism. Therefore, (\ref{isomorphism 3}) induces the required isomorphism (\ref{required isomorphism}) which finishes the proof.
	\end{proof}

	\begin{thm}\label{f.p. D_{A_{mv,E}}^{(i)} is fully faithful}
		The functor $D_{A_{\mathrm{mv},E}}^{(i)}: \rep_{\mathcal{O}_E}\left(\gal(\overline{K}/K)\right)\to\modetfp_{\varphi_q,\mathcal{O}_K^{\times}}(A_{\mathrm{mv},E})$ is fully faithful and exact. If $\rho$ is a finite free $\mathcal{O}_E$-representation of $\gal(\overline{K}/K)$, then $D_{A_{\mathrm{mv},E}}^{(i)}(\rho)$ is a finite free étale $(\varphi_q,\mathcal{O}_K^{\times})$-module over $A_{\mathrm{mv},E}$, and
		\[\mathrm{rank}_{\mathcal{O}_E}\rho=\mathrm{rank}_{A_{\mathrm{mv},E}}D_{A_{\mathrm{mv},E}}^{(i)}(\rho).\]
	\end{thm}

	\begin{proof}
		The claim that $D_{A_{\mathrm{mv},E}}^{(i)}$ sends finite free $\mathcal{O}_E$-representations to finite free étale $(\varphi_q,\mathcal{O}_K^{\times})$-modules over $A_{\mathrm{mv},E}$ and preserves the rank follows directly from the construction. The exactness and faithfulness has been explained in the discussion after Lemma \ref{ralation of different f.p. D_{A_E}^i}, hence it remains to show the fullness of $D_{A_{\mathrm{mv},E}}^{(i)}$. 
		By Lemma \ref{ralation of different f.p. D_{A_E}^i}, it suffices to show the fullness of $D_{A_{\mathrm{mv},E}}^{(0)}$. Let $\rho_1,\rho_2$ be finite type $\mathcal{O}_E$-representations of $\gal(\overline{K}/K)$, there is a commutative diagram of abelian groups (Proposition \ref{relation to classical cycltomic (phi,Gamma)-modules}):
		\begin{equation*}
			\begin{tikzcd}[column sep=huge,row sep=huge]
				\mathrm{Hom}_{\gal(\overline{K}/K)}(\rho_1,\rho_2) \arrow[rd,"\sim","D_{\mathrm{cycl},\sigma_0}"'] \arrow[r,"D_{A_{\mathrm{mv},E}}^{(0)}"] & \mathrm{Hom}_{\modet_{\varphi_q,\mathcal{O}_K^{\times}}(A_{\mathrm{mv},E})}(D_{A_{\mathrm{mv},E}}^{(0)}(\rho_1),D_{A_{\mathrm{mv},E}}^{(0)}(\rho_2))\arrow[d,"A_{\mathrm{cycl},E}\otimes_{\mathrm{tr},A_{\mathrm{mv},E}}-"]\\
				 &\mathrm{Hom}_{\modet_{\varphi_q,\mathbb{Z}_p^{\times}}(A_{\mathrm{cycl},E})}(D_{\mathrm{cycl},\sigma_{d}}(\rho_1),D_{\mathrm{cycl},\sigma_{d}}(\rho_2)).
			\end{tikzcd}
		\end{equation*}
		Thus if we can show the injectivity of the right vertical map, then the fully faithfulness of $D_{A_{\mathrm{mv},E}}^{(0)}$ follows. Let $h:D_{A_{\mathrm{mv},E}}^{(0)}(\rho_1)\to D_{A_{\mathrm{mv},E}}^{(0)}(\rho_2)$ be a homomorphism of finitely presented étale $(\varphi_q,\mathcal{O}_K^{\times})$-modules such that $\mathrm{id}\otimes h:D_{\mathrm{cycl},\sigma_{d}}(\rho_1)\to D_{\mathrm{cycl},\sigma_{d}}(\rho_2)$ is the zero map.
		Let $\mathfrak{p}\subset A_{\mathrm{mv},E}$ be the kernel of $\mathrm{tr}: A_{\mathrm{mv},E}\twoheadrightarrow A_{\mathrm{cycl},E}$. Since $A_{\mathrm{cycl},E}$ is an integral domain, $\mathfrak{p}$ is a prime ideal. As $\mathrm{id}\otimes h=0$, we have
		\[h(D_{A_{\mathrm{mv},E}}^{(0)}(\rho_1))\subseteq \mathfrak{p}\cdot D_{A_{\mathrm{mv},E}}^{(0)}(\rho_2).\]
		Since $h$ commutes with $\varphi_q$ and $D_{A_{\mathrm{mv},E}}^{(0)}(\rho_1)$ is étale, we have
		\[h(D_{A_{\mathrm{mv},E}}^{(0)}(\rho_1))\subseteq \varphi_q^n(\mathfrak{p})\cdot D_{A_{\mathrm{mv},E}}^{(0)}(\rho_2),\quad n\geq 1.\]
		For any $n\geq 1$, the map $\varphi_q^{nf'}: A_{\mathrm{mv},E}/\varpi\to A_{\mathrm{mv},E}/\varpi$ is the $q^{nf'}$-th power map where $f'=[\mathbb{F}:\mathbb{F}_q]$, hence $\varphi_q^{nf'}(\mathfrak{p})\subseteq \mathfrak{p}^{q^{nf'}}+\varpi A_{\mathrm{mv},E}$. Thus
		\[h(D_{A_{\mathrm{mv},E}}^{(0)}(\rho_1))\subseteq \left(\mathfrak{p}^{q^{nf'}}+\varpi A_{\mathrm{mv},E}\right)\cdot D_{A_{\mathrm{mv},E}}^{(0)}(\rho_2),\quad n\geq 1.\]
		Let $M\coloneq \bigcap_{n=1}^{\infty}\left(\mathfrak{p}^{q^{nf'}}+\varpi A_{\mathrm{mv},E}\right)\cdot D_{A_{\mathrm{mv},E}}^{(0)}(\rho_2)$, we have $h(D_{A_{\mathrm{mv},E}}^{(0)}(\rho_1))\subseteq M$ and $\varpi D_{A_{\mathrm{mv},E}}^{(0)}(\rho_2)\subseteq M$. Moreover, since $\varpi D_{A_{\mathrm{mv},E}}^{(0)}(\rho_2)$ is contained in $\left(\mathfrak{p}^{q^{nf'}}+\varpi A_{\mathrm{mv},E}\right)\cdot D_{A_{\mathrm{mv},E}}^{(0)}(\rho_2)$ for every $n\geq 1$, we have
		\begin{align*}
			M\big/\varpi D_{A_{\mathrm{mv},E}}^{(0)}(\rho_2)&=\bigcap_{n=1}^{\infty}\left(\left.\left(\mathfrak{p}^{q^{nf'}}+\varpi A_{\mathrm{mv},E}\right)\cdot D_{A_{\mathrm{mv},E}}^{(0)}(\rho_2)\middle/\varpi D_{A_{\mathrm{mv},E}}^{(0)}(\rho_2)\right.\right)\\
			&=\bigcap_{n=1}^{\infty}\left(\overline{\mathfrak{p}}^{q^{nf'}}\cdot \left( D_{A_{\mathrm{mv},E}}^{(0)}(\rho_2)/\varpi D_{A_{\mathrm{mv},E}}^{(0)}(\rho_2)\right)\right),
		\end{align*}
		where $\overline{\mathfrak{p}}$ is the image of $\mathfrak{p}$ under the surjection $A_{\mathrm{mv},E}\twoheadrightarrow A$ which is a proper ideal of $A$.	
		Although $D_{A_{\mathrm{mv},E}}^{(0)}(\rho_2)$ may not be a free $A_{\mathrm{mv},E}$-module, by Proposition \ref{description of finitely presented etale phi_q-module over A_{mv,E}}, $D_{A_{\mathrm{mv},E}}^{(0)}(\rho_2)/\varpi D_{A_{\mathrm{mv},E}}^{(0)}(\rho_2)$ is always a finite free $A$-module. Therefore,
		\begin{align*}
			M/\varpi D_{A_{\mathrm{mv},E}}^{(0)}(\rho_2)&=\bigcap_{n=1}^{\infty}\left(\overline{\mathfrak{p}}^{q^{nf'}}\cdot \left( D_{A_{\mathrm{mv},E}}^{(0)}(\rho_2)/\varpi D_{A_{\mathrm{mv},E}}^{(0)}(\rho_2)\right)\right)\\
			&=\left(\bigcap_{n=1}^{\infty}\overline{\mathfrak{p}}^{q^{nf'}}\right)\cdot \left( D_{A_{\mathrm{mv},E}}^{(0)}(\rho_2)/\varpi D_{A_{\mathrm{mv},E}}^{(0)}(\rho_2)\right).
		\end{align*}
		As $\overline{\mathfrak{p}}\subsetneq A$ is a proper ideal and $A$ is a Noetherian domain, by Krull's intersection theorem, $\bigcap_{n=1}^{\infty}\overline{\mathfrak{p}}^{q^{nf'}}=0$, thus
		\[M=\varpi D_{A_{\mathrm{mv},E}}^{(0)}(\rho_2),\]
		which implies that
		\begin{align}\label{starting step of induction, general case}
			h(D_{A_{\mathrm{mv},E}}^{(0)}(\rho_1))\subseteq \varpi D_{A_{\mathrm{mv},E}}^{(0)}(\rho_2)=D_{A_{\mathrm{mv},E}}^{(0)}(\varpi \rho_2).
		\end{align}
		By induction on $m\geq 1$, one can show that
		\[h (D_{A_{\mathrm{mv},E}}^{(0)}(\rho_1))\subseteq D_{A_{\mathrm{mv},E}}^{(0)}(\varpi^m\rho_2)=\varpi^m D_{A_{\mathrm{mv},E}}^{(0)}(\rho_2),\quad m\geq1\]
		hence
		\[h (D_{A_{\mathrm{mv},E}}^{(0)}(\rho_1))\subseteq \bigcap_{m=1}^{\infty}\varpi^m\cdot D_{A_{\mathrm{mv},E}}^{(0)}(\rho_2).\]
		By \cite[\href{https://stacks.math.columbia.edu/tag/00IQ}{Tag 00IQ}]{stacks-project}, as $\varpi$ is contained in the Jacobson radical of $A_{\mathrm{mv},E}$ and $D_{A_{\mathrm{mv},E}}^{(0)}(\rho_2)$ is a finite $A_{\mathrm{mv},E}$-module, we have
		\[\bigcap_{m=1}^{\infty}\varpi^m\cdot D_{A_{\mathrm{mv},E}}^{(0)}(\rho_2)=0.\]
		Therefore, $h$ is the zero map, which finishes the proof.
	\end{proof}

	\begin{remark}
		If $K=\qp$, then $A_{\mathrm{mv},E}$ is nothing but $A_{\mathrm{cycl},E}=A_{\mathrm{LT},E,\sigma_{0}}$, and the functor $D_{A_{\mathrm{mv},E}}^{(i)}$ is $D_{\mathrm{cycl},\sigma_{0}}=D_{\mathrm{LT},\sigma_{0}}=D_{\mathrm{LT}}$, hence is an equivalence of categories. If $K\neq \qp$, as the following proposition shows, the functor $D_{A_{\mathrm{mv},E}}^{(i)}$ is not essentially surjective. 
	\end{remark}
	
	\begin{prop}\label{not essentially surjective for integrel f>1}
		For $0\leq i\leq f-1$, the functor $D_{A_{\mathrm{mv},E}}^{(i)}: \rep_{\mathcal{O}_E}\left(\gal(\overline{K}/K)\right)\to\modetfp_{\varphi_q,\mathcal{O}_K^{\times}}(A_{\mathrm{mv},E})$ is not essentially surjective if $f\geq 2$.
	\end{prop}
	
	\begin{proof}
		By Corollary \ref{descent of f.p. phi_q,O_K*-modules}, it is equivalent to show that the functor $D_{W_E(A_{\infty})}^{(i)}$ is not essentially surjective.
		
		For finitely presented étale $(\varphi_q,\mathcal{O}_K^{\times})$-modules $M_1,M_2$ over $W_E(A_{\infty})$, we set
		\[\mathrm{Ext}^1(M_1,M_2)\coloneq\mathrm{Ext}^1_{\modetfp_{\varphi_q,\mathcal{O}_K^{\times}}(W_E(A_{\infty}))}(M_1,M_2),\]
		which is an $\mathcal{O}_E$-module. 
		Similarly, for finite type continuous $\mathcal{O}_E$-representations $\rho_1,\rho_2$ of $\gal(\overline{K}/K)$, we set $\mathrm{Ext}^1(\rho_1,\rho_2)\coloneq\mathrm{Ext}^1_{\rep_{\mathcal{O}_E}(\gal(\overline{K}/K))}(\rho_1,\rho_2)$ which is also an $\mathcal{O}_E$-module. 
		Viewing $\mathcal{O}_E$ as a trivial representation of $\gal(\overline{K}/K)$, we have $D_{W_E(A_{\infty})}^{(i)}(\mathcal{O}_E)=W_E(A_{\infty})$ where $W_E(A_{\infty})$ is regarded as a trivial étale $(\varphi_q,\mathcal{O}_K^{\times})$-module. Theorem \ref{f.p. D_{A_{mv,E}}^{(i)} is fully faithful} and Corollary \ref{descent of f.p. phi_q,O_K*-modules} imply that the functor $D_{W_E(A_{\infty})}^{(i)}$ is fully faithful and exact, hence induces an injection of $\mathcal{O}_E$-modules
		\begin{align}\label{canonical injection}
			\mathrm{Ext}^1(\mathcal{O}_E,\mathcal{O}_E)\hookrightarrow \mathrm{Ext}^1\left(W_E(A_{\infty}),W_E(A_{\infty})\right).
		\end{align}
		We will show that (\ref{canonical injection}) is not surjective, hence $D_{W_E(A_{\infty})}^{(i)}$ cannot be essentially surjective.
		
		Suppose that $0\to W_E(A_{\infty})\xrightarrow{f} V\xrightarrow{g}W_E(A_{\infty})\to 0$ is an extension of $W_E(A_{\infty})$ by itself in ${\modetfp_{\varphi_q,\mathcal{O}_K^{\times}}(W_E(A_{\infty}))}$. Let $e_0\coloneq f(1)$, and take $e_1\in g^{-1}(1)$, we have $V=W_E(A_{\infty})e_0\oplus W_E(A_{\infty})e_1$, and there exists $x\in W_E(A_{\infty})$ such that $\varphi_q(e_1)=xe_0+e_1$. Moreover, for any $a\in\mathcal{O}_K^{\times}$, there exists $c_a\in W_E(A_{\infty})$ such that $a(e_1)=c_ae_0+e_1$.
		Since $\varphi_q$ commutes with the action of $\mathcal{O}_K^{\times}$, we get
		\begin{align}\label{equivalent description of extensions of W_E(A_{infty}) by itself}
			\begin{cases}
				a(x)-x=\varphi_q(c_a)-c_a, &a\in\mathcal{O}_K^{\times}\\
				c_{a_1a_2}=a_1(c_{a_2})+c_{a_2}, &a_1,a_2\in\mathcal{O}_K^{\times}.
			\end{cases}
		\end{align}
		Let $Z^1(W_E(A_{\infty}))$ be the $\mathcal{O}_E$-module $\left\{(x,c)\in W_E(A_{\infty})\oplus Z^1_{\mathrm{cont}}(\mathcal{O}_K^{\times},W_E(A_{\infty})): (x,c)\ \text{satisfies (\ref{equivalent description of extensions of W_E(A_{infty}) by itself})}\right\}$ where $Z^1_{\mathrm{cont}}(\mathcal{O}_K^{\times},W_E(A_{\infty}))$ is the abelian group of continuous $1$-cocycles. Let $d$ be the homomorphism of $\mathcal{O}_E$-modules $W_E(A_{\infty})\to Z^1(W_E(A_{\infty})),\ \alpha\mapsto \big(\varphi_q(\alpha)-\alpha,a\mapsto(a(\alpha)-\alpha)\big)$,
		and let $H^1(W_E(A_{\infty}))$ be the quotient $Z^1(W_E(A_{\infty}))/\mathrm{Im}d$. It is straightforward to verify that the map
		\begin{align}\label{description of Ext^1}
			h: \mathrm{Ext}^1(W_E(A_{\infty}),W_E(A_{\infty}))\to H^1(W_E(A_{\infty})),\ V\mapsto [(x,c)]
		\end{align}
		is well-defined and is an isomorphism of $\mathcal{O}_E$-modules. Hence we identify $\mathrm{Ext}^1(W_E(A_{\infty}),W_E(A_{\infty}))$ with $H^1(W_E(A_{\infty}))$. Similarly, there is an isomorphism of $\mathcal{O}_E$-modules
		\[\mathrm{Ext}^1(\mathcal{O}_E,\mathcal{O}_E)\cong H^1_{\mathrm{cont}}(\gal(\overline{K}/K),\mathcal{O}_E)\cong \mathcal{O}_E^{f+1}.\]
		where the last isomorphism uses the fact that $\gal(\overline{K}/K)$ acts trivially on $\mathcal{O}_E$.
		
		Now we prove that $H^1(W_E(A_{\infty}))$ is not a finite type $\mathcal{O}_E$-module. First, we claim that $H^1(W_E(A_{\infty}))$ is torsion-free as an $\mathcal{O}_E$-module. For $[(x,c)]\in H^1(W_E(A_{\infty}))$ satisfies $\varpi\cdot[(x,c)]=0$, we fix a representative $(x,c)\in Z^1(W_E(A_{\infty}))$, then there exists $\alpha\in W_E(A_{\infty})$ such that $\varpi\cdot(x,c)=d(\alpha)$. In particular, $\varphi_q(\alpha)-\alpha=\varpi x\in \varpi\cdot W_E(A_{\infty})$. Write $\alpha=\sum_{i=0}^{+\infty}[\alpha_i]\varpi^i$, $\alpha_i\in A_{\infty}$, then $\varphi_q(\alpha_0)=\alpha_0$,
		which easily implies that $\alpha_0\in\mathbb{F}$. Write $\alpha-[\alpha_0]=\varpi \beta$ for some $\beta\in W_E(A_{\infty})$, we have $(x,c)=d(\beta)$, hence $[(x,c)]=0$, which proves the claim. 
		
		For $k\geq 1$, let $x_k\coloneq [T_{\sigma_{0}}^{-(q-1)k}]\cdot[T_{\sigma_{1}}^{(q-1)k}]\in W_E(A_{\infty})$ be the Teichmüller lift of $T_{\sigma_{0}}^{-(q-1)k}\cdot T_{\sigma_{1}}^{(q-1)k}\in A_{\infty}$. For $a\in\mathcal{O}_K^{\times}$, applying (\ref{action of phi_q on A}), we have 
		\begin{align*}
			\left|a\left(\dfrac{T_{\sigma_{1}}}{T_{\sigma_{0}}}\right)-\dfrac{\sigma_{1}(\overline{a})T_{\sigma_{1}}}{\sigma_{0}(\overline{a})T_{\sigma_{0}}}\right|&=\left|\dfrac{a(T_{\sigma_{1}})\Big(\sigma_{0}(\overline{a})T_{\sigma_{0}}-a(T_{\sigma_{0}})\Big)+a(T_{\sigma_{0}})\Big(a(T_{\sigma_{1}})-\sigma_{1}(\overline{a})T_{\sigma_{1}}\Big)}{a(T_{\sigma_{0}})\sigma_{0}(\overline{a})T_{\sigma_{0}}}\right|\\
			&\leq\dfrac{\max\left(\left|a(T_{\sigma_{1}})\Big(\sigma_{0}(\overline{a})T_{\sigma_{0}}-a(T_{\sigma_{0}})\Big)\right|,\left|a(T_{\sigma_{0}})\Big(a(T_{\sigma_{1}})-\sigma_{1}(\overline{a})T_{\sigma_{1}}\Big)\right|\right)}{\left|T_{\sigma_{0}}\right|^2}\\
			&\leq |T_{\sigma_{0}}|^{p-1},
		\end{align*}
		where $|\cdot|: A_{\infty}\to \mathbb{R}_{\geq 0}$ is the multiplicative norm. This implies that
		\begin{align}\label{a(T_1/T_0)^{q-1}-(T_1/T_0) is small}
			\left|a\left(\dfrac{T_{\sigma_{1}}}{T_{\sigma_{0}}}\right)^{(q-1)k}-\left(\dfrac{T_{\sigma_{1}}}{T_{\sigma_{0}}}\right)^{(q-1)k}\right|\leq |T_{\sigma_{0}}|^{p-1},\quad k\geq 1.
		\end{align}
		The construction of ramified Witt vectors implies that there exists $P_n(X,Y)\in \mathbb{F}[X,Y]$, $n\geq0$, such that 
		\begin{align*}
			a(x_k)-x_k=\sum_{n=0}^{\infty}\varpi^n\left[a\left(\dfrac{T_{\sigma_{1}}}{T_{\sigma_{0}}}\right)^{(q-1)k}-\left(\dfrac{T_{\sigma_{1}}}{T_{\sigma_{0}}}\right)^{(q-1)k}\right]^{{q'}^{-n}}\cdot\left[P_n\left(a\left(\dfrac{T_{\sigma_{1}}}{T_{\sigma_{0}}}\right)^{(q-1)k},\left(\dfrac{T_{\sigma_{1}}}{T_{\sigma_{0}}}\right)^{(q-1)k}\right)\right]^{{q'}^{-n}}.
		\end{align*}
		For $m\geq 0$, using (\ref{a(T_1/T_0)^{q-1}-(T_1/T_0) is small}), we have
		\begin{align*}
			N_m\left(a(x_k)-x_k\right)&=\sup_{0\leq n\leq m}\left|a\left(\dfrac{T_{\sigma_{1}}}{T_{\sigma_{0}}}\right)^{(q-1)k}-\left(\dfrac{T_{\sigma_{1}}}{T_{\sigma_{0}}}\right)^{(q-1)k}\right|^{{q'}^{-n}}\cdot\left|P_n\left(a\left(\dfrac{T_{\sigma_{1}}}{T_{\sigma_{0}}}\right)^{(q-1)k},\left(\dfrac{T_{\sigma_{1}}}{T_{\sigma_{0}}}\right)^{(q-1)k}\right)\right|^{{q'}^{-n}}\\
			&\leq |T_{\sigma_{0}}|^{{q'}^{-m} (p-1)},
		\end{align*}
		where $N_m$ is defined in Lemma \ref{equivalent description of weak topology}.
		Therefore, for $n\geq 1$, $a\in\mathcal{O}_K^{\times}$, 
		\begin{align*}
			N_m\left(\varphi_q^n(a(x_k)-x_k)\right)=N_m\left(a(x_k)-x_k\right)^{q^n}\leq |T_{\sigma_{0}}|^{q^n{q'}^{-m}(p-1)}.
		\end{align*}
		This implies that $\lim\limits_{n\to+\infty}\varphi_q^n(a(x_k)-x_k)=0$, hence $c_{k,a}\coloneq \sum_{n=0}^{+\infty}\varphi_q^n(a(x_k)-x_k)\in W_E(A_{\infty})$ exists, and the map $c_k: \mathcal{O}_K^{\times}\to W_E(A_{\infty})$, $a\mapsto c_{k,a}$ is continuous for $k\geq 1$. Moreover, it is easy to check that $(x_k,c_k)$ satisfies (\ref{equivalent description of extensions of W_E(A_{infty}) by itself}), hence represents an element $[(x_k,c_k)]\in H^1(W_E(A_{\infty}))$ for $k\geq 1$. 
		
		Let $S\coloneq\left\{k\in\mathbb{Z}_{\geq 1}: p\nmid k\right\}$ be the set of positive integers which are not divisible by $p$. 
		Let $S_1$ be a finite subset of $S$. Suppose that $\left\{[(x_k,c_k)], k\in S_1\right\}$ are $\mathcal{O}_E$-linearly dependent in $H^1(W_E(A_{\infty}))$. Since $H^1(W_E(A_{\infty}))$ is torsion-free, we may assume that there exists $a_k\in\mathcal{O}_E, k\in S_1$, $a_{k_0}\notin \varpi \mathcal{O}_E$ for some $k_0\in S_1$ and $\alpha\in W_E(A_{\infty})$ such that
		\[\sum_{i\in S_1}a_k\cdot(x_k,c_k)=d(\alpha)\in Z^1(W_E(A_{\infty})).\]
		In particular, by reducing modulo $\varpi$, we have
		\[\sum_{k\in S_1}\overline{a}_k\overline{x}_k=\sum_{k\in S_1}\overline{a}_k \left(\dfrac{T_{\sigma_{k}}}{T_{\sigma_{0}}}\right)^{(q-1)k}=\varphi_q(\overline{\alpha})-\overline{\alpha}\in A_{\infty}\]
		with $\overline{a}_k\in\mathbb{F}$ and $\overline{a}_{k_0}\neq 0$ for some $k_0\in S_1$. Write
		\[\overline{\alpha}=\sum_{\underline{n}\in \left(\mathbb{Z}[1/p]\right)^{f-1}}f_{\underline{n}}(T_{\sigma_{0}})\prod_{j=1}^{f-1}\left(\dfrac{T_{\sigma_{j}}}{T_{\sigma_{0}}}\right)^{n_j}\in A_{\infty}\]
		with $f_{\underline{n}}(T_{\sigma_{0}})\in \mathbb{F}(\negthinspace(T_{\sigma_{0}}^{1/p^{\infty}})\negthinspace)$ satisfying that for any $l\geq 1$, there is only finitely many $\underline{n}\in \left(\mathbb{Z}[1/p]\right)^{f-1}$ such that $f_{\underline{n}}(T_{\sigma_{0}})\notin T_{\sigma_{0}}^{l}\mathbb{F}[\negthinspace[T_{\sigma_{0}}^{1/p^{\infty}}]\negthinspace]$. In particular, let $n_0\coloneq((q-1)k_0,0,\dots,0)\in\mathbb{Z}^{f-1}$, we have
		\[\lim_{m\to+\infty}|f_{\frac{1}{q^m}\underline{n_0}}(T_{\sigma_{0}})|=\lim_{m\to+\infty}|f_{{q^m}\underline{n_0}}(T_{\sigma_{0}})|=0.\]
		
		Note that
		\[\varphi_q(\overline{\alpha})-\overline{\alpha}=\sum_{\underline{n}\in \left(\mathbb{Z}[1/p]\right)^{f-1}}\left(\varphi_q(f_{\frac{1}{q}\underline{n}}(T_{\sigma_{0}}))-f_{\underline{n}}(T_{\sigma_{0}})\right)\prod_{j=1}^{f-1}\left(\dfrac{T_{\sigma_{j}}}{T_{\sigma_{0}}}\right)^{n_j}.\]
		By our choice for $S$, we have
		\begin{align}\label{phi_q changes index}
			\begin{cases}
				\varphi_q(f_{\frac{1}{q}\underline{n}_0}(T_{\sigma_{0}}))-f_{\underline{n}_0}(T_{\sigma_{0}})=\overline{a}_{k_0}\neq 0, & \\
				\varphi_q(f_{\frac{1}{q^m}\underline{n}_0}(T_{\sigma_{0}}))-f_{\frac{1}{q^{m-1}}\underline{n}_0}(T_{\sigma_{0}})=0, & m\in\mathbb{Z}, m\neq 1.
			\end{cases}
		\end{align}
		If $-f_{\underline{n}_0}(T_{\sigma_{0}})=\overline{a}_{k_0}\in\mathbb{F}^{\times}$, then (\ref{phi_q changes index}) implies
		\[f_{q^m\underline{n}_0}(T_{\sigma_{0}})=\varphi_q(f_{q^{m-1}\underline{n}_0}(T_{\sigma_{0}}))=\cdots=\varphi_q^{m}(f_{\underline{n_0}}(T_{\sigma_{0}}))=\overline{a}_{k_0}\neq 0, \quad m\geq 1,\]
		hence $\lim\limits_{m\to+\infty}|f_{{q^m}\underline{n_0}}(T_{\sigma_{0}})|=1$ which is impossible. Hence $-f_{\underline{n}_0}(T_{\sigma_{0}})\neq \overline{a}_{k_0}$, i.e. $|f_{\underline{n}_0}(T_{\sigma_{0}})+ \overline{a}_{k_0}|\neq 0$, and for $m\geq 1$, using (\ref{phi_q changes index}) and the fact that $|\varphi_q(x)|=|x|^q$ for any $x\in A_{\infty}$, we have
		\begin{align*}
			|f_{\frac{1}{q^m}\underline{n}_0}(T_{\sigma_{0}})|=|\varphi_q(f_{\frac{1}{q^m}\underline{n}_0})(T_{\sigma_{0}})|^{q^{-1}}&=|f_{\frac{1}{q^{m-1}}\underline{n}_0}(T_{\sigma_{0}})|^{q^{-1}}=\cdots\\
			&=\left|f_{\frac{1}{q}\underline{n}_0}(T_{\sigma_{0}})\right|^{q^{-(m-1)}}=\left|\varphi_q(f_{\frac{1}{q}\underline{n}_0}(T_{\sigma_{0}}))\right|^{q^{-m}}\\
			&=|\overline{a}_{k_0}+f_{\underline{n}_0}(T_{\sigma_{0}})|^{q^{-m}},
		\end{align*}
		thus
		\[\lim_{m\to+\infty}|f_{\frac{1}{q^m}\underline{n_0}}(T_{\sigma_{0}})|=\lim_{m\to+\infty}|\overline{a}_{k_0}+f_{\underline{n}_0}(T_{\sigma_{0}})|^{q^{-m}}=1\]
		which is also impossible. Therefore, for any finite subset $S_1\subset S$, $\left\{[(x_k,c_k)]: k\in S_1\right\}$ are linearly independent, thus the $\mathcal{O}_E$-submodule of $H^1(W_E(A_{\infty}))$ generated by $\left\{[(x_k,c_k)]: k\in S\right\}$ is not finitely generated, hence $H^1(W_E(A_{\infty}))$ is not finitely generated, thus the inclusion (\ref{canonical injection}) is not surjective. Therefore, $D_{W_E(A_{\infty})}^{(i)}$ is not essentially surjective.
	\end{proof}

	\subsection{Finitely presented étale $(\varphi,\mathcal{O}_K^{\times})$-modules attached to finite type continuous $\mathcal{O}_E$-representations of $\gal(\overline{K}/K)$}
	As in \cite[$\S$2.7]{breuil2023multivariable}, one can also define the functor $D_{A_{\mathrm{mv},E}}^{\otimes}$ from the category of finite type $\mathcal{O}_E$-representations to the category of finitely presented étale $(\varphi,\mathcal{O}_K^{\times})$-modules over $A_{\mathrm{mv},E}$.
	
	Recall that the multiplication by $p$ on $\mathcal{O}_K\cong N_0$ induces a continuous endomorphism of $\mathcal{O}_K[\negthinspace[N_0]\negthinspace]$ which extends continuously and uniquely to an endomorphism $\varphi$ of $A_{\mathrm{mv},K}$. We put $\varphi\coloneq \mathrm{id}\otimes \varphi: A_{\mathrm{mv},E}= \mathcal{O}_E\otimes_{\sigma_{0},\mathcal{O}_K}A_{\mathrm{mv},K}\to A_{\mathrm{mv},E}$, then $\varphi^{f}=\varphi_q$, and $\varphi$ commutes with the action of $\mathcal{O}_K^{\times}$. 
	Using the same arguments as in the proof of Lemma \ref{phi_q is flat}, we can show that the ring map $\varphi: A_{\mathrm{mv},E}\to A_{\mathrm{mv},E}$ is also finitely presented and faithfully flat.
	
	For a finite type $\mathcal{O}_E$-representation $\rho$ of $\gal(\overline{K}/K)$, we put
	\begin{align}\label{definition of D_A^temsor}
		D_{A_{\mathrm{mv},E}}^{\otimes}(\rho)\coloneq \bigotimes_{i=0}^{f-1}D_{A_{\mathrm{mv},E}}^{(i)}(\rho).
	\end{align}
	Let $\mathcal{O}_K^{\times}$ acts on $D_{A_{\mathrm{mv},E}}^{\otimes}(\rho)$ diagonally and we define the endomorphism $\varphi$ of $D_{A_{\mathrm{mv},E}}^{\otimes}(\rho)$:
	\[\varphi(x_0\otimes\cdots\otimes x_{f-1})\coloneq h_{f-1}(x_{f-1})\otimes h_{0}(x_0)\otimes\cdots\otimes h_{f-2}(x_{f-2}),\]
	where $h_i$ is the composition for $0\leq i\leq f-1$
	\[h_i: D_{A_{\mathrm{mv},E}}^{(i)}(\rho)\xhookrightarrow{x\mapsto 1\otimes x}A_{\mathrm{mv},E}\otimes_{\varphi,A_{\mathrm{mv},E}}D_{A_{\mathrm{mv},E}}^{(i)}(\rho)\xrightarrow[\sim]{\phi_i} D_{A_{\mathrm{mv},E}}^{(i+1)}(\rho).\]
	For the definition of $\phi_i$, see Lemma \ref{ralation of different f.p. D_{A_E}^i}.
	
	\begin{prop}\label{etaleness of phi,Gamma}
		For any finite type $\mathcal{O}_E$-representation $\rho$ of $\gal(\overline{K}/K)$, $D_{A_{\mathrm{mv},E}}^{\otimes}(\rho)$ is a finitely presented étale $(\varphi,\mathcal{O}_K^{\times})$-module over $A_{\mathrm{mv},E}$.
	\end{prop}
	\begin{proof}
		Clearly, $D_{A_{\mathrm{mv},E}}^{\otimes}(\rho)$ is a finite $A_{\mathrm{mv},E}$-module with a semi-linear $(\varphi,\mathcal{O}_K^{\times})$-action. Note that the linearization map
		\[A_{\mathrm{mv},E}\otimes_{\varphi}D_{A_{\mathrm{mv},E}}^{\otimes}(\rho)\xrightarrow{\mathrm{id}\otimes \varphi}D_{A_{\mathrm{mv},E}}^{\otimes}(\rho)\]
		decomposes as
		\begin{align*}
			A_{\mathrm{mv},E}\otimes_{\varphi}D_{A_{\mathrm{mv},E}}^{\otimes}(\rho)&=A_{\mathrm{mv},E}\otimes_{\varphi}\bigotimes_{i=0}^{f-1}D_{A_{\mathrm{mv},E}}^{(i)}(\rho)\\
			&\cong \bigotimes_{i=0}^{f-1}\left(A_{\mathrm{mv},E}\otimes_{\varphi}D_{A_{\mathrm{mv},E}}^{(i)}(\rho)\right)\\
			&\cong \bigotimes_{i=0}^{f-1}D_{A_{\mathrm{mv},E}}^{(i+1)}(\rho)\\
			&\cong \bigotimes_{i=0}^{f-1}D_{A_{\mathrm{mv},E}}^{(i)}(\rho)=D_{A_{\mathrm{mv},E}}^{\otimes}(\rho),
		\end{align*}
		where the third isomorphism uses (\ref{tensoring with phi changes i}), thus $D_{A_{\mathrm{mv},E}}^{\otimes}(\rho)$ is étale.
	\end{proof}
	Therefore, we obtain a functor $D_{A_{\mathrm{mv},E}}^{\otimes}: \rep_{\mathcal{O}_E}(\gal(\overline{K}/K))\to\modetfp_{\varphi,\mathcal{O}_K^{\times}}(A_{\mathrm{mv},E})$. The functor $D_{A_{\mathrm{mv},E}}^{\otimes}$ is neither full nor additive if $f\geq 2$, but might be related to $p$-adic representations of $\mathrm{GL}_2(K)$ (see \cite[Corollary 3.1.4]{breuil2023conjectures} and \cite[Theorem 1.1]{wang2024lubintatemultivariablevarphimathcaloktimesmodulesdimension} for mod $p$ representations).
	
	\subsection{Finite free étale $(\varphi_q,\mathcal{O}_K^{\times})$-modules attached to finite dimensional continuous $E$-representations of $\gal(\overline{K}/K)$}
	
	In this section, for $0\leq i\leq f-1$, we associate to every finite dimensional $E$-representation of $\gal(\overline{K}/K)$ a finite free étale $(\varphi_q,\mathcal{O}_K^{\times})$-module over $A_{\mathrm{mv},E}\left[\dfrac{1}{p}\right]$, and show that this construction gives a fully faithful exact functor.
	
	Let $B_{\mathrm{mv},E}\coloneq A_{\mathrm{mv},E}\left[\dfrac{1}{p}\right]$. We endow $B_{\mathrm{mv},E}=\bigcup_{n=1}p^{-n}A_{\mathrm{mv},E}$ with the colimit topology, then the continuous $(\varphi_q,\mathcal{O}_K^{\times})$-action on $A_{\mathrm{mv},E}$ extends to a continuous $(\varphi_q,\mathcal{O}_K^{\times})$-action on $B_{\mathrm{mv},E}$.
	For a finite $B_{\mathrm{mv},E}$-module $M$, we can choose a surjection $B_{\mathrm{mv},E}^{n}\twoheadrightarrow M$ of $B_{\mathrm{mv},E}$-modules, and then we endow $M$ with the quotient topology, where $B_{\mathrm{mv},E}^{n}$ is equipped with the product topology. It is easy to check that the topology on $M$ does not depend on the choice of the surjection $B_{\mathrm{mv},E}^{n}\twoheadrightarrow M$.

	\begin{defn}
		Let $M$ be a finite free module over $B_{\mathrm{mv},E}$ with a continuous semi-linear endomorphism $\varphi_q$ and a continuous semi-linear action of $\mathcal{O}_K^{\times}$, we say that $M$ is a \textit{finite free étale} $(\varphi_q,\mathcal{O}_K^{\times})$-\textit{module} over $B_{\mathrm{mv},E}$, if there is a $B_{\mathrm{mv},E}$-basis $e_1,\dots,e_r$ of $M$ such that $\oplus_{i=1}^rA_{\mathrm{mv},E}e_i$ is an étale $(\varphi_q,\mathcal{O}_K^{\times})$-module over $A_{\mathrm{mv},E}$. We denote by $\modet_{\varphi_q,\mathcal{O}_K^{\times}}(B_{\mathrm{mv},E})$ the category of finite free étale $(\varphi_q,\mathcal{O}_K^{\times})$-modules over $B_{\mathrm{mv},E}$.
	\end{defn}
	
	From Proposition \ref{description of finitely presented etale phi_q-module over A_{mv,E}}, Lemma \ref{f.p. etale (phi_q,O_K*)-modules over A is an abelian category}, and the definition, we deduce:
	
	\begin{prop}
		The category $\modet_{\varphi_q,\mathcal{O}_K^{\times}}(B_{\mathrm{mv},E})$ is an abelian category.
	\end{prop}

	Let $\rep_{E}(\gal(\overline{K}/K))$ be the category of finite dimensional continuous $E$-representations of $\gal(\overline{K}/K)$. For any object $\rho$ in $ \rep_{E}(\gal(\overline{K}/K))$, recall that there exists an open finite $\mathcal{O}_E$-submodule $\rho_0$ of $\rho$ which is stable under the action of $\gal(\overline{K}/K)$. Such $\rho_0$ is called an $\mathcal{O}_E$-\textit{lattice} of $\rho$.
	
	For any finite dimensional $E$-representation $\rho$ of $\gal(\overline{K}/K)$, $0\leq i\leq f-1$, we choose an $\mathcal{O}_E$-lattice $\rho_0$ of $\rho$ and put
	\[D_{B_{\mathrm{mv},E}}^{(i)}(\rho)\coloneq D_{A_{\mathrm{mv},E}}^{(i)}(\rho_0)\left[\frac{1}{p}\right].\]
	It is easy to check that $D_{B_{\mathrm{mv},E}}^{(i)}(\rho)$ does not depend on the choice of the $\mathcal{O}_E$-lattice $\rho_0$. Furthermore, for any homomorphism $f:\rho\to\rho'$ of finite dimensional continuous $E$-representations of $\gal(\overline{K}/K)$, we can take $\mathcal{O}_E$-lattices $\rho_0$ and $\rho_0'$ of $\rho$ and $\rho'$ respectively such that $f(\rho_0)\subseteq \rho_0'$, hence there is a homomorphism $D_{A_{\mathrm{mv},E}}^{(i)}(f|_{\rho_0}): D_{A_{\mathrm{mv},E}}^{(i)}(\rho_0)\to D_{A_{\mathrm{mv},E}}^{(i)}(\rho_0')$ of finite free étale $(\varphi_q,\mathcal{O}_K^{\times})$-modules over $A_{\mathrm{mv},E}$, and inverting $p$ induces a homomorphism of finite free étale $(\varphi_q,\mathcal{O}_K^{\times})$-modules over $B_{\mathrm{mv},E}$:
	\[D_{B_{\mathrm{mv},E}}^{(i)}(f): D_{B_{\mathrm{mv},E}}^{(i)}(\rho)\to D_{B_{\mathrm{mv},E}}^{(i)}(\rho').\]
	We can easily check that $D_{B_{\mathrm{mv},E}}^{(i)}(f)$ is also well-defined. Thus we obtain a functor
	\begin{equation}
		\begin{aligned}
			D_{B_{\mathrm{mv},E}}^{(i)}: \rep_{E}(\gal(\overline{K}/K)) &\to \modet_{\varphi_q,\mathcal{O}_K^{\times}}(B_{\mathrm{mv},E}),\\
			\rho&\mapsto D_{B_{\mathrm{mv},E}}^{(i)}(\rho).
		\end{aligned}
	\end{equation}
	
	\begin{thm}
		Let $0\leq i\leq f-1$. The functor $D_{B_{\mathrm{mv},E}}^{(i)}$ is exact, fully faithful and preserves the rank, i.e. for $\rho\in \rep_{E}(\gal(\overline{K}/K)) $, we have
		\[\dim_{E}\rho =\mathrm{rank}_{B_{\mathrm{mv},E}}D_{B_{\mathrm{mv},E}}^{(i)}(\rho).\]
	\end{thm}
	\begin{proof}
		Let $0\to\rho'\xrightarrow{f} \rho\xrightarrow{g} \rho''\to0$ be a short exact sequence in $\rep_{E}(\gal(\overline{K}/K))$, and $\rho_0$ be an $\mathcal{O}_E$-lattice of $\rho$, then $f^{-1}(\rho_0)$ and $g(\rho_0)$ are $\mathcal{O}_E$-lattices of $\rho'$ and $\rho''$ respectively. By Theorem \ref{f.p. D_{A_{mv,E}}^{(i)} is fully faithful}, the sequence
		\[0\to D_{A_{\mathrm{mv},E}}^{(i)}(f^{-1}(\rho_0))\to D_{A_{\mathrm{mv},E}}^{(i)}(\rho_0)\to D_{A_{\mathrm{mv},E}}^{(i)}(g(\rho_0))\to 0\]
		is exact. Inverting $p$, we deduce that $D_{B_{\mathrm{mv},E}}^{(i)}$ is an exact functor.
		
		Let $\rho$, $\rho'$ be finite dimensional continuous $E$-representations of $\gal(\overline{K}/K)$, and let $\rho_0$, $\rho_0'$ be  $\mathcal{O}_E$-lattices of $\rho$, $\rho'$ respectively. If $f:\rho\to\rho'$ is a homomorphism of representations such that $D_{B_{\mathrm{mv},E}}^{(i)}(f)=0$, we may assume that $f(\rho_0)\subseteq\rho_0'$, then $D_{A_{\mathrm{mv},E}}^{(i)}(f|_{\rho_0})=0$, and by Theorem \ref{f.p. D_{A_{mv,E}}^{(i)} is fully faithful}, $f|_{\rho_0}:\rho_0\to\rho_0'$ is the zero map, hence $f=0$. Conversely, let $g:D_{B_{\mathrm{mv},E}}^{(i)}(\rho)\to D_{B_{\mathrm{mv},E}}^{(i)}(\rho')$ be a homomorphism of finite free étale $(\varphi_q,\mathcal{O}_K^{\times})$-modules, replacing $\rho_0$ by $\varpi^n\rho_0$ for some $n\geq 1$, we may assume that $g\left(D_{A_{\mathrm{mv},E}}^{(i)}(\rho_0)\right)\subseteq D_{A_{\mathrm{mv},E}}^{(i)}(\rho_0')$, hence $g|_{D_{A_{\mathrm{mv},E}}^{(i)}(\rho_0)}=D_{A_{\mathrm{mv},E}}^{(i)}(h_0)$ for the unique $h_0: \rho_0\to \rho_0'$ by Theorem \ref{f.p. D_{A_{mv,E}}^{(i)} is fully faithful}. Then let $h:\rho\to\rho'$ be the extension of $h_0$, we have $D_{B_{\mathrm{mv},E}}^{(i)}(h)=g$. Therefore, $D_{B_{\mathrm{mv},E}}^{(i)}$ is fully faithful. The last assertion comes from the following equality:
		\[\dim_{E}\rho =\mathrm{rank}_{\mathcal{O}_E}\rho_0=\mathrm{rank}_{A_{\mathrm{mv},E}}(D_{A_{\mathrm{mv},E}}^{(i)}(\rho_0))=\mathrm{rank}_{_{\mathrm{mv},E}}(D_{B_{\mathrm{mv},E}}^{(i)}(\rho)).\]
	\end{proof}
	
	\begin{prop}
		Let $f\geq 2$. The functor $D_{B_{\mathrm{mv},E}}^{(i)}$ is not essentially surjective.
	\end{prop}
	\begin{proof}
		Let $\gal(\overline{K}/K)$ acts trivially on $E$, then $D_{B_{\mathrm{mv},E}}^{(i)}(E)=B_{\mathrm{mv},E}$.
		Using an analogue of (\ref{description of Ext^1}), we have
		\[\mathrm{Ext}_{\modetfp_{\varphi_q,\mathcal{O}_K^{\times}}(A_{\mathrm{mv},E})}^1(A_{\mathrm{mv},E},A_{\mathrm{mv},E})\left[\frac{1}{p}\right]\cong \bigcup_{n=0}^{+\infty}\mathrm{Ext}_{\modetfp_{\varphi_q,\mathcal{O}_K^{\times}}(A_{\mathrm{mv},E})}^1(A_{\mathrm{mv},E},p^{-n}A_{\mathrm{mv},E}).\]
		Let $0\to p^{-n}A_{\mathrm{mv},E}\xrightarrow{f_0} M_0\xrightarrow{g_0}A_{\mathrm{mv},E}\to 0$ be an extension in $\modetfp_{\varphi_q,\mathcal{O}_K^{\times}}(A_{\mathrm{mv},E})$. If there exists a section $s: B_{\mathrm{mv},E}\to M_0[1/p]$ of $g_0: M_0[1/p]\to B_{\mathrm{mv},E}$, then $s$ restricts to a section of $g_0: M_0\to A_{\mathrm{mv},E}$, hence the extension $0\to p^{-n}A_{\mathrm{mv},E}\xrightarrow{f_0} M_0\xrightarrow{g_0}A_{\mathrm{mv},E}\to 0$ splits. Thus the natural map $\mathrm{Ext}_{\modetfp_{\varphi_q,\mathcal{O}_K^{\times}}(A_{\mathrm{mv},E})}^1(A_{\mathrm{mv},E},p^{-n}A_{\mathrm{mv},E})\to \mathrm{Ext}^1_{\modet_{\varphi_q,\mathcal{O}_K^{\times}}(B_{\mathrm{mv},E})}(B_{\mathrm{mv},E},B_{\mathrm{mv},E})$ is injective for every $n\geq 0$. Therefore, the natural map
		\[\mathrm{Ext}_{\modetfp_{\varphi_q,\mathcal{O}_K^{\times}}(A_{\mathrm{mv},E})}^1(A_{\mathrm{mv},E},A_{\mathrm{mv},E})\left[\frac{1}{p}\right]\longrightarrow\mathrm{Ext}^1_{\modet_{\varphi_q,\mathcal{O}_K^{\times}}(B_{\mathrm{mv},E})}(B_{\mathrm{mv},E},B_{\mathrm{mv},E})\]
		is injective. By the proof of Proposition \ref{not essentially surjective for integrel f>1}, $\mathrm{Ext}^1_{\modet_{\varphi_q,\mathcal{O}_K^{\times}}(B_{\mathrm{mv},E})}(B_{\mathrm{mv},E},B_{\mathrm{mv},E})$ is an infinite dimensional vector space over $E$. On the other hand, there is a natural isomorphism
		\[ \mathrm{Ext}^1_{\rep_{E}(\gal(\overline{K},K))}(E,E)\cong H^1_\mathrm{cont}(\gal(\overline{K}/K),E)\cong E^{f+1},\]
		thus the natural injection 
		\[\mathrm{Ext}^1_{\rep_{E}(\gal(\overline{K},K))}(E,E)\hookrightarrow \mathrm{Ext}^1_{\modet_{\varphi_q,\mathcal{O}_K^{\times}}(B_{\mathrm{mv},E})}(B_{\mathrm{mv},E},B_{\mathrm{mv},E})\]
		induced by the exact fully faithful functor $D_{B_{\mathrm{mv},E}}^{(i)}$ is not surjective. Therefore, $D_{B_{\mathrm{mv},E}}^{(i)}$ is not essentially surjective.
	\end{proof}
	
	For any finite dimensional $E$-representation $\rho$ of $\gal(\overline{K}/K)$, similar to (\ref{definition of D_A^temsor}), we define $D_{B_{\mathrm{mv},E}}^{\otimes}(\rho)\coloneq \bigotimes_{i=0}^{f-1}D_{B_{\mathrm{mv},E}}^{i}(\rho)$. By an argument similar to the proof of Proposition \ref{etaleness of phi,Gamma}, $D_{B_{\mathrm{mv},E}}^{\otimes}(\rho)$ is a finite type étale $(\varphi,\mathcal{O}_K^{\times})$-module over $B_{\mathrm{mv},E}$. When $\rho$ is the restriction of an automorphic Galois representation, we expect $D_{B_{\mathrm{mv},E}}^{\otimes}(\rho)$ to be related to $p$-adic unitary Banach representations of $\mathrm{GL}_2(K)$ over $E$ appearing in the completed cohomology of Shimura curves (see \cite[Corollary 3.1.4]{breuil2023conjectures} and \cite[Theorem 1.1]{wang2024lubintatemultivariablevarphimathcaloktimesmodulesdimension} for mod $p$ representations).
	
	\printbibliography
\end{document}